\documentclass{amsart}
\usepackage{graphicx}
\usepackage[margin=1in]{geometry}
\usepackage{color}
\usepackage{amsmath,amsfonts,amssymb}
\usepackage{hyperref}
\hypersetup{urlcolor=blue, citecolor=red, }

\newtheorem{theorem}{Theorem}[section]
\newtheorem{lemma}{Lemma}[section]

\theoremstyle{definition}
\newtheorem{definition}{Definition}[section]
\theoremstyle{corollary}
\newtheorem{corollary}{Corollary}[section]

\theoremstyle{remark}
\newtheorem{remark}{Remark}[section]

\theoremstyle{proposition}
\newtheorem{proposition}{Proposition}[section]

\numberwithin{equation}{section}

\let\div\undefined
\DeclareMathOperator{\div}{div}

\DeclareMathOperator{\ess}{ess}

\DeclareMathOperator{\sign}{sign}

\catcode`\@=11

\@addtoreset{equation}{section}

\title[Parabolic equations with double phase flux]
      {Existence and regularity results for a class of parabolic problems with double phase flux of variable growth}


 \email{arora@math.muni.cz}
 \email{shmarev@uniovi.es}

\thanks{The first author acknowledges the support from Czech Science Foundation,
project GJ19-14413Y, the second author was supported by the
Research Grant MTM2017-87162-P}

\thanks{$^*$ Corresponding author: S.~Shmarev}

\begin{document}

\maketitle

\centerline{\scshape Rakesh Arora}

\medskip

{\footnotesize
 \centerline{Department of Mathematics and Statistics, Masaryk University}
   \centerline{Building 08
Kotl\'{a}\v{r}sk\'{a} 2}
   \centerline{Brno 611 37, Czech Republic}
}

\medskip

\centerline{\scshape Sergey Shmarev$^*$}

\medskip

{\footnotesize
\centerline{Department of Mathematics, University of Oviedo}
\centerline{c/Federico García Lorca 18}
\centerline{Oviedo 33007, Spain}
}

\medskip

\bigskip

\begin{abstract}
We study the homogeneous Dirichlet problem for the equation
\[
u_t-\operatorname{div}\left((a(z)\vert \nabla u\vert ^{p(z)-2}+b(z)\vert \nabla u\vert ^{q(z)-2})\nabla u\right)=f\quad \text{in $Q_T=\Omega\times (0,T)$},
\]
where $\Omega\subset \mathbb{R}^N$, $N\geq 2$, is a bounded domain with $\partial\Omega \in C^2$. The variable exponents $p$, $q$ and the nonnegative modulating coefficients $a$, $b$ are given Lipschitz-continuous functions of the argument $z=(x,t)\in Q_T$. It is assumed that $\frac{2N}{N+2}<p(z),\ q(z)$ and that the modulating coefficients and growth exponents satisfy the balance conditions
\[
\text{$a(z)+b(z)\geq \alpha>0$ in $\overline{Q}_T$},\; \alpha=const;\qquad
\text{$\vert p(z)-q(z)\vert <\frac{2}{N+2}$ in $\overline{Q}_T$}.
\]
We find conditions on the source $f$ and the initial data $u(\cdot,0)$ that guarantee the existence of a unique strong solution $u$ with $u_t\in L^2(Q_T)$ and $a\vert \nabla u\vert ^{p}+b\vert \nabla u\vert ^q\in L^\infty(0,T;L^1(\Omega))$. The solution possesses the property of global higher integrability of the gradient,
\[
\vert \nabla u\vert ^{\min\{p(z),q(z)\}+r}\in L^1(Q_T)\quad \text{with any $r\in \left(0,\frac{4}{N+2}\right)$},
\]
which is derived with the help of new interpolation inequalities in the variable Sobolev spaces. The second-order differentiability of the strong solution is proven:
\[
D_{x_i}\left(\left(a\vert \nabla u\vert ^{p-2}+b\vert \nabla u\vert ^{q-2}\right)^{\frac{1}{2}}D_{x_j}u\right)\in L^2(Q_T),\quad i,j=1,2,\ldots,N.
\]
\end{abstract}

\tableofcontents

\section{Introduction}
Let $\Omega \subset \mathbb{R}^N$ be a smooth bounded domain, $N \geq 2$ and $0<T< \infty$. We consider the following parabolic equation with the homogeneous Dirichlet boundary conditions:
\begin{equation}\label{eq:main}
         \begin{cases}
             & u_t - \div\left(\mathcal{F}(z,\nabla u)\nabla u\right)
             = f(z)
             \quad\mbox{in $Q_T$},
             \\
             & \mbox{$u=0$ on $\Gamma_T$}, \quad  \mbox{$u(x,0)=u_0(x)$ in $\Omega$},
          \end{cases}
\end{equation}
where $z=(x,t)$ denotes the point in the cylinder $Q_T=\Omega \times (0,T]$ and $\Gamma_T= \partial\Omega \times (0,T)$ is the lateral boundary of the cylinder. The function

\[
\mathcal{F}(z,\nabla u)\nabla u=\left(a(z)\vert \nabla u\vert ^{p(z)-2} + b(z) \vert \nabla u\vert ^{q(z)-2}\right)\nabla u
\]
represents the flux that depends on the given coefficients $a(z)$, $b(z)$ and the variable exponents $p(z)$, $q(z)$. It is assumed that the coefficients $a$, $b$ are nonnegative in $Q_T$ and satisfy the condition
\begin{equation}
\label{eq:null-a-b}
\text{$a(z)+b(z)\geq \alpha$ in $\overline{Q}_T$ with a positive constant $\alpha$}.
\end{equation}

Equations of the type \eqref{eq:main} fall into the class of \textbf{double phase} equations intensively studied in the last decades. This name, introduced in \cite{colombo2015bounded,colombo2015regularity}, reflects the fact that the coercivity and growth conditions on the flux vary from point to point in the problem domain and depend on the relation between $p(z)$ and $q(z)$ and the properties of the coefficients. As special cases, equation \eqref{eq:main} contains the equations

\begin{equation}
\label{eq:example-1}
u_t=\operatorname{div}\left(\vert \nabla u\vert ^{p(z)-2}\nabla u\right)+f,\qquad a\equiv 1,\;b\equiv 0, \end{equation}
and

\begin{equation}
\label{eq:example-2}
\begin{split}
& u_t=\operatorname{div}\left(\vert \nabla u\vert ^{p(z)-2}\nabla u+b(z)\vert \nabla u\vert ^{q(z)-2}\nabla u\right)+f,
\\
& \qquad a(z)\equiv 1,\quad b(z)\geq 0,\quad q(z)\geq p(z).
\end{split}
\end{equation}
Equations \eqref{eq:example-1}, \eqref{eq:example-2} are prototypes of PDEs with variable growth conditions with a gap between the conditions of monotonicity and coercivity. In case of \eqref{eq:example-1}

\[
-c+\vert \xi\vert ^{p^-}\leq \mathcal{F}(z,\xi)\xi\cdot \xi \leq \vert \xi\vert ^{p^+}+C,
\]
where $c$, $C$ are nonnegative constants, $p^-=\inf p(z)<\sup p(z)=p^+$ with the supremun and infimum taken over the problem domain. Much attention has been paid in the literature to the study of the special case  \eqref{eq:example-1}. This interest is motivated by applications to the mathematical modelling of various real-world phenomena, such as the flows of electrorheological or thermorheological fluids, the thermistor problem, or the problem of processing of digital images. With no intention to provide an exhaustive list of references concerning equation \eqref{eq:example-1}, we confine ourselves to \cite{AMS-2004, Ant-Zh-2005, A-S, B-D-2011, E, Duzaar-Mingione-Steffen-2011, ZZX, ant-shm-book-2015}
and references therein for a review of the current state of the subject.

Equation \eqref{eq:example-2} furnishes an example of the double phase operator where the growth of the flux is controlled by the first or the second term depending upon the values of the modulating coefficient $b(z)$. More precisely, on the set $\{z \in Q_T: b(z)=0\}$ the growth is controlled by the $p(z)$-power of the gradient, while on
the set $\{z \in Q_T: b(z) \neq 0\}$ it is given by the sum of the $p(z)$ and $q(z)$ powers of the gradient. This is one of the reasons why we
regard \eqref{eq:example-2} as an equation with the double phase operator. Such operators appeared for the first time in the late 1970s and 1980s in the works by Ball \cite{ball-1976} and Zhikov \cite{zhikov-1986} for interpretation of physical processes in the nonlinear elasticity theory.
In particular, Zhikov \cite{zhikov-1986,Zhikov-1995} investigated the models of strongly anisotropic materials and introduced the following energy functional
\[
u \mapsto \int_{\Omega} \left( \vert \nabla u\vert ^p + b(x) \vert \nabla u\vert ^q\right) ~dx, \quad 0 \leq b(x) \leq L, \quad  1<p<q,
\]
where the modulating coefficient $b(x)$ dictates the geometry of the composite made of two materials with ordered hardening exponent $p$ and $q$. It is easy to see that the corresponding differential form of the integral functional can be drafted as
\[ u \mapsto  - \div\left(\vert \nabla u\vert ^{p-2}\nabla u+b(z)\vert \nabla u\vert ^{q-2}\nabla u \right).\]

The study of such operators was continued in the seminal works of Marcellini \cite{marcellini-1991, marcellini-2020} and Mingione et al. \cite{AMS-2004, colombo2015bounded, colombo2015regularity}. Later on, problems involving the double phase operators attracted attention of many researchers. Despite significant progress in the study of the double phase problems, they remain the subject of active research.
On the one hand, the study of equations with the double phase operators is a challenging mathematical problem. On the other hand, such problems  started appearing in a variety of other physical models. We refer here to \cite{bahrouni-2019} for applications in transonic flows, \cite{benci-2000} for quantum physics and  \cite{cherfils-2005} for reaction-diffusion systems. There is an extensive literature devoted to the study of the questions of existence and qualitative properties of solutions to the stationary and evolution double phase problems for single PDEs and systems of equations. For the stationary problems we refer to \cite{Alves-Radulescu-2020, chlebicka-2018,colombo2015bounded,colombo2015regularity, esposito2004sharp,Gasinski-Winkert-2020,Hasto-Ok-2019,Liu-Dai-2018,marcellini-1991,
ok-2020,Radulescu-2019,Radulescu-Zhang-2018}.

The results on the existence and qualitative properties of solutions to evolution problems can be found in  \cite{Alves-2021,Gwiazda-2021,Chlebicka-2019-1, Elmani-Meskine-2005-1,GWWrZ-2015,marcellini-2020,Swirczewska-Gwiazda-2014}. For a review of the existence results we refer to papers \cite{Gwiazda-2021,chlebicka-2018,Chlebicka-2019-1}. The natural analytic framework for the study of double-phase problems is the theory of Musielak-Orlicz spaces, for the $p(z)$-Laplace equation the natural function space is the variable Sobolev space $\mathbb{W}_{p(\cdot)}(Q_T)$. These functions space are introduced and briefly described in Section \ref{sec:results}. The existence results of \cite{Gwiazda-2021,Chlebicka-2019-1} are obtained in the context of the  Musielak-Orlicz spaces. When applied to the model equation \eqref{eq:example-2}, these results guarantee the existence of distributional solutions if the exponents $p$, $q$ and the coefficient $b$ satisfy certain regularity and balance conditions (see \cite[Example 1.18 (E2)]{Gwiazda-2021}) sufficient for the density of smooth functions in the Musielak-Orlicz space the solution belongs to.

The variational approach to a wide class of parabolic equations and systems with nonstandard growth is discussed in \cite{BDM-2013,marcellini-2020}, see also references therein for the previous results. The existence results of \cite{marcellini-2020} apply to the Dirichlet problem for the equations
\[
u_t=\operatorname{div} D_\xi f(x,u,\nabla u)-D_u f(x,u,\nabla u)
\]
with a convex integrand $f:\Omega\times \mathbb{R}\times \mathbb{R}^N\mapsto [0,\infty]$ satisfying the coercivity condition $f(x,u,\xi)\geq \nu \vert \xi\vert ^p-g(x)(1+\vert u\vert )$ with some constants $\nu>0$, $p>1$ and a nonnegative function $g\in L^{\frac{p}{p-1}}(\Omega)$. The initial function is assumed to satisfy the conditions
$u_0\in L^{2}(\Omega)\cap W^{1,p}(\Omega)$, $f(x,u_0,\nabla u_0)\in L^1(\Omega)$
(cf. with the assumptions on $u_0$ in Theorem \ref{th:main-result-1}). This class of equations includes \eqref{eq:example-1} and \eqref{eq:example-2} with constant $p,q$ and $b\equiv b(x)\geq 0$. For the latter, the integrand $f$ has the form $f(x,\xi) =\frac{1}{p}\vert \xi\vert ^{p}+\frac{b(x)}{q}\vert \xi\vert ^q$.

\medskip

In the present work, we address the following two questions.

\medskip

- Sufficient conditions of existence of strong solutions to problem \eqref{eq:main}. By the strong solution we mean a solution whose time derivative belongs to $L^2(Q_T)$ and $\mathcal{F}(z,\nabla u)\vert \nabla u\vert ^2\in L^{\infty}(0,T;L^{1}(\Omega))$, see Definition \ref{def:weak}.

- The global higher integrability of the gradient and second-order spatial regularity of the strong solution.\\

Our approach to the question of existence is different from \cite{Gwiazda-2021,Chlebicka-2019-1,marcellini-2020}. We find a solution via Galerkin's method as the limit of a sequence of finite-dimensional approximations. We show thereafter that the terms of this approximation sequence possess uniform estimates on the higher-order derivatives, which allows one to extract a subsequence with better regularity properties. For the solutions of problem \eqref{eq:main} we prove that
\begin{equation}
\label{eq:second-order-intr}
\sqrt{\mathcal{F}(z,\nabla u)}\nabla u \in W^{1,2}(Q_T).
\end{equation}

For the evolution $p$-Laplace equation with constant $p$, it is shown in \cite{Duzaar-Mingione-Steffen-2011} that $u_t\in L^{\frac{p}{p-1}}_{loc}(Q_T)$ and $D_{x_i}\left((\mu^2+\vert \nabla u\vert ^2)^{\frac{p-2}{4}}D_{x_j}u\right)\in L^2_{loc}(Q_T)$. The stronger result on the differentiablity of the flux is obtained in \cite{CM-2020}. It is shown that $u_t\in L^2(Q_T)$ and $\vert \nabla u\vert ^{p-2}\nabla u\in W^{1,2}(Q_T)$, provided that $f\in L^2(Q_T)$ and $u_0\in L^2(Q_T)\cap W^{1,p}_0(\Omega)$. This regularity result is sharp - see \cite[Remark 2.3]{CM-2020}. It is worth noting here that the use of Galerkin's approximations prevents one from employing the techniques developed in these works.

A crucial element of the proofs is the property of global higher integrability of the gradient. This property allows us to show that both terms of the flux function $\mathcal{F}$ are well-defined everywhere in $Q_T$, to prove strong convergence of the gradients of the approximations, and their higher regularity with respect to the spatial variables. The property of higher integrability of the gradient is interesting in itself. This is an intrinsic property of solutions of nonlinear elliptic and parabolic equations. For the solutions of evolution equations and systems of $p$-Laplace structure, it was established in \cite{Lewis-Kinnunen-2000} and has been studied since then in numerous works under different conditions on the nonlinear structure of the equation. We refer to \cite{AMS-2004, AB-2019, ABO-2020, Ant-Zh-2005, BDM-2013-ARMA, B-D-2011, Gianetti-Passarelli-Scheven-2020, hasto-ok, Parviainen-2009, parviainen-2007, ZZX, ZP} for the parabolic equations with nonstandard power growth. For the solutions of equation \eqref{eq:example-1}, the local version of this property reads as follows: if $u$ is a weak solution of equation \eqref{eq:example-1} and $\vert \nabla u\vert ^{p(z)}\in L^{1}(Q_T)$, then for every strictly interior sub-cylinder $Q'=\Omega'\times (\epsilon,T)$, $\Omega'\Subset \Omega$, $\epsilon>0$, there is a constant $\delta>0$ such that $\vert \nabla u\vert ^{p(z)+\delta}\in L^1(Q')$. The constant $\delta$ depends on the distance between the parabolic boundaries of $Q_T$ and $Q'$. It was recently shown in \cite{AB-2019} for constant $p$ and in \cite{ABO-2020} for variable $p(x,t)$ that this property remains valid in the domains $\Omega\times (\epsilon,T)$ under mild conditions on the regularity of $\partial\Omega$. The global higher integrability of the gradient in the whole cylinder $Q_T$ is proven in \cite{A-S} for the solutions of equation \eqref{eq:example-1}, now we prove its analogue for the solutions of equation \eqref{eq:main} in the absence of the order relation between the variable exponents $p(z)$ and $q(z)$. We show that
\begin{equation}\label{eq:higher:integra}
  \vert \nabla u\vert ^{\min\{p(z),q(z)\}+r}\in L^1(Q_T)\quad \text{with any $r\in \left(0,\frac{4}{N+2}\right)$}.
\end{equation}
In application to the solutions of equation \eqref{eq:example-1}, or equation \eqref{eq:example-2} with the exponents subject to the additional order condition $p(z)\leq q(z)$, this property reads as follows:
\[
\text{$\displaystyle \vert \nabla u\vert ^{p(z)+r}\in L^1(Q_T)$ with any $0<r<\frac{4}{N+2}$}.
\]
This result refines the property of global higher integrability of the gradient proven in \cite{A-S} for the solutions of equation \eqref{eq:example-1} and recovers the best order of local integrability of the gradient proven for weak solutions of equations with constant $(p,q)$-growth conditions in \cite{BDM-2013} in the case $2\leq p\leq q$, and in \cite{Singer-2015} for the case $\frac{2N}{N+2}<p<2$ and $p\leq q$.\\

Let us summarize the results and novelties of the present work.

\begin{itemize}
\item It is proven that problem \eqref{eq:main} has a unique strong solution. This result does not require any assumption on the null sets of the coefficients $a(z)$, $b(z)$,  except for condition \eqref{eq:null-a-b}. The constructed solution is continuous with respect to the problem data.

    \item The solution possesses the property of global higher integrability of the gradient \eqref{eq:higher:integra}.

\item Under the conditions of the existence theorem the solution $u(z)$ possesses the second-order spatial regularity \eqref{eq:second-order-intr} and $u_t\in L^2(Q_T)$.

    \item The existence and regularity results do not require any order relation between the variable exponents $p(z)$ and $q(z)$. This assumption is substituted by the balance condition on the admissible gap between the values of $p(z)$ and $q(z)$:
         \begin{equation}
         \label{eq:gap-intro}
         \max\{p(z),q(z)\}<\min\{p(z),q(z)\}+\frac{2}{N+2}.
         \end{equation}
         Moreover, we do not distinguish between the cases of degenerate or singular equations. The exponents are only subject to condition \eqref{eq:gap-intro} and may vary within the intervals $[p^-,p^+], [q^-,q^+]\subset (\frac{2N}{N+2},\infty)$ independently of each other.
\end{itemize}

\section{Assumptions and results}
\label{sec:results}

To formulate the results we have to introduce the function spaces the data and solutions of problem \eqref{eq:main} belong to.

\subsection{The functions spaces}

\subsubsection{Variable Lebesgue spaces} We begin with a brief description of the Lebesgue and
Sobolev spaces with variable exponents. A detailed insight into
the theory of these spaces and a review of the bibliography can be
found in \cite{UF,DHHR-2011,KR}. Let $\Omega \subset \mathbb{R}^N$, $N \geq 2$, be a
bounded domain with Lipschitz continuous boundary $\partial
\Omega$. Define the set

\[
\mathcal{P}(\Omega):= \{\text{measurable functions on} \ \Omega \ \text{with values in}\ (1, \infty)\}.
\]
Given $r \in \mathcal{P}(\Omega)$, we introduce the modular

\begin{equation}
\label{eq:modular} A_{r(\cdot)}(f)= \int_{\Omega} \vert f(x)\vert ^{r(x)}
~dx
\end{equation}
and the set

\[
L^{r(\cdot)}(\Omega) = \{f: \Omega \to \mathbb{R}\ \vert \ \text{measurable on}\ \Omega,\ A_{r(\cdot)}(f) < \infty\}.
\]
The set $L^{r(\cdot)}(\Omega)$ equipped with the Luxemburg norm

\[
\|  f\|  _{r(\cdot), \Omega}= \inf \left\{\lambda>0 : A_{r(\cdot)}\left(\frac{f}{\lambda}\right) \leq 1\right\}
\]
becomes a Banach space. By convention, from now on we use the notation

$$
r^- := \ess\min_{x \in \Omega} r(x), \quad r^+ := \ess\max_{x \in \Omega} r(x).
$$
If $r \in \mathcal{P}(\Omega)$ and $1< r^- \leq r(x) \leq r^+ < \infty$ in $\Omega$, then the following properties hold.

\begin{enumerate}
\item $L^{r(\cdot)}(\Omega)$ is a reflexive and separable Banach space.
\item For every $f \in L^{r(\cdot)}(\Omega)$

\begin{equation}
\label{eq:mod-2}
\min\{\|  f\|  ^{r^-}_{r(\cdot), \Omega}, \|  f\|  ^{r^+}_{r(\cdot), \Omega}\} \leq A_{r(\cdot)}(f) \leq \max\{\|  f\|  ^{r^-}_{r(\cdot), \Omega}, \|  f\|  ^{r^+}_{r(\cdot), \Omega}\}.
\end{equation}
\item For every $f \in L^{r(\cdot)}(\Omega)$ and $g \in L^{r'(\cdot)}(\Omega)$, the generalized H\"older inequality holds:

\begin{equation}
\label{eq:Holder}
\int_{\Omega} \vert fg\vert  \leq \left(\frac{1}{r^-} + \frac{1}{(r')^-} \right) \|  f\|  _{r(\cdot), \Omega} \|  g\|  _{r'(\cdot), \Omega} \leq 2 \|  f\|  _{r(\cdot), \Omega} \|  g\|  _{r'(\cdot), \Omega},
\end{equation}
where $r'= \frac{r}{r-1}$ is the conjugate exponent of $r$.
\item If $p_1, p_2 \in \mathcal{P}(\Omega)$ and satisfy the inequality $p_1(x) \leq p_2(x)$ a.e. in $\Omega$, then $L^{p_2(\cdot)}(\Omega)$ is continuously embedded in $L^{p_1(\cdot)}(\Omega)$ and for all $u \in L^{p_2(\cdot)}(\Omega)$
 \begin{equation}
 \label{eq:emb-1}\|  u\|  _{p_1(\cdot), \Omega} \leq C\|  u\|  _{p_2(\cdot), \Omega},\qquad C=C(\vert \Omega\vert , p_1^\pm, p_2^\pm).
 \end{equation}
 \item For every sequence $\{f_k\} \subset L^{r(\cdot)}(\Omega)$ and $f \in L^{r(\cdot)}(\Omega)$
\begin{equation}
\label{eq:conv-modular}
\|  f_k-f\|  _{r(\cdot),\Omega} \to 0 \quad \text{iff $ A_{r(\cdot)}(f_k-f) \to 0$ as $k \to \infty$}.
\end{equation}
\end{enumerate}

\subsubsection{Variable Sobolev spaces}
The variable Sobolev space $W^{1,r(\cdot)}_0(\Omega)$ is the set of functions

\[
W^{1,r(\cdot)}_0(\Omega)= \{u: \Omega \to \mathbb{R}\ \vert \  u \in L^{r(\cdot)}(\Omega)\cap W^{1,1}_0(\Omega),\; \vert \nabla u\vert  \in L^{r(\cdot)}(\Omega) \}
\]
equipped with the norm

\[
\|  u\|  _{W^{1,r(\cdot)}_0(\Omega)}= \|  u\|  _{r(\cdot), \Omega} + \|  \nabla u\|  _{r(\cdot), \Omega}.
\]

\noindent If $r \in C^0(\overline{\Omega})$, the Poincar\'e inequality holds: for every $u\in W_0^{1,r(\cdot)}(\Omega)$

\begin{equation}
\label{eq:Poincare}
\|  u\|  _{r(\cdot),\Omega} \leq C \|  \nabla u\|  _{r(\cdot),\Omega}.
\end{equation}
Inequality \eqref{eq:Poincare} means that the equivalent norm of $W^{1,r(\cdot)}_0(\Omega)$ is given by

\begin{equation}
\label{eq:equiv-norm}
\|  u\|  _{W^{1,r(\cdot)}_0(\Omega)}=\|  \nabla u\|  _{r(\cdot),\Omega}.
\end{equation}
Let us denote by $C_{{\rm log}}(\overline{\Omega})$ the subset of $\mathcal{P}(\Omega)$ composed of the functions continuous on $\overline{\Omega}$ with the logarithmic modulus of continuity:

\begin{equation}
\label{eq:log-cont}
{p} \in C_{{\rm log}}(\overline{\Omega})\quad \Leftrightarrow\quad \vert p(x)-p(y)\vert \leq \omega(\vert x-y\vert )\quad \forall x,y\in \overline{\Omega}, \;\vert x-y\vert <\frac{1}{2},
\end{equation}
where {$\omega$} is a nonnegative function such that

\[
\limsup_{s\to 0^+}\omega(s)\ln \frac{1}{s}=C,\quad C=const.
\]
If {$r \in C_{{\rm log}}(\overline{\Omega})$}, then the set $C_{0}^\infty(\Omega)$ of smooth functions with finite support is dense in $W^{1,r(\cdot)}_0(\Omega)$. This property allows one to use the equivalent definition of the space $W^{1,r(\cdot)}_0(\Omega)$:

\[
W^{1,r(\cdot)}_0(\Omega)=\left\{\text{the closure of $C_{0}^{\infty}(\Omega)$ with respect to the norm $\|  \cdot\|  _{W^{1,r(\cdot)}_0(\Omega)}$}\right\}.
\]
Given a function $u\in W^{1,r(\cdot)}_0(\Omega)$ with {$r \in C_{{\rm log}}(\overline{\Omega})$}, the smooth approximations of $u$ in $W^{1,r(\cdot)}_0(\Omega)$ can be obtained by mollification.

\subsubsection{Musielak-Orlicz spaces}
Let $a_0,b_0:\overline{\Omega}\mapsto [0,\infty)$ be given functions, $a_0,b_0\in C^{0,1}(\overline{\Omega})$. Assume that the exponents $p(\cdot), q(\cdot)\in C^{0,1}(\overline{\Omega})$ take values in the intervals $[p^-,p^+]$, $[q^-,q^+]$.
Set

\[
s(x)=\max\{2,\min\{p(x),q(x)\}\},\quad r(x)=\max\{2,p(x)\},\quad \sigma(x)=\max\{2,q(x)\}
\]
and consider the function

\[
\mathcal{H}(x,t)=t^{s(x)}+a_0(x)t^{r(x)}+b_0(x)t^{\sigma(x)},\quad t\geq 0,\quad x\in \Omega.
\]
The set

\[
L^\mathcal{H}(\Omega)=\left\{u:\Omega\mapsto \mathbb{R}\ \vert  \ \text{$u$ is measurable},\,\rho_{\mathcal{H}}(u)=\int_{\Omega}\mathcal{H}(x,\vert u\vert )\,dx<\infty \right\}
\]
equipped with the Luxemburg norm

\[
\|  u\|  _{\mathcal{H}}=\inf \left\{\lambda>0:\,\rho_{\mathcal{H}}\left(\dfrac{u}{\lambda}\right)\leq 1\right\}
\]
becomes a Banach space. The space $L^\mathcal{H}(\Omega)$ is separable and reflexive \cite{Fan-2012}. By $W^{1,\mathcal{H}}(\Omega)$ we denote the Musielak-Sobolev space

\begin{equation}
\label{eq:M-O}
W^{1,\mathcal{H}}(\Omega)=\left\{u\in L^{\mathcal{H}}(\Omega):\,\vert \nabla u\vert \in L^{\mathcal{H}}(\Omega)\right\}
\end{equation}
with the norm

\[
\|  u\|  _{1,\mathcal{H}}=\|  u\|  _{\mathcal{H}}+\|  \nabla u\|  _{\mathcal{H}}.
\]
The space $W^{1,\mathcal{H}}_0(\Omega)$ is defined as the closure of $C_{0}^\infty(\Omega)$ with respect to the norm of $W^{1,\mathcal{H}}(\Omega)$.

For the elements of the space $W^{1,\mathcal{H}}_0(\Omega)$, the norm and modular convergence are equivalent:

\begin{equation}
\label{eq:norm-modular-M-O}
\begin{split}
& \|  u\|  _{W^{1,\mathcal{H}}_0(\Omega)}\to 0\quad \Leftrightarrow\quad \rho_{\mathcal{H}}(u)+\rho_{\mathcal{H}}(\vert \nabla u\vert )\to 0.
\end{split}
\end{equation}
The detailed presentation of the theory of Musielak-Orlicz spaces can be found in the monograph \cite{Hasto-Harjulehto-2019-book} and the references therein.

\subsubsection{Spaces of function depending on $x$ and $t$}
For the study of parabolic problem \eqref{eq:main} we use the spaces of functions depending on $z=(x,t)\in Q_T$. Assume that $\partial\Omega\in C^2$, $p(z)\in C^{0,1}(\overline{Q}_T)$ and define the spaces
\begin{equation}
\label{eq:spaces}
\begin{split}
 & \mathbb{V}_{p(\cdot,t)}(\Omega) = \{u: \Omega \to \mathbb{R}\ \vert \  u \in L^2(\Omega)
\cap W_0^{1,1}(\Omega),\,\vert \nabla u\vert ^{p(x,t)}\in L^{1}(\Omega)
\},\quad t\in (0,T),
\\
& \mathbb{W}_{p(\cdot)}(Q_T)= \{u : (0,T) \to \mathbb{V}_{p(\cdot,t)}(\Omega) \ \vert \ u \in L^2(Q_T), \vert \nabla u\vert ^{p(z)}\in L^1(Q_T)\}.
\end{split}
\end{equation}
The norm of $\mathbb{W}_{p(\cdot)}(Q_T)$ is defined by
\[
\|  u\|  _{\mathbb{W}_{p(\cdot)}(Q_T)}=\|  u\|  _{2,Q_T}+\|  \nabla u\|  _{p(\cdot),Q_T}.
\]
Since $p\in C^{0,1}(\overline{Q}_T)\subset C_{{\rm log}}(\overline{Q}_T)$, the space $\mathbb{W}_{p(\cdot)}(Q_T)$ is the closure of $C_{0}^{\infty}(Q_T)$ with respect to this norm.

\subsection{Main results}

Let $p,q: Q_T \mapsto \mathbb{R}$  be functions satisfying the following conditions:
\begin{equation}\label{assum1}
\begin{split}
&\frac{2N}{N+2} < p_- \leq p(z) \leq p_+,\quad
\frac{2N}{N+2} < q_- \leq q(z) \leq q_+ \quad  \text{in $\overline{Q}_T$}
\end{split}
\end{equation}
with positive constants $p^\pm$, $q^\pm$,

\begin{equation}
\label{eq:Lip-p-q}
\begin{split}
\text{$p(z),q(z)\in C^{0,1}(\overline{Q}_T)$ with the Lipschitz constant $L_{p,q}$}.
\end{split}
\end{equation}
Let us define the functions
\begin{equation}
\label{eq:s-prelim}
\underline{s}(z)=\min\{p(z),q(z)\},\quad \overline{s}(z)=\max\{p(z),q(z)\}
\end{equation}
and accept the notation
\begin{equation}
\label{eq:s}
\underline{s}^-=\min_{\overline{Q}_T}\underline{s}(z),\quad \overline{s}^+=\max_{\overline{Q}_T}\overline{s}(z), \quad \dfrac{2N}{N+2}<\underline{s}^-\leq \overline{s}^+\leq \max\{p^+,q^+\}.
\end{equation}
We will repeatedly use the following threshold numbers:
\begin{equation}
\notag
\label{eq:constants}
r^\sharp=\dfrac{4}{N+2}, \qquad r_\ast=\dfrac{2}{N+2}.
\end{equation}
The modulating coefficients {$a(\cdot)$ and $ b(\cdot)$} satisfy the following conditions:
\begin{equation}
\label{eq:a-b}
\begin{split}
& \text{$a(z)$, $b(z)$  are nonnegative in $\overline{Q}_T$},
\\
& \text{$a,b \in C^{0,1}(\overline{Q}_T)$ with the Lipschitz constant $L_{a,b}$},
\\
& \max_{\overline{Q}_T}a(z)=a^+,\quad \max_{\overline{Q}_T}b(z)=b^+,
\\
& \text{$a(z) + b(z) \geq \alpha >0$ in $\overline{Q}_T$ with a positive constant $\alpha$}.
\end{split}
\end{equation}

\begin{definition}
\label{def:weak}
A function {$u:Q_T\mapsto \mathbb{R}$} is called {\bf strong solution} of problem \eqref{eq:main} if

\begin{enumerate}

\item $u \in  \mathbb{W}_{\overline{s}(\cdot)}(Q_T)$,
    $u_t \in L^2(Q_T)$,
    $\vert \nabla u\vert  \in L^{\infty}(0,T; L^{r(\cdot)}(\Omega))$

    with $r(z)=\max\{2,\underline{s}(z)\}$,
\item for every $\phi \in \mathbb{W}_{\overline{s}(\cdot)}(Q_T)$
\begin{equation}
\label{eq:def}
\int_{Q_T} u_t \phi ~dz + \int_{Q_T} \mathcal{F}(z,\nabla u)\nabla u\cdot \nabla \phi ~dz=
\int_{Q_T} f \phi \,dz,
\end{equation}
\item for every $\psi \in C_0^1(\Omega)$
\[
\int_{\Omega} (u(x,t)-u_0(x)) \psi  ~dx \to 0\quad\text{as $t \to
0$}.
\]
\end{enumerate}
\end{definition}

The main results are given in the following theorems. The first result concerns the existence and uniqueness of strong solution of the problem \eqref{eq:main} in the sense of Definition \ref{def:weak}. We assume that the initial datum belongs to $W_0^{1,\mathcal{H}}(\Omega)$ (see \eqref{eq:M-O}):
\begin{equation}
\label{eq:ini}
\|  \nabla u_0\|  _{2,\Omega}^2+\int_\Omega \mathcal{F}((x,0),\nabla u_0)\vert \nabla u_0\vert ^2\,dx=K<\infty
\end{equation}
with the coefficients $a$, $b$ and exponents $p$, $q$ in the flux function $\mathcal{F}$ taken at the instant $t=0$.

\begin{theorem}[Existence and uniqueness]
\label{th:main-result-1}
Let $\Omega\subset \mathbb{R}^N$, $N\geq 2$, be a bounded domain with the boundary $\partial\Omega\in C^2$.
Assume that {$p(\cdot)$, $q(\cdot)$} satisfy conditions \eqref{assum1}, \eqref{eq:Lip-p-q},

\begin{equation}
\label{eq:gap-z}
\max_{z \in \overline{Q}_T} \vert p(z)-q(z)\vert  < r_\ast,
\end{equation}
and the coefficients $a(z), b(z)$ satisfy conditions \eqref{eq:a-b}. Then for every
$f_0 \in L^2(0,T; W_0^{1,2}(\Omega))$ and every $u_0 \in
W^{1,\mathcal{H}}_0(\Omega)$ satisfying \eqref{eq:ini} problem \eqref{eq:main} has a unique strong solution $u(z)$. The solution satisfies the estimate

\begin{equation}
\label{eq:strong-esti}
\|  u_t\|  _{2,Q_T}^{2} +
\operatorname{ess}\sup_{(0,T)}\int_{\Omega} \left(\vert \nabla
u\vert ^{2}+\mathcal{F}(z,\nabla u)\vert \nabla u\vert ^2
 \right)\,dx \leq C
\end{equation}
with a constant $C$ depending upon $N$, $\partial \Omega$, $T$, $p^\pm$, $q^\pm$, the
constants in conditions \eqref{eq:Lip-p-q}, \eqref{eq:a-b}, \
$\|  u_0\|  _{W_0^{1,\mathcal{H}}(\Omega)}$ and $\|  f_0\|  _{L^{2}(0,T;W^{1,2}_0(\Omega))}$.
\end{theorem}

\begin{theorem}[Continuity with respect to the data]
\label{th:energy}
Under the conditions of Theorem \ref{th:main-result-1} the solution $u(z)$ satisfies the energy equality
\begin{equation}
\label{eq:energy}
\dfrac{1}{2}\|  u(t)\|  _{2,\Omega}^2+\int_{Q_T} \mathcal{F}(z,\nabla u)\vert \nabla u\vert ^2
\,dz= \dfrac{1}{2}\|  u_0\|  _{2,\Omega}^2+\int_{Q_T}uf\,dz.
\end{equation}
The solution is continuous with respect to the data: if $u(z)$, $v(z)$ are two strong solutions of problem \eqref{eq:main} corresponding to the initial data $u_0$, $v_0$ and the free terms $f$, $g$, then

\begin{equation}
  \label{eq:stab}
  \begin{split}
  & \sup_{(0,T)}\|  u-v\|  ^2_{2,\Omega}(t) + \int_{Q_T}\mathcal{F}(z,\nabla (u-v))\ \vert \nabla (u-v)\vert ^2\,dz
\to 0,
\\
& \int_{Q_T}\vert \nabla (u-v)\vert ^{\underline{s}(z)}\,dz\to 0
\quad \text{as $\|  u_0-v_0\|  _{2,\Omega}+\|  f-g\|  _{2,Q_T}\to 0$}.
\end{split}
\end{equation}
\end{theorem}

\begin{theorem}[Global regularity]
\label{th:global-reg}
Let the conditions of Theorem \ref{th:main-result-1} be fulfilled. Then
\begin{enumerate}
    \item  for any $\varsigma \in \left.\left[0,\frac{4}{N+2}\right.\right)$
    \begin{equation}
\label{eq:strong-est}
\int_{Q_T}\vert \nabla u\vert ^{\underline{s}(z)+\frac{4}{N+2}-\varsigma}\,dz \leq C
\end{equation}
with
a constant $C$ which depends on $N, \partial \Omega$, $T, p^\pm, q^\pm$, $\varsigma$, the
constants in conditions \eqref{eq:Lip-p-q}, \eqref{eq:a-b},
$\|  f_0\|  _{L^{2}(0,T;W^{1,2}_0(\Omega))}$ and $K$;
\item
\[
\begin{split}
& D_{x_i}\left(\sqrt{\mathcal{F}(z,\nabla u)}
 D_{x_j}u\right)\in
L^2(Q_T),
\\
& D^2_{x_ix_j} u\in L^{\underline{s}(\cdot)}_{loc}(Q_T \cap \{z: \max\{p(z), q(z)\}< 2\}),
\end{split}
\]
and the corresponding norms are bounded by constants depending only on the data.
\end{enumerate}
\end{theorem}

Outline of the paper. In Section \ref{sec:prelim}, we collect several technical assertions which are repeatedly referred to in the rest of the text. These are elementary algebraic relations and some useful relations between the elements of space  $\mathbb{W}_{\underline{s}(\cdot)}(Q_T)$ and the functions with $\mathcal{F}(z,\nabla v)\vert \nabla v\vert ^2\in L^1(Q_T)$.

Section \ref{sec:interpolation} is of utter importance for the rest of the work. This is where we derive the interpolation inequalities which yield the property of global higher integrability of the gradient, the key tool for the further study. These inequalities are proven for the regularized fluxes $\mathcal{F}\left(z,(\epsilon^2+(\nabla v,\nabla v))^{\frac{1}{2}}\right)$, $\epsilon\in (0,1)$. The strategy of the proof consists in reduction to a situation where the classical Gagliardo-Nirenberg inequalities in Sobolev spaces with constant exponents become applicable. To this end, we make use of the uniform continuity of the exponents $p(z)$, $q(z)$ and choose a finite cover of the domain $\Omega$, or the cylinder $\Omega\times (0,T)$, by subdomains where the oscillations of $p$ and $q$ are suitably small. Proceeding in this way we overcome the difficulties typical for variable Sobolev spaces where, unlike the classical Sobolev spaces, the inequalities for the norms and modulars need not coincide - see, e.g., \eqref{eq:mod-2}. In the result, we show that for all sufficiently smooth functions $v$ and every $\beta>0$
\[
\alpha\int_{Q_T}\vert \nabla v\vert ^{\underline{s}(z)+\frac{4}{N+2}-\varsigma}\,dz\leq \beta\int_{Q_T}\mathcal{F}\left(z,(\epsilon^2+(\nabla
v,\nabla v))^{\frac{1}{2}}\right)\sum_{i,j=1}^{N}\vert D^2_{x_ix_j}v\vert ^{2}\,dz+C
\]
with $\underline{s}(z)=\min\{p(z),q(z)\}$, any $\varsigma\in \left(0,\frac{4}{N+2}\right)$, and a constant $C$ depending on the known constants, $\beta$,  and $\operatorname{ess}\sup_{(0,T)}\|  u(t)\|  _{2,\Omega}$.

In Section \ref{sec:reg-problem} we formulate the problem with the regularized flux and describe the scheme of construction of a solution as the limit of a sequence of finite-dimensional approximations. Special attention is paid to the choice of the basis. On the one hand, in our conditions on the data, the set of smooth functions with compact support $C_0^\infty(\Omega)$ is dense in the variable Sobolev spaces and the generalized Musielak-Orlicz space $W^{1,\mathcal{H}}_0(\Omega)$. On the other hand, it is convenient to take for the basis the set of eigenfunctions of the Dirichlet problem for the Laplace operator but this set approximates the elements of $C_0^\infty(\Omega)$ in a suitable function space only if the boundary $\partial \Omega$ possesses an extra regularity. For this reason, we proceed in two steps. In the first step, we prove the main assertions for the regularized problem in a smooth domain with $\partial\Omega\in C^{k}$, $k\geq 2+[\frac{N}{2}]$. The solution is constructed as the limit of Galerkin's approximations. In the second step, we extend all these assertions to the domain with minimal required regularity $\partial\Omega\in C^2$. This is done by approximation of $\Omega$ from the interior by a family of expanding smooth domains.

Solvability of the regularized problem and the regularity of the solution are studied in Sections \ref{sec:a-priori}, \ref{sec:reg-existence}. In Section \ref{sec:a-priori} we derive uniform a priori estimates on the sequence of Galerkin's approximations in a smooth domain, including the estimates on the nonlinear terms that involve the second-order derivatives. In Section \ref{sec:reg-existence} we use these estimates to extract a subsequence with suitable convergence properties, and to show that the limit is a solution of the regularized problem. The second-order higher regularity of the limit follows from the a priori estimates on the regularized fluxes and the pointwise convergence of the sequence of gradients of the approximations.

Theorems \ref{th:main-result-1}, \ref{th:energy}, \ref{th:global-reg} are proven in Section \ref{sec:proofs-main-results} by passing to the limit with respect to the regularization parameter.

In the special case $p(z)=q(z)$ and $a=b=const$ equation \eqref{eq:main} transforms into the evolution $p(z)$-Laplace equation \eqref{eq:example-1}. For this equation, the questions of global higher integrability of the gradient and the second-order spatial regularity were studied in \cite{A-S}. The assertions of Theorem \ref{th:main-result-1} and \ref{th:global-reg} improve the corresponding results in \cite{A-S} and, by the same token, complete the results of  \cite{arora_shmarev2020} for another special case \eqref{eq:example-2}.

\section{Auxiliary propositions}
\label{sec:prelim}
We collect here the technical assertions used throughout the rest of the work. With certain abuse of notation, here we denote by $p(x)$, $q(x)$, $\underline{s}(x)$, $\overline{s}(x)$, $a(x)$, $b(x)$ the functions of the variables $(x,t)$ with ``frozen" $t\in [0,T]$. By continuity of these functions in $\overline{Q}_T$, conditions \eqref{eq:gap-z}, \eqref{eq:s} and \eqref{eq:a-b} remain in force if we fix $t$ and consider $p(x,t)$, $q(x,t)$ and $a(x,t)$, $b(x,t)$ as functions of the variable $x\in \Omega$. So, we write
\[
\begin{split}
& \overline{s}(x):= \max\{p(x), q(x)\},\qquad \overline{s}^+:= \max_{x \in \overline{\Omega}} \overline{s}(x),
\\
&
\underline{s}(x):= \min\{p(x), q(x)\},
\qquad \underline{s}^-:= \min_{x \in \overline{\Omega}} \underline{s}(x).
\end{split}
\]
Given $\epsilon\in (0,1)$, $\xi\in \mathbb{R}^N$,  and non-negative functions $s_1, s_2$, we denote
\begin{equation}
\notag
\label{eq:gamma}
\begin{split}
& \beta_{\epsilon}(\xi)=\epsilon^2+\vert \xi\vert ^{2},
\qquad
\gamma^{(p(x))}_{\epsilon}(\xi)=\beta^{\frac{p(x)-2}{2}}_{\epsilon}(\xi),
\\
&
\mathcal{F}^{(s_1, s_2)}_{\epsilon}(x,\xi)= a(x) \beta^{\frac{p(x)+s_1(x)-2}{2}}_{\epsilon}(\xi) + b(x) \beta^{\frac{q(x)+s_2(x)-2}{2}}_{\epsilon}(\xi).
\end{split}
\end{equation}
The function $\mathcal{F}_{\epsilon}^{(s_1, s_2)}(y,\xi)$ depends on $y$ implicitly, through the exponents $p$, $q$, $s_1, s_2$ and the coefficients $a$, $b$. For this reason, we will write $\mathcal{F}_{\epsilon}^{(s_1, s_2)}(x,\xi)$ if  the exponents and coefficients depend on $x\in\Omega$, or $\mathcal{F}_{\epsilon}^{(s_1, s_2)}(z,\xi)$ if at least one of these functions depends on $z=(x,t)\in Q_T$. The functions $\mathcal{F}_{\epsilon}^{(s_1,s_2)}(y,\xi)\xi$ approximate the flux: $\mathcal{F}_{0}^{(0,0)}(y,\xi)\xi\equiv \mathcal{F}(y,\xi)\xi$.

\begin{enumerate}
\item
For every $\epsilon\in (0,1)$, $\mu>0$, and $\xi\in \mathbb{R}^N$
\begin{equation}\label{eq:interchange}
    \vert \xi\vert ^{2\mu}\leq \beta^{\mu}_{\epsilon}(\xi)\leq \begin{cases}
(2\vert \xi\vert ^2)^{\mu}=2^\mu \vert \xi\vert ^{2\mu} & \text{if $\vert \xi\vert \geq \epsilon$},
\\
(2\epsilon^2)^\mu & \text{if $\vert \xi\vert <\epsilon$}
\end{cases}
\leq 2^\mu\left(1+\vert \xi\vert ^{2\mu}\right).
\end{equation}
\item For the nonnegative coefficients $a, b \in L^\infty(\Omega)$ and the exponents $s_1,s_2 \geq 0$, $p$, $q$ such that $p+s_1\geq 2$, $q+s_2\geq 2$
\[
\begin{split}
a(x)\vert \xi\vert ^{p+s_1-2}+b(x)\vert \xi\vert ^{q+s_2-2} & \leq \mathcal{F}^{(s_1,s_2)}_{\epsilon}(x,\xi)
\\
&
\leq C\left(1+a(x)\vert \xi\vert ^{p+s_1-2}+b(x)\vert \xi\vert ^{q+s_2-2}\right)
\end{split}
\]
with a constant $C$ independent of $x$ and $\xi$.
\item  For the nonnegative coefficients $a, b \in L^\infty(\Omega)$, $\epsilon \in (0,1)$, $\xi\in \mathbb{R}^N$,  and the parameters $s_1,s_2 \geq 0$
\begin{equation}
\label{eq:null-eps-prelim}
\begin{split}
\mathcal{F}^{(s_1, s_2)}_{\epsilon}(x,\xi)\beta_{\epsilon}(\xi) & \leq \begin{cases}
a(x)(2\epsilon^2)^{p+s_1}+b(x)(2\epsilon^2)^{q+s_2} & \text{if $\vert \xi\vert \leq \epsilon$},
\\
2 \mathcal{F}^{(s_1, s_2)}_\epsilon(x,\xi)\vert \xi\vert ^2 & \text{if $\vert \xi\vert >\epsilon$}
\end{cases}
\\
&
\leq C+ 2\mathcal{F}_\epsilon^{(s_1, s_2)}(x,\xi)\vert \xi\vert ^2.
\end{split}
\end{equation}
In particular,
\begin{equation}
\label{eq:null-eps}
\begin{split}
\mathcal{F}^{(s_1, s_2)}_0(x,\xi)\vert \xi\vert ^2\,dx & \equiv a(x)\vert \xi\vert ^{p+s_1}+b(x)\vert \xi\vert ^{q+s_2}
\\
&
\leq \mathcal{F}^{(s_1, s_2)}_{\epsilon}(x,\xi)\beta_{\epsilon}(\xi)
\leq  C+2\mathcal{F}^{(s_1, s_2)}_\epsilon(x,\xi)\vert \xi\vert ^2.
\end{split}
\end{equation}
\item For every $\zeta>0$, $\xi\in \mathbb{R}^N$, and $\mu\in (0,\zeta)$
\begin{equation}
\label{eq:log}
\begin{split}
\vert \xi\vert ^{\zeta}\vert \ln \vert \xi\vert \vert  & =\begin{cases}
\left(\vert \xi\vert ^{\zeta+\mu}\right)\left(\vert \xi\vert ^{-\mu}\vert \ln \vert \xi\vert \vert \right) & \text{if $\vert \xi\vert \geq 1$},
\\
\left(\vert \xi\vert ^{\zeta-\mu}\right)\left(\vert \xi\vert ^{\mu}\vert \ln \vert \xi\vert \vert \right)  & \text{if $\vert \xi\vert < 1$},
\end{cases}
\\
&
\leq C(\mu,\zeta)\left(1+\vert \xi\vert ^{\zeta+\mu}\right).
\end{split}
\end{equation}
\end{enumerate}

\begin{proposition}[Strict monotonicity]
\label{pro:strict-monotone}
Let $\epsilon\in [0,1)$. There exists a constant $C=C(p^\pm)$ such that for all $\xi,\zeta\in \mathbb{R}^N \setminus \{0\}$ and every $z\in Q_T$

\begin{equation}
\label{eq:mon-strict}
\begin{split}
\mathcal{S}_p(\xi,\eta) & :=\big(\gamma^{(p)}_\epsilon(\xi)\xi  - \gamma^{(p)}_\epsilon(\eta)\eta,\xi-\eta\big)
\\
&
\geq  C\begin{cases} \vert \xi-\zeta\vert ^{p} & \text{if $p\geq 2$},
\\
\vert \xi-\eta\vert ^2\left(\epsilon^2 +\vert \xi\vert ^2+\vert \eta\vert ^2\right)^{\frac{p-2}{2}} & \text{if $p\in (1,2)$}.
\end{cases}
\end{split}
\end{equation}
\end{proposition}

\begin{proof}
The second line of \eqref{eq:mon-strict} with $\epsilon\not =0$ is proven in \cite[Lemma 6.1]{A-S}, in the case $\epsilon=0$ the proof can be found in \cite[Lemma 17.3]{Chipot-mon}. For this reason we only have to consider the case $p\geq 2$. For every $\xi,\eta\in \mathbb{R}^N$

\[
\begin{split}
\mathcal{S}_p(\xi,\eta)= \gamma^{(p)}_\epsilon(\xi)\vert \xi\vert ^2 + \gamma^{(p)}_\epsilon(\eta)\vert \eta\vert ^2- \left(\gamma^{(p)}_\epsilon(\xi)+\gamma^{(p)}_\epsilon(\eta)\right)(\xi,\eta).
\end{split}
\]
Expressing the inner product $(\xi,\eta)$ from the relation

\[
\vert \xi-\eta\vert ^2=(\xi-\eta,\xi-\eta)=\vert \xi\vert ^2+\vert \eta\vert ^2-2(\xi,\eta)
\]
we transform the previous equality to the form

\[
\begin{split}
\mathcal{S}_p(\xi,\eta) & = \frac{1}{2}\left(\gamma^{(p)}_\epsilon(\xi) +\gamma^{(p)}_\epsilon(\eta)\right)\vert \xi-\eta\vert ^2 + \frac{1}{2}\left(\gamma^{(p)}_\epsilon(\xi)-\gamma^{(p)}_\epsilon(\eta)\right) (\vert \xi\vert ^2-\vert \eta^2\vert ).
\end{split}
\]
For $p\geq 2$ the function $(\epsilon^2+\vert \xi\vert ^2)^{\frac{p-2}{2}}$ is monotone increasing as a function of $\vert \xi\vert $, therefore the second term is nonnegative. Dropping it, we estimate $\mathcal{S}_p(\xi,\eta)$ from below in the following way:

\[
\begin{split}
\mathcal{S}_p(\xi,\eta) & \geq \frac{1}{2}\left(\gamma^{(p)}_\epsilon(\xi) +\gamma^{(p)}_\epsilon(\eta)\right)\vert \xi-\eta\vert ^2
\\
&
\geq \frac{1}{2}\left(\vert \xi\vert ^{p-2}+\vert \eta\vert ^{p-2}\right)\vert \xi-\eta\vert ^2\geq \frac{1}{2C_p}(\vert \xi\vert +\vert \eta\vert )^{p-2}\vert \xi-\eta\vert ^2,
\end{split}
\]
with the constant $C_p$ from the inequality $(a+b)^{p-2}\leq C_p(a^p+b^p)$, $p\geq 2$, $a,b\geq 0$.
Finally,

\[
\begin{split}
\vert \xi-\eta\vert ^{p} & =\left(\vert \xi-\eta\vert ^2\right)^{\frac{p-2}{2}}\vert \xi-\eta\vert ^2
= \left((\xi-\eta,\xi-\eta)\right)^{\frac{p-2}{2}}\vert \xi-\eta\vert ^2
\\
& =(\vert \xi\vert ^2 -2(\xi,\eta)+\vert \eta\vert ^2)^{\frac{p-2}{2}}\vert \xi-\eta\vert ^2
\leq (\vert \xi\vert ^2+2\vert \xi\vert \vert \eta\vert +\vert \eta\vert ^2)^{\frac{p-2}{2}}\vert \xi-\eta\vert ^2
\\
&
=(\vert \xi\vert +\vert \eta\vert )^{p-2}\vert \xi-\eta\vert ^2
\leq 2C_p\mathcal{S}_p(\xi,\eta).
\end{split}
\]
\end{proof}

Now we consider the functions defined on the cylinder $Q_T$ and depending on the variables $z=(x,t)\in \Omega\times (0,T)$. Let us denote

\begin{equation}
\label{eq:prelim-notation}
\begin{split}
& \mathcal{N}(\nabla w)=\int_{Q_T}(a(z)\vert \nabla w\vert ^{p(z)}+b(z)\vert \nabla w\vert ^{q(z)})\,dz,
\\
& \mathcal{G}_{\epsilon}(\nabla u,\nabla v)= \int_{Q_T} \left( \mathcal{F}_{\epsilon}^{(0,0)}(z,\nabla u)\nabla u-\mathcal{F}_{\epsilon}^{(0,0)}(z,\nabla v)\nabla v\right)\cdot \nabla (u-v)\,dz.
\end{split}
\end{equation}

\begin{proposition}
\label{pro:strict-1}
For every $u,v\in \mathbb{W}_{\overline{s}(\cdot)}(Q_T)$ and any $\epsilon\in [0,1)$
\begin{equation}
\label{eq:mon-strict-cor}
\mathcal{N}(\nabla u-\nabla v)\leq C\left(\mathcal{G}_\epsilon^{\frac{\overline{s}^+}{2}}(\nabla u,\nabla v)+\mathcal{G}_\epsilon^{\frac{\underline{s}^-}{2}}(\nabla u,\nabla v)+\mathcal{G}_\epsilon(\nabla u,\nabla v)\right)
\end{equation}
with a constant $C$ depending on $p^\pm$, $q^\pm$, $a^+$, $b^+$, $\|  \nabla u\|  _{\overline{s}(\cdot),Q_T}$, $\|  \nabla v\|  _{\overline{s}(\cdot),Q_T}$.
\end{proposition}

\begin{proof}
Because of \eqref{eq:mon-strict} it is sufficient to consider in detail the terms with the coefficient $a(z)$, the terms with the coefficient $b(z)$ are considered likewise.
Let us split the domain $Q_T$ into two parts,

\[
\begin{split}
& Q_+=Q_T\cap \{p\geq 2\}, \quad Q_-=Q_T\cap \{p< 2\}.
\end{split}
\]
Given $\epsilon \in [0,1)$, we denote
\[
\mathcal{R}(\nabla u,\nabla v)=(\epsilon+\vert \nabla u\vert ^2+\vert \nabla v\vert ^2)^{\frac{2-p(z)}{2}},\qquad \left(\mathcal{R}^{\frac{p}{2}}(\nabla u,\nabla v)\right)^{\frac{2}{2-p}}\leq C(1+\vert \nabla u\vert ^p+\vert \nabla v\vert ^p)\quad \text{on $Q_-$}.
\]
By the generalized H\"older inequality \eqref{eq:Holder}

\[
\begin{split}
\int_{Q_-} & a(z)\vert \nabla (u-v)\vert ^{p(z)}\,dz =\int_{Q_-} a^{1-\frac{p}{2}} \mathcal{R}^{\frac{p}{2}}(\nabla u,\nabla v)\left(a\mathcal{R}^{-1}(\nabla u,\nabla v)\vert \nabla(u-v)\vert ^2 \right)^{\frac{p}{2}}\,dz
\\
& \leq 2(a^+)^{1-\frac{p}{2}}\|  \mathcal{R}^{\frac{p}{2}}(\nabla u,\nabla v)\|  _{\frac{2}{2-p(\cdot)},Q_-} \left\|  a^{\frac{p}{2}}\mathcal{R}^{-\frac{p}{2}}(\nabla u,\nabla v)\vert \nabla (u-v)\vert ^{p}\right\|  _{\frac{2}{p(\cdot)},Q_-},
\end{split}
\]
where

\[
\begin{split}
\|  \mathcal{R}^{\frac{p}{2}}(\nabla u,\nabla v)\|  _{\frac{2}{2-p(\cdot)},Q_-} & \leq
C(p^\pm, \|  \nabla u\|  _{p(\cdot),Q_T},\|  \nabla v\|  _{p(\cdot),Q_T}),
\end{split}
\]

\[
\begin{split}
& \left\|  a^{\frac{p}{2}}\mathcal{R}^{-\frac{p}{2}}(\nabla u,\nabla v)\vert \nabla (u-v)\vert ^{p}\right\|  _{\frac{2}{p(\cdot)},Q_-}
\\
 & \qquad \qquad
\leq \max\left\{\left(\int_{Q_-}\frac{a}{\mathcal{R}(\nabla u, \nabla v)}\vert \nabla (u-v)\vert ^2\,dz\right)^{\frac{p^+}{2}}, \left(\int_{Q_T}\ldots\right)^{\frac{p^-}{2}}\right\}.
\end{split}
\]
Gathering these inequalities and using \eqref{eq:mon-strict} we obtain:

\[
\begin{split}
\int_{Q_-}a\vert \nabla (u-v)\vert ^p\,dz & \leq C\left( \left(\int_{Q_-}a\mathcal{S}_{p(z)}(\nabla u,\nabla v)\,dz\right)^{\frac{p^+}{2}}
+ \left(\int_{Q_-}
\ldots
\right)^{\frac{p^-}{2}}\right)
\\
& \leq C\left(\mathcal{G}_{\epsilon}^{\frac{\overline{s}^+}{2}}(\nabla u,\nabla v) + \mathcal{G}_{\epsilon}^{\frac{\underline{s}^-}{2}}(\nabla u,\nabla v)\right).
\end{split}
\]
The estimate on the integral of $a\vert \nabla (u-v)\vert ^p$ over $Q_+$ is straightforward and follows directly from the first line in \eqref{eq:mon-strict}:

\[
\int_{Q_+}a\vert \nabla (u-v)\vert ^{p}\,dz\leq C \mathcal{G}_{\epsilon}(\nabla u,\nabla v).
\]
The proof is concluded by gathering these inequalities with the corresponding inequalities for the terms with the coefficient $b(z)$ and the exponent $q(z)$.
\end{proof}
Let

\[
\mathcal{W}(Q_T)=\left\{u:Q_T\mapsto \mathbb{R}: u\in L^2(Q_T), \,\mathcal{N}(\nabla u)<\infty\right\}.
\]

\begin{lemma}
\label{le:cont-emb-1}
There is the continuous embedding $\mathcal{W}(Q_T)\subset \mathbb{W}_{\underline{s}(\cdot)}(Q_T)$.
\end{lemma}

\begin{proof} We have to show that for every $u\in \mathcal{W}(Q_T)$ the norm $\|  u\|  _{\mathbb{W}_{\underline{s}(\cdot)}(Q_T)}$ is bounded and

\[
\|  u\|  _{\mathbb{W}_{\underline{s}(\cdot)}(Q_T)}=\|  u\|  _{2,Q_T}+\|  \nabla u\|  _{\underline{s}(\cdot),Q_T}\to 0\quad \text{as $\|  u\|  ^2_{2,Q_T}+\mathcal{N}(\nabla u)\to 0$}.
\]
Let $u\in \mathcal{W}(Q_T)$. The following inequality holds:
\[
\begin{split}
\alpha\int_{Q_T} & \vert \nabla u\vert ^{\underline{s}(z)}\,dz \leq \int_{Q_T}(a(z)+b(z))\vert \nabla u\vert ^{\underline{s}(z)}\,dz
\\
&
\equiv \int_{Q_T}a^{1-\frac{\underline{s}}{p}}\left(a\vert \nabla u\vert ^{p}\right)^{\frac{\underline{s}}{p}}\,dz + \int_{Q_T}b^{1-\frac{\underline{s}}{q}}\left(b\vert \nabla u\vert ^{q}\right)^{\frac{\underline{s}}{q}}\,dz
\\
& \leq A^+\int_{Q_T}\left(a\vert \nabla u\vert ^{p}\right)^{\frac{\underline{s}}{p}}\,dz + B^+ \int_{Q_T}\left(b\vert \nabla u\vert ^{q}\right)^{\frac{\underline{s}}{q}}\,dz
\equiv A^+\mathcal{J}_a+B^+\mathcal{J}_b
\end{split}
\]
with the constants $A^+=\sup_{Q_T}a^{1-\frac{\underline{s}(z)}{p(z)}}(z)$, $B^+=\sup_{Q_T}b^{1-\frac{\underline{s}(z)}{q(z)}}(z)$. The estimates on $\mathcal{J}_a$ and $\mathcal{J}_b$ are similar, for this reason we provide the details only for the first one. By the generalized H\"older inequality

\[
\begin{split}
\mathcal{J}_a & \leq  2C_a\|  (a\vert \nabla u\vert ^{p})^{\frac{\underline{s}}{p}} \|  _{\frac{p(\cdot)}{\underline{s}(\cdot)},Q_T}
  \\
  &
  \leq 2C_a
 \max \left\{
 \left(\int_{Q_T}a\vert \nabla u\vert ^{p}\,dz\right)^{\frac{\underline{s}^-}{p^+}}, \left(\int_{Q_T}a\vert \nabla u\vert ^{p}\,dz\right)^{\frac{\overline{s}^+}{p^-}}\right\}
 \\
 & \leq 2C_a\left(\mathcal{N}^{\frac{\underline{s}^-}{p^+}}(\nabla u)+\mathcal{N}^{\frac{\overline{s}^+}{p^-}}(\nabla u)\right)
 \end{split}
\]
with the constant $C_a=\|  1\|  _{({p}(\cdot)/\underline{s}(\cdot))',Q_T}$. Gathering this estimate with the similar estimate for $\mathcal{J}_b$ with constant $C_b=\|  1\|  _{(q(\cdot)/\underline{s}(\cdot))',Q_T}$ we obtain

\[
\alpha \int_{Q_T}\vert \nabla u\vert ^{\underline{s}(z)}\,dz\leq 4(C_a+C_b)(A^++B^+)\left(\mathcal{N}^{\frac{\underline{s}^-}{\overline{s}^+}}(\nabla u)+\mathcal{N}^{\frac{\overline{s}^+}{\underline{s}^-}}(\nabla u)\right).
\]
\end{proof}
The next assertion is an immediate byproduct of Lemma \ref{le:cont-emb-1}.
\begin{lemma}
\label{le:cont-convergence}
Let $\epsilon\in [0,1)$. If $u,v\in \mathbb{W}_{\overline{s}(\cdot)}(Q_T)$, then

\[
\int_{Q_T}\vert \nabla (u-v)\vert ^{\underline{s}(z)}\,dz\to 0 \quad \text{when}\quad
\mathcal{G}_{\epsilon}(\nabla u,\nabla v)\to 0.
\]
\end{lemma}

\section{Interpolation inequalities}
\label{sec:interpolation}
The bulk of this section is devoted to deriving interpolation inequalities for functions of the variable $x\in \Omega$.
By agreement, we will write $\displaystyle \vert u_{xx}\vert ^2=\sum_{i,j=1}^{N}\left\vert D^2_{x_ix_j}u\right\vert ^2.$

\subsection{Global inequalities}
\begin{lemma}\label{lem:interpol} Let $\partial \Omega \in C^2$, $u\in
C^1(\overline{\Omega})\cap H^{2}_0(\Omega)$. Assume that $p$, $q$, $a$, $b$ are functions of $x\in \Omega$ and satisfy conditions \eqref{eq:s} and \eqref{eq:a-b} in $\overline{\Omega}$. If

\[
\int_\Omega\mathcal{F}^{(0,0)}_{\epsilon}(x,\nabla
u)\vert u_{xx}\vert ^2\,dx<\infty,\qquad \int_{\Omega}u^2\,dx= M_0,
\]
then for every pair of nonnegative functions $r_1, r_2: \overline{\Omega} \to \mathbb{R}$, $r_1, r_2 \in C^{0,1}(\overline{\Omega})$ with the Lipschitz constants $L_{r_1}$, $L_{r_2}$, and arbitrary parameters $\delta, \nu \in (0,1)$
\begin{equation}
\label{eq:principal}
\begin{split}
\int_{\Omega} & \mathcal{F}_{\epsilon}^{(r_1, r_2)}(x,\nabla u)\vert \nabla u\vert ^2\,dx  \leq \delta \int_{\Omega}\mathcal{F}^{(0,0)}_{\epsilon}(x,\nabla
u)\vert u_{xx}\vert ^{2}\,dx
\\
&
\leq C_0 + C_1  \int_{\Omega}u^2\,dx
\\
& + C_2 \int_{\Omega} \left(\vert u\vert ^{\frac{p(x)+r_1(x)}{1-\nu}}+ \vert u\vert ^{\frac{q(x)+r_2(x)}{1-\nu}}\right)\,dx
\\
&
+ C_3  \int_\Omega  \vert u\vert \ \left( \beta_{\epsilon}^{\frac{p(x)+r_1(x)-1}{2}}(\nabla u) + \beta_{\epsilon}^{\frac{q(x)+r_2(x)-1}{2}}(\nabla u)\right)\,dx\\
& + C_4\int_\Omega \vert u\vert ^2 \ \left(  \beta_{\epsilon}^{\frac{p(x)+2r_1(x)-2}{2}}(\nabla u)+  \beta_{\epsilon}^{\frac{q(x)+2r_2(x)-2}{2}}(\nabla u)\right)\,dx
\end{split}
\end{equation}
with independent of $u$ constants $C_i=C_i(\partial
\Omega,\delta,\alpha, \nu, s^\pm,N,r_i, L_{r_i}, M_0, a^+, b^+)$.
\end{lemma}

\begin{proof}
Let $\mathbf{n}$ denotes the outer normal to $\partial\Omega$. By employing Green's formula

\begin{equation*}
\begin{split}
\int_{\Omega}a(x) \beta_{\epsilon}^{\frac{p(x)+r_1(x)-2}{2}} &(\nabla u)
\vert \nabla u\vert ^2\,dx  =\int_{\Omega} a(x)\ \beta_{\epsilon}^{\frac{p(x)+r_1(x)-2}{2}}(\nabla u) \ \nabla u\cdot \nabla u\,dx
\\
& = \int_{\partial \Omega} a(x) u \beta_{\epsilon}^{\frac{p(x)+r_1(x)-2}{2}}(\nabla u) \ \nabla u \cdot
\mathbf{n} \,dS
\\
&\quad
-\int_{\Omega}u\ \operatorname{div}\left(a(x)\beta_{\epsilon}^{\frac{p(x)+r_1(x)-2}{2}}(\nabla
u) \nabla u\right)\,dx\\
& = -\int_{\Omega}u\ \operatorname{div}\left(a(x)\beta_{\epsilon}^{\frac{p(x)+r_1(x)-2}{2}}(\nabla
u) \nabla u\right)\,dx := -J.
\end{split}
\end{equation*}
A straightforward computation leads to the following representation
\begin{equation*}
\begin{split}
J &= \int_\Omega a(x) u\ \beta_{\epsilon}^{\frac{p(x)+r_1(x)-2}{2}}(\nabla
u)\ \Delta u\,dx +
 \int_{\Omega} u\ \beta_{\epsilon}^{\frac{p(x)+r_1(x)-2}{2}} (\nabla
u) \  \nabla u \cdot
\nabla a \,dx
\\
&  + \int_\Omega (p(x)+r_1(x)-2)\ a(x)u\ \beta_{\epsilon}^{\frac{p(x)+r_1(x)-2}{2}-1}(\nabla
u) \sum_{i=1}^n
\left(u_{x_i}\sum_{j=1}^{n}u_{x_j}u_{x_ix_j}\right)\,dx
\\
&  + \int_{\Omega}a(x)u\ \beta_{\epsilon}^{\frac{p(x)+r_1(x)-2}{2}} (\nabla
u) \ln(\beta_{\epsilon}(\nabla u)) \nabla u \cdot
\nabla (p+ r_1) \,dx
\\
& := J_1 + J_2 + J_3 + J_4,
\end{split}
\end{equation*}
The terms $J_i$ are estimated as follows. Using Young's inequality, for any $\lambda_1 \in (0,1)$ we have
\begin{equation}\label{est:J1}
    \begin{split}
        & \vert J_1\vert  \leq \int_\Omega \left(\vert u\vert  \ (a(x))^{\frac{1}{2}} \beta_{\epsilon}^{\frac{p(x)+r_1(x)-2}{2}- \frac{p(x)-2}{4}}(\nabla u) \right)
        \\
        & \qquad
        \times \left((a(x))^{\frac{1}{2}} \ \beta_{\epsilon}^{\frac{p(x)-2}{4}}(\nabla u)\ \vert \Delta u\vert \right)\,dx\\
        & \leq \lambda_1 \int_\Omega a(x) \ \beta_{\epsilon}^{\frac{p(x)-2}{2}}(\nabla u)\ \vert \Delta u\vert ^2\,dx + C' \int_\Omega \vert u\vert ^2 \ a(x)\ \beta_{\epsilon}^{\frac{p(x)+2r_1(x)-2}{2}}(\nabla u)\,dx\\
        & \leq \lambda_1 \int_\Omega a(x) \ \beta_{\epsilon}^{\frac{p(x)-2}{2}}(\nabla u)\ \vert u_{xx}\vert ^2\,dx + C \int_\Omega \vert u\vert ^2  \beta_{\epsilon}^{\frac{p(x)+2r_1(x)-2}{2}}(\nabla u)\,dx
    \end{split}
\end{equation}
with $C=C(\lambda_1, a^+)$
 and
\begin{equation}\label{est:J2}
    \begin{split}
        \vert J_2\vert  & \leq \int_\Omega \vert u\vert  \ \vert \nabla a\vert \ \beta_{\epsilon}^{\frac{p(x)+r_1(x)-2}{2}}(\nabla u)\ \vert \nabla u\vert \,dx
        \\
        &
        \leq L_{a,b} \int_\Omega  \vert u\vert \  \beta_{\epsilon}^{\frac{p(x)+r_1(x)-1}{2}}(\nabla u)\,dx.
    \end{split}
\end{equation}
By Young's inequality, for any $\lambda_2 \in (0,1)$ we obtain
\begin{equation}\label{est:J3}
\begin{split}
\vert J_3\vert  & \leq  \lambda_2 \int_\Omega a(x) \ \beta_{\epsilon}^{\frac{p(x)-2}{2}}(\nabla u)\ \vert u_{xx}\vert ^2\,dx + C \int_\Omega \vert u\vert ^2  \beta_{\epsilon}^{\frac{p(x)+2r_1(x)-2}{2}}(\nabla u)\,dx
\end{split}
\end{equation}
with $C= C(p^\pm,r_1, \lambda_2, a^+)$, $a^+=\max_{\overline{\Omega}}a(x)$,  and
\begin{equation}\label{est:J4}
\begin{split}
\vert J_4\vert  & \leq
(L_{p,q} + L_{r_1}) \int_{\Omega}\vert u\vert  \ a(x) \ \beta_{\epsilon}^{\frac{p(x)+r_1(x)-1}{2}}(\nabla
u) \vert \ln \beta_{\epsilon}(\nabla u)\vert \,dx.
\end{split}
\end{equation}
Applying \eqref{eq:log} with $\vert \xi\vert =\beta_\epsilon(\nabla u)$, $\zeta=\frac{p(x)+r_1(x)-1}{2}$ and $\mu=\nu/2\in (0,1/2)$, we estimate
\begin{equation}\label{eq:elem1}
\begin{split}
\beta_{\epsilon}^{\frac{p(x)+r_1(x)-1}{2}}(\nabla u)\vert \ln
\beta_{\epsilon}(\nabla u)\vert  & \leq
C+C(\nu)\beta^{\frac{p(x)+r_1(x)-1+\nu}{2}}_{\epsilon}(\nabla u).
\end{split}
\end{equation}
By virtue of \eqref{eq:null-eps}
\begin{equation}\label{eq:elem2}
\vert \nabla u\vert ^{{\frac{p(x)+r_1(x)}{2}}} \leq \beta_{\epsilon}^{\frac{p(x)+r_1(x)}{2}}(\nabla u) \leq
C+2
\beta_{\epsilon}^{\frac{p(x)+r_1(x)-2}{2}}(\nabla u)\vert \nabla u\vert ^2.
\end{equation}
By using \eqref{eq:elem1} and \eqref{eq:elem2} in \eqref{est:J4}, and applying then Young's inequality with $\lambda_3 \in (0,1)$, we obtain
\begin{equation}\label{est:J4-final}
\begin{split}
& \vert J_4\vert  \leq C \int_{\Omega}a(x) \vert u\vert  \,dx
+ C\int_{\Omega} \left(\vert u\vert \ (a(x))^{\frac{1-\nu}{p(x)+r_1(x)}}\right)
\\
& \qquad
\times
\left( (a(x))^{\frac{p(x)+r_1(x)-1+\nu}{p(x)+r_1(x)}} \beta^{\frac{p(x)+r_1(x)-1+\nu}{2}}_{\epsilon}(\nabla u) \right)\,dx
\\
& \leq  C (1 + M_0^{\frac{1}{2}}) + \frac{\lambda_3}{2}
\int_{\Omega} a(x) \beta_{\epsilon}^{\frac{p(x)+r_1(x)}{2}}(\nabla
u)\,dx+
C_{\lambda_3}' \int_{\Omega}a(x) \vert u\vert ^{\frac{p(x)+r_1(x)}{1-\nu}}\,dx\\
& \leq C_1 +
C_2 \int_{\Omega} \vert u\vert ^{\frac{p(x)+r_1(x)}{1-\nu}}\,dx + \lambda_3 \int_{\Omega} a(x) \beta_{\epsilon}^{\frac{p(x)+r_1(x)-2}{2}}(\nabla
u) \vert \nabla u\vert ^2\,dx,
\end{split}
\end{equation}
where $C_1= C_1(a^+,  p^\pm,r_1, \lambda_3, M_0)$ and $C_2= C_2(\lambda_3, p^\pm,\nu, a^+)$. By combining \eqref{est:J1}-\eqref{est:J3} and \eqref{est:J4-final}, we get the following estimate:
\[
\begin{split}
\int_{\Omega}a(x) &  \beta_{\epsilon}^{\frac{p(x)+r_1(x)-2}{2}} (\nabla u)
\vert \nabla u\vert ^2\,dx \leq C_1 + C_2 \int_{\Omega} \vert u\vert ^{\frac{p(x)+r_1(x)}{1-\nu}}\,dx
\\
&
+  C_3 \int_\Omega  \vert u\vert \  \beta_{\epsilon}^{\frac{p(x)+r_1(x)-1}{2}}(\nabla u)\,dx  + C_4 \int_\Omega \vert u\vert ^2 \ \beta_{\epsilon}^{\frac{p(x)+2r_1(x)-2}{2}}(\nabla u)\,dx
\\
&
+ (\lambda_1 + \lambda_2) \int_\Omega a(x) \ \beta_{\epsilon}^{\frac{p(x)-2}{2}}(\nabla u)\ \vert u_{xx}\vert ^2\,dx.
\end{split}
\]
We get the claim by repeating the same arguments for the term with the exponent $q$ on the left-hand side of \eqref{eq:principal} and gathering the results.
\end{proof}
Let us rewrite inequality \eqref{eq:principal} in the form
\begin{equation}\label{eq:complete}
\begin{split}
\int_\Omega & \mathcal{F}_\epsilon^{(r_1, r_2)}(x,\nabla u)\vert \nabla u\vert ^2\,dx \leq \delta \int_{\Omega}\mathcal{F}^{(0,0)}_{\epsilon}(x,\nabla
u)\vert u_{xx}\vert ^{2}\,dx
\\
&
+ C_0 + C_1  \int_{\Omega}u^2\,dx  + C_2  \mathcal{Q}_1^{(r_1, r_2)} + C_3  \mathcal{Q}_2^{(r_1, r_2)} + C_4 \mathcal{Q}_3^{(r_1, r_2)},
\end{split}
\end{equation}
where
\[
\begin{split}
& \mathcal{Q}_1^{(r_1, r_2)}:= \int_{\Omega} \left(\vert u\vert ^{\frac{p(x)+r_1(x)}{1-\nu}}+ \vert u\vert ^{\frac{q(x)+r_2(x)}{1-\nu}}\right)\,dx \quad \ \text{for some} \ \nu \in (0,1),
\\
&
\mathcal{Q}_2^{(r_1, r_2)}:= \int_\Omega  \vert u\vert \ \left( \beta_{\epsilon}^{\frac{p(x)+r_1(x)-1}{2}}(\nabla u) + \beta_{\epsilon}^{\frac{q(x)+r_2(x)-1}{2}}(\nabla u)\right)\,dx,
\\
& \mathcal{Q}_3^{(r_1, r_2)}:=  \int_\Omega \vert u\vert ^2 \ \left(  \beta_{\epsilon}^{\frac{p(x)+2r_1(x)-2}{2}}(\nabla u)+  \beta_{\epsilon}^{\frac{q(x)+2r_2(x)-2}{2}}(\nabla u)\right)\,dx.
\end{split}
\]
The integrals $\mathcal{Q}_i^{(r_1, r_2)}$ will be estimated separately and under a special choice of the functions $r_1(x)$, $r_2(x)$. We assume that the exponents $p$, $q$ are subject to the balance condition \eqref{eq:gap-z},  which for the independent of $t$ functions $p$, $q$ reads
\begin{equation}\label{eq:osc-p-q}
    \max_{x \in \overline{\Omega}} \vert p(x) -q(x)\vert  < r_\ast.
\end{equation}
Given $p(x)$, $q(x)$, we introduce the Lipschitz-continuous functions
\begin{equation}
\label{eq:r-i}
r_1(x):= r^\sharp + \underline{s}(x) - p(x), \quad
r_2(x):=
r^\sharp + \underline{s}(x) - q(x).
\end{equation}
Let $\{\Omega_i\}_{i=1}^{K}$ be a finite cover of $\Omega$, $\Omega_i\subset \Omega$,
$\partial\Omega_i\in C^2$, $\Omega\subseteq \bigcup_{i=1}^K\Omega_i$. Set
\[
\begin{split}
\underline{s}_i^+=\max_{\overline{\Omega}_i}\underline{s}(x), \quad \underline{s}_i^-=\min_{\overline{\Omega}_i}\underline{s}(x).
\end{split}
\]
Fix an arbitrary number $\varsigma \in (0, r^\sharp)$. By continuity of the map $\underline{s}(x)$,  the cover $\{\Omega_i\}_{i=1}^{K}$ can be chosen in such a way that for every $i=1,2,\ldots,K$
\begin{equation}
\label{eq:osc-loc}
0\leq  \underline{s}_i^+-\underline{s}_i^- <  \varsigma.
\end{equation}
The number of elements in the cover $\{\Omega_i\}$ depends on $\varsigma$ and the modules of continuity $\omega_p$, $\omega_q$ of functions $p$ and $q$ in $\Omega$: $K=K(\omega_p,\omega_q,\varsigma)$.
Under the choice \eqref{eq:r-i} we have $p(x)+r_1(x)=q(x)+r_2(x)=\underline{s}(x)+r^\sharp$, where the value of the constant $r^\sharp$ is prompted by the classical embedding inequalities used in the proofs of the forthcoming estimates on the integrals $\mathcal{Q}^{(r_1,r_2)}_{i}$.

\begin{lemma}
\label{le:Q1}
Let the conditions of Lemma \ref{lem:interpol} be fulfilled.
Then for every $
\varsigma \in \left(0,r^\sharp\right)$ with $r^\sharp\equiv \dfrac{4}{N+2}$, every $\sigma >0$, and $r_1(x)$, $r_2(x)$ defined in \eqref{eq:r-i}

\[
\mathcal{Q}_{1}^{(r_1-\varsigma, r_2-\varsigma)}\leq C+\sigma  \int_{\Omega}\mathcal{F}_{\epsilon}^{(r_1-\varsigma, r_2-\varsigma)}(x,\nabla u)\vert \nabla u\vert ^2\,dx
\]
with a constant $C$ depending on $\sigma$, $N$, $M_0$, $r_1$, $r_2$, $p^\pm$, $q^\pm$, $\varsigma$, and the modules of continuity $\omega_p$, $\omega_q$ of the exponents $p$, $q$ in $\Omega$.
\end{lemma}

\begin{proof}
For a fixed $\varsigma \in (0, r^\sharp)$, it is sufficient to prove the required inequality for each of the sets $\Omega_i$ from the chosen cover of $\Omega$. Gathering these estimates for all $i=1,2,\ldots,K$, we obtain the needed assertion. Let $\mathcal{Q}_{1,i}^{(r_1-\varsigma, r_2-\varsigma)}$ denote the integrals $\mathcal{Q}_{1}^{(r_1-\varsigma, r_2-\varsigma)}$ over the domains $\Omega_i$ instead of $\Omega$. By Young's inequality, for every $i=1,2,\ldots,K$ \begin{equation}\label{est:Q1}
\begin{split}
  & \mathcal{Q}_{1,i}^{(r_1-\varsigma, r_2-\varsigma)}
  =\int_{\Omega_i}\left(\vert u\vert ^{\frac{p+r_1-\varsigma}{1-\nu}} + \vert u\vert ^{\frac{q+r_2- \varsigma}{1-\nu}}\right)\,dx = 2 \int_{\Omega_i}  \vert u\vert ^{\frac{\underline{s}_i(x)+r^{\sharp}-\varsigma}{1-\nu}}\,dx
  \\
  & \quad = \left( \int_{\Omega_i \cap \{\vert u\vert  \leq 1\}} \ldots + \int_{\Omega_i \cap \{\vert u\vert  > 1\}}\ldots \right) \leq C + \int_{\Omega_i} \vert u\vert ^{\frac{\underline{s}^+_i+r^{\sharp}-\varsigma}{1-\nu}}\,dx
\end{split}
\end{equation}
with a finite constant $C$ independent of $u$.
To estimate the second term we apply the Gagliardo-Nirenberg inequality (\cite[Lemma 1.29, Chapter 1]{ant-shm-book-2015}):

\begin{equation}
\label{eq:G-N} \|  u\|  _{\sigma,\Omega_i}^{\sigma}\leq C_1\|  \nabla
u\|  _{\underline{s}^-_i,\Omega_i}^{\sigma\theta}\|  u\|  _{2,\Omega_i}^{\sigma(1-\theta)}
+C_2\|  u\|  ^{\sigma}_{2,\Omega_i}\leq C_1'\|  \nabla
u\|  _{\underline{s}^-_i,\Omega_i}^{\sigma\theta}+C_2M_0^{\frac{\sigma}{2}},
\end{equation}
with $
C'_1=C_1M_0^{\frac{\sigma}{2}(1-\theta)}$ and

\[
\sigma=\dfrac{\underline{s}^+_i+r^{\sharp}-\varsigma}{1-\nu} >\underline{s}^+_i+r^{\sharp}-\varsigma> \underline{s}^-_i,\quad \theta=\dfrac{\underline{s}^-_i}{\sigma}\in (0,1),\quad \dfrac{1}{\sigma}=\left(\frac{1}{\underline{s}^-_i} -\frac{1}{N}\right)\theta+\frac{1-\theta}{2}.
\]
Such a choice of the parameters $\sigma$, $\theta$ is possible, provided that
\[
\nu=1-\dfrac{\underline{s}^+_i+r^{\sharp}-\varsigma}{\underline{s}^-_i}\dfrac{N}{N+2}>0\quad \Leftrightarrow\quad \underline{s}^+_i-\underline{s}^-_i+r^{\sharp}-\varsigma<\underline{s}^-_i\dfrac{2}{N}.
\]
The last inequality is true due to the assumption $\underline{s}^-_i>\frac{2N}{N+2}$, the choice of cover $\{\Omega_i\}$ with condition \eqref{eq:osc-loc} and the inequality
\[
\begin{split}
\underline{s}_i^+-\underline{s}_i^- + r^{\sharp}-\varsigma <  \frac{4}{N+2} =\dfrac{2}{N}\dfrac{2N}{N+2} <\frac{2}{N}\underline{s}^-_i.
\end{split}
\]
By the definition of functions $r_1$, $r_2$, and by virtue of \eqref{eq:null-eps-prelim} and Lemma \ref{le:cont-emb-1}, we obtain
\begin{equation}\label{est:transfer}
    \alpha \int_{\Omega_i}\vert \nabla u\vert ^{\underline{s}(x)+r^{\sharp}-\varsigma} \,dx \leq \int_{\Omega_i} \mathcal{F}_{\epsilon}^{(r_1-\varsigma, r_2-\varsigma)}(x,\nabla u)\beta_\epsilon(\nabla u)\,dx.
\end{equation}
Applying Young's inequality and using \eqref{est:transfer} and \eqref{eq:null-eps}, we conclude that have for every $\lambda>0$
\[
\begin{split}
  \mathcal{Q}_{1,i}^{(r_1-\varsigma, r_2-\varsigma)}\leq C\left(1+\|  \nabla u\|  ^{\underline{s}_i^-}_{\underline{s}_i^-,\Omega_i}\right)&\leq C''+\lambda \alpha\int_{\Omega_i}\vert \nabla u\vert ^{\underline{s}(x)+r^{\sharp}-\varsigma} \,dx\\
&\leq C+\lambda \int_{\Omega_i} \mathcal{F}_{\epsilon}^{(r_1-\varsigma, r_2-\varsigma)}(x,\nabla u)\vert \nabla u\vert ^2\,dx.
\end{split}
\]
\end{proof}

\begin{lemma}
\label{le:Q2}
Under the conditions of Lemma \ref{le:Q1}, for every $ \varsigma \in (0,r_\ast)$

\[
\begin{split}
\mathcal{Q}_2^{(r_1-\varsigma, r_2-\varsigma)} & \leq C +\lambda\int_{\Omega}\mathcal{F}_\epsilon^{(r_1-\varsigma, r_2-\varsigma)}(x,\nabla u)\vert \nabla
u\vert ^2\,dx
\end{split}
\]
with an arbitrary $\lambda>0$ with a constant $C=C(\lambda, \omega_p,\omega_q,N,M_0,\varsigma,\alpha, p^\pm, q^\pm)$.
\end{lemma}

\begin{proof}
As in the proof of Lemma \ref{le:Q1}, it is sufficient to consider the integrals $\mathcal{Q}_{2,i}^{(r_1-\varsigma, r_2-\varsigma)}$ over the elements of the cover $\{\Omega_i\}$ satisfying condition \eqref{eq:osc-loc}. For $\varsigma \in (0, r_\ast)$ we have $r^{\sharp}-\varsigma > r_\ast$. Set $b=\underline{s}_i^-+r_\ast$. By Young's inequality

\begin{equation}
\label{eq:Q-2-prim}
\begin{split}
\mathcal{Q}_{2,i}^{(r_1-\varsigma, r_2-\varsigma)} & \leq C M_0^{\frac{1}{2}} + \int_{\Omega_i} \vert u\vert \  \beta_{\epsilon}^{\frac{\underline{s}_i(x)+r^{\sharp}-\varsigma-1}{2}}(\nabla u) \,dx
\\
&
\leq C + C_\lambda\int_{\Omega_i}\vert u\vert ^{b\frac{N+2}{N}}\,dx+\lambda \int_{\Omega_i} \beta_\epsilon^{\frac{\kappa}{2}}(\nabla u)\,dx
\end{split}
\end{equation}
with an arbitrary $\lambda>0$ and the exponent

\[
\quad \kappa=(\underline{s}_i(x)+r^{\sharp}-\varsigma-1) \frac{b\frac{N+2}{N}}{b\frac{N+2}{N}-1}>0.
\]
The second term on the right-hand side of \eqref{eq:Q-2-prim} is estimated by the Gagliardo-Nirenberg inequality

\[
\int_{\Omega_i}\vert u\vert ^{b\frac{N+2}{N}}\,dx\leq C\left(1+\|  \nabla u\|  _{b,\Omega_i}^{\theta b \frac{N+2}{N}}\right)\quad\text{with}\quad \theta=\frac{\frac{1}{2}- \frac{N}{b(N+2)}}{\frac{N+2}{2N}-\frac{1}{b}} =\frac{N}{N+2}\in (0,1),
\]
whence

\[
\int_{\Omega_i}\vert u\vert ^{\underline{s}_i^-\frac{N+2}{N}}\,dx\leq C\left(1+\int_{\Omega_i}\vert \nabla u\vert ^{\underline{s}_i^-+r_\ast}\,dx\right),\qquad C=C(\underline{s}_i^-,\overline{s}_i^+,N,M_0).
\]
To estimate the last term of \eqref{eq:Q-2-prim} we claim that $\kappa<\underline{s}_i^-+r^\sharp-\varsigma$:

\begin{equation}
\label{eq:equiv-check}
\begin{split}
\underline{s}_i^++r^{\sharp}-\varsigma-1 & <(\underline{s}_i^-+r^{\sharp}-\varsigma)\left(1-\frac{N}{b(N+2)}\right)
\\
&
\Leftrightarrow\quad \underline{s}_i^+- \underline{s}_i^- <1-\frac{N}{N+2}\frac{\underline{s}_i^-+r^{\sharp}-\varsigma}{\underline{s}_i^-+r_\ast}.
\end{split}
\end{equation}
The last inequality is equivalent to

\[
\begin{split}
\underline{s}_i^+-\underline{s}_i^- +( r^\sharp-\varsigma-r_\ast)\frac{N}{b(N+2)}<\frac{2}{N+2},
\end{split}
\]
which is fulfilled due to \eqref{eq:osc-loc} and the choice of $\varsigma$:
\begin{equation}
\label{eq:equiv-cond}
\begin{split}
\underline{s}_i^+-\underline{s}_i^- & +(r^{\sharp}-\varsigma-r_\ast)\frac{N}{b(N+2)}
\\
&
<  \frac{4}{N+2} - r^{\sharp} + \varsigma +(r^{\sharp}-\varsigma-r_\ast)\frac{N}{b(N+2)}
< \frac{2}{N+2}
\\
&
\Leftrightarrow\quad (r^{\sharp}-\varsigma-r_\ast) \left( 1-\frac{N}{b(N+2)} \right) >0.
\end{split}
\end{equation}
Gathering these estimates and applying once again the Young inequality we arrive at the estimate

\[
\begin{split}
\mathcal{Q}_{2,i}^{(r_1-\varsigma, r_2-\varsigma)} & \leq C+\int_{\Omega}\vert \nabla u\vert ^{\underline{s}_i^-+r_\ast}\,dx+\lambda \int_{\Omega_i}\vert \nabla u\vert ^{\underline{s}_i^-+r^{\sharp}-\varsigma}\,dx
\\
&
\leq C+\delta \int_{\Omega_i}\vert \nabla u\vert ^{\underline{s}(x)+r^{\sharp}-\varsigma}\,dx
\end{split}
\]
with any $\delta>0$ and $\varsigma \in (0, r_\ast)$. By collecting the estimates for all $i=1,2,\ldots,K$ and using \eqref{est:transfer}, we obtain the needed inequality.
\end{proof}

\begin{lemma}
\label{le:Q3}
Let the conditions of Lemma \ref{le:Q1} be fulfilled. Then, for every $\varsigma \in (0, r_\ast/2)$ and every $\mu>0$

\[
\begin{split}
\mathcal{Q}_{3}^{(r_1-\varsigma, r_2-\varsigma)} & \leq C+\mu \int_{\Omega}\mathcal{F}_\epsilon^{(r_1-\varsigma, r_2-\varsigma)}(x,\nabla u)\vert \nabla u\vert ^2\,dx
\end{split}
\]
with a constant $C=C(\mu, \omega_p,\omega_q,N,M_0,K,\alpha,\varsigma, p^\pm, q^\pm)$
\end{lemma}

\begin{proof}
Fix a number $\varsigma \in (0, r_\ast/2)$ and take a finite cover $\{\Omega_i\}$ of $\Omega$ such that, instead of \eqref{eq:osc-loc},
\[
\underline{s}_i^+-\underline{s}_i^-<\dfrac{\varsigma}{4}\quad \text{on each $\Omega_i$}.
\]
Set $h=\underline{s}_i^-+r^{\sharp}-(\varsigma + \sigma_{i, \varsigma})>0$, where $\sigma_{i, \varsigma}>0$ is a number to be defined. By Young's inequality, for every $\lambda>0$

\begin{equation}
\label{eq:l-4-3}
\mathcal{Q}_{3,i}^{(r_1-\varsigma, r_2-\varsigma)}\leq C_\lambda\int_{\Omega_i}\vert u\vert ^{h\frac{N+2}{N}}\,dx+\lambda \left( \int_{\Omega_i}\beta_{\epsilon}^{\frac{\kappa_1}{2}}(\nabla u)\,dx + \int_{\Omega_i}\beta_{\epsilon}^{\frac{\kappa_2}{2}}(\nabla u)\,dx \right)
\end{equation}
with
\[
\kappa_1=(p(x)+2(r_1-\varsigma-1))\frac{h\frac{N+2}{2N}}{h\frac{N+2}{2N}-1},  \qquad \kappa_2=(q(x)+2(r_2-\varsigma-1))\frac{h\frac{N+2}{2N}}{h\frac{N+2}{2N}-1}.
\]
The first term on the right-hand side is estimated by the Gagliardo-Nirenberg inequality:

\[
\int_{\Omega_i}\vert u\vert ^{h\frac{N+2}{N}}\,dx\leq C\left(1+\|  \nabla u\|  _{h,\Omega_i}^{\theta h\frac{N+2}{N}}\right),\qquad \theta=\dfrac{\frac{1}{2} -\frac{N}{h(N+2)}}{\frac{N+2}{2N}-\frac{1}{h}}=\frac{N}{N+2}\in (0,1),
\]
with a constant $C$ depending on $\underline{s}_i^+$, $\underline{s}_i^-$, $N$, $M_0$ and $\varsigma$. Thus, for every $\nu>0$

\[
\int_{\Omega_i}\vert u\vert ^{h\frac{N+2}{N}}\,dx\leq C\left(1+\int_{\Omega_i}\vert \nabla u\vert ^{\underline{s}_i^-+r^{\sharp}-(\varsigma + \sigma_{i, \varsigma})}\,dx\right)\leq C'+\nu\int_{\Omega_i}\vert \nabla u\vert ^{\underline{s}(x)+r^{\sharp}-\varsigma}\,dx.
\]
To estimate the second term we claim $0<\kappa_1<\underline{s}(x)+r^{\sharp}-\varsigma$ on $\Omega_i$. The first inequality is satisfied due to the assumption $\underline{s}^- > \frac{2N}{N+2}$, the choice $\varsigma \in (0, r_\ast/2)$ and \eqref{eq:osc-p-q}:

\[
\begin{split}
p(x) & +2(\underline{s}(x)+r^\sharp-p(x)-\varsigma-1)
\geq \underline{s}(x)-p(x)+\frac{8}{N+2}+\frac{2N}{N+2}-2\varsigma-2
\\
&
=
\begin{cases}
\frac{4}{N+2}-2\varsigma & \text{if $p(x)\leq q(x)$}
\\
q(x)-p(x)+\frac{4}{N+2} -2\varsigma & \text{if $p(x)>q(x)$}
\end{cases}
> \begin{cases}
0 & \text{if $p(x)\leq q(x)$}
\\
\frac{4}{N+2}-2r_\ast & \text{otherwise}
\end{cases}
=0.
\end{split}
\]
Since $h>\dfrac{2N}{N+2}$ by definition, $\kappa_1>0$.
To fulfil the second inequality for $\kappa_1$, we claim that the stronger condition holds:
\[
p(x)+2(r_1-\varsigma-1) <(\underline{s}_i^-+r^{\sharp}-\varsigma) \left(1-\dfrac{2N}{h(N+2)}\right).
\]
Due to the choice of $h$ and $r_1$, and because of the inequalities $\underline{s}_i^-\leq \underline{s}(x)\leq p(x)$ on $\Omega_i$, this is true if
\[
\begin{split}
2\underline{s}(x)-p(x) & +r^\sharp-\varsigma\leq 2\underline{s}_i^+- \underline{s}_i^-+\frac{4}{N+2}-\varsigma<2+\underline{s}_i^- -\frac{2N}{N+2} -\frac{2N\sigma_{i,\varsigma}}{h(N+2)}.
\end{split}
\]
The second inequality holds if $\sigma_{i,\varsigma}$ is chosen from the condition
\[
2(\underline{s}_i^+-\underline{s}_i^-)<\varsigma -\frac{2N\sigma_{i,\varsigma}}{(N+2)(\underline{s}_i^- +r^\sharp-\varsigma-\sigma_{i,\varsigma})}\quad \Leftarrow\quad 0<\sigma_{i,\varsigma}<\dfrac{(N+2)(2-\varsigma)\varsigma}{4N+(N+2)\varsigma}.
\]

The third term on the right-hand side of \eqref{eq:l-4-3} is estimated in the same way. Gathering these estimates and applying the Young inequality we conclude that for every $\mu>0$ and $\varsigma \in (0,r_\ast/2)$

\[
\begin{split}
\mathcal{Q}_{3,i}^{(r_1-\varsigma, r_2-\varsigma)} & \leq C
+ \mu \int_{\Omega_i}\mathcal{F}_\epsilon^{(r_1-\varsigma, r_2-\varsigma)}(x,\nabla u)\vert \nabla u\vert ^2\,dx.
\end{split}
\]
The proof is completed by gathering the estimates for all $\Omega_i$.
\end{proof}

\begin{lemma}
\label{le:final-ell}
Let the conditions of Lemma \ref{lem:interpol} and condition  \eqref{eq:osc-p-q} be fulfilled. For every $u\in C^1(\overline{\Omega})\cap H^2_0(\Omega)$, any $\varsigma \in (0,r^\sharp)$, and an arbitrary $\delta>0$

\begin{equation}\label{eq:complete-new}
\int_{\Omega}\mathcal{F}_{\epsilon}^{(r_1-\varsigma, r_2-\varsigma)}(x,\nabla u)\vert \nabla u\vert ^2\,dx\leq \delta \int_{\Omega}\mathcal{F}_{\epsilon}^{(0,0)}(x,\nabla u)\vert u_{xx}\vert ^2\,dx +C
\end{equation}
with a constant $C=C(\omega_p,\omega_q,N,M_0,\alpha,\delta,\varsigma)$. \end{lemma}
\begin{proof} For $\varsigma\in (0,r_\ast/2)$ inequality \eqref{eq:complete} follows from the estimates of Lemmas \ref{le:Q1}, \ref{le:Q2}, \ref{le:Q3} with $\sigma+\lambda+\mu<1$. If $\varsigma \in (r_\ast/2,r^\sharp)$, it is sufficient to observe that by virtue of Young's inequality and \eqref{eq:null-eps}
\[
\begin{split}
  & \mathcal{F}_{\epsilon}^{(r_1-\varsigma, r_2-\varsigma)}(x,\nabla u)\vert \nabla u\vert ^2\,dx \leq \mathcal{F}_{\epsilon}^{(r_1-\varsigma, r_2-\varsigma)}(x,\nabla u) \beta_\epsilon(\nabla u) \\
  &\leq C'+\mathcal{F}_{\epsilon}^{(r_1-\widetilde \varsigma, r_2-\widetilde \varsigma)}(x,\nabla u) \beta_\epsilon(\nabla u)
    \leq C_2 + 2 \mathcal{F}_{\epsilon}^{(r_1-\widetilde\varsigma , r_2-\widetilde \varsigma)}(x,\nabla u)\vert \nabla u\vert ^2\,dx
\end{split}
\]
with $\widetilde \varsigma < \varsigma$, $\widetilde \varsigma \in (0, r_\ast/2)$ and independent of $u$ constants $C_1$, $C_2$.
\end{proof}

The estimate of Lemma \ref{le:final-ell} can be extended to the functions defined on the cylinder $Q_T$, provided that for a.e. every $t\in (0,T)$ the exponents $p,q$ and the coefficients $a,b$ satisfy the conditions of Lemma \ref{lem:interpol} and \eqref{eq:osc-p-q} with $r^\sharp= \frac{4}{N+2}$.

\begin{theorem}
\label{th:integr-par} Let $\partial \Omega \in C^2$, $u\in
C^{0}([0,T];C^1(\overline{\Omega})\cap H^2_0(\Omega))$.  Assume that the exponents $p(z)$, $q(z)$ satisfy conditions \eqref{assum1}, \eqref{eq:Lip-p-q}, \eqref{eq:gap-z}, and the coefficients $a(z)$, $b(z)$ satisfy conditions \eqref{eq:a-b}. If
\begin{equation}
\label{eq:cond-embed-par}
\int_{Q_T}\mathcal{F}^{(0,0)}_{\epsilon}(z,\nabla
u)\vert u_{xx}\vert ^2\,dz<\infty,\qquad \sup_{(0,T)}\|  u(t)\|  _{2,\Omega}^2= M_0,
\end{equation}
then for every $\varsigma \in (0,r^\sharp)$ and every $\beta\in (0,1)$

\begin{equation}
\label{eq:principal2}
\begin{split}
\int_{Q_T} \mathcal{F}_{\epsilon}^{(r_1-\varsigma, r_2-\varsigma)}(z,\nabla u)\vert \nabla u\vert ^2\,dz\leq \beta\int_{Q_T}\mathcal{F}^{(0,0)}_{\epsilon}(z,\nabla
u)\vert u_{xx}\vert ^{2}\,dx+C
\end{split}
\end{equation}
with
$$r_1(z):=\underline{s}(z)+r^\sharp - p(z), \quad \text{and} \quad r_2(z):=\underline{s}(z)+r^\sharp - q(z)$$
and a constant $C=C(\partial
\Omega,\beta,\alpha, \omega_p,\omega_q,N,M_0, L, \varsigma)$, where $\omega_p$, $\omega_q$ denote the modules of continuity of $p$ and $q$ in $Q_T$. Moreover, for every $\varsigma\in (0,r^\sharp)$
\begin{equation}
\label{eq:principal-3}
\alpha\int_{Q_T}\vert \nabla u\vert ^{\underline{s}(z)+r^\sharp-\varsigma}\,dz\leq \beta\int_{Q_T}\mathcal{F}^{(0,0)}_{\epsilon}(x,\nabla
u)\vert u_{xx}\vert ^{2}\,dx+C
\end{equation}
with a constant $C$ depending on the same quantities as the constant in \eqref{eq:principal2}.
\end{theorem}

\begin{proof}
Since the exponents $p,q$ are Lipschitz-continuous in $\overline{Q}_T$, for every $0 < \varsigma < r_\ast/2$ there exists a finite cover of $Q_T$ composed of the cylinders $Q^{(i)}=\Omega_i\times (t_{i-1},t_i)$, $i=1,2,\ldots,K$, such that

\[
\begin{split}
& t_0=0,\quad t_K=T,\quad t_i-t_{i-1}=\rho,\quad Q_T\subset \bigcup_{i=1}^{K}Q^{(i)},\quad \partial\Omega_i\in C^2,
\\
& \underline{s}_i^+=\max_{\overline{Q}^{(i)}} \underline{s}(z),\quad \underline{s}_i^{-}=\min_{\overline{Q}^{(i)}}\underline{s}(z),
\quad \underline{s}_i^+-\underline{s}_i^- <\dfrac{\varsigma}{4},\quad i=1,2,\ldots,K,\;\;K=K(\omega_p,\omega_q,\varsigma).
\end{split}
\]
For a.e. $t\in (0,T)$ the function $u(\cdot,t):\Omega\mapsto \mathbb{R}$  satisfies inequality \eqref{eq:complete-new} on each of $\Omega_i$. Integrating these inequalities over the intervals $(t_{i-1},t_i)$ and summing the results we obtain \eqref{eq:principal2} with $\varsigma<r_\ast/2$. The case $\varsigma\in [r_\ast/2,r^\sharp)$ follows as in the proof of Lemma \ref{le:final-ell}. Due to the choice of $r_1(z)$, $r_2(z)$, inequality \eqref{eq:principal-3} follows from \eqref{eq:principal2} with the help of \eqref{eq:null-eps} and Lemma \ref{le:cont-emb-1}.
\end{proof}

\subsection{Estimates on the traces}
We will need estimates on the traces on the lateral boundary of the cylinder $Q_T$. In the assertions formulated below the exponents $p$, $q$ and coefficients $a$, $b$ are considered as functions of $x\in \Omega$ and the variable $t$ is regarded as a parameter.
\begin{lemma}[Lemma 3.3, \cite{arora_shmarev2020}]
\label{le:trace-old}
Let $\Omega\subset \mathbb{R}^N$, $N\geq 2$ be a bounded domain with the boundary $\partial\Omega\in C^{2}$, and $ {a,b \in W^{1,\infty}(\Omega)}$ be given nonnegative functions. Assume that $v\in H^{3}_0(\Omega)$ and denote

\begin{equation}
\label{eq:K}
\mathcal{K}=\int_{\partial\Omega} \mathcal{F}^{(0,0)}_\epsilon(x, \nabla v) \left(\Delta v \,(\nabla v\cdot \mathbf{n})-\nabla (\nabla v\cdot \mathbf{n})\cdot \nabla v\right)\,dS,
\end{equation}
where $\mathbf{n}$ stands for the exterior normal to ${\partial \Omega}$. There exists a constant $\kappa=\kappa(\partial\Omega)$ such that
\[
\mathcal{K}\leq \kappa\int_{\partial\Omega}\mathcal{F}^{(0,0)}_\epsilon(x, \nabla v) \vert \nabla v\vert ^2\,dS.
\]
\end{lemma}

\begin{theorem}\label{th:trace-main} Let $\partial\Omega\in C^2$, $u\in C^1(\overline{\Omega})\cap H^2_0(\Omega)$. Assume that $a(x)$, $b(x)$, $p(x)$, $q(x)$ satisfy the conditions of Lemma \ref{lem:interpol} and \eqref{eq:osc-p-q}. Then for every $\lambda\in (0,1)$
\begin{equation}
\label{eq:trace-3} \int_{\partial \Omega} \mathcal{F}^{(0,0)}_{\epsilon}(x,\nabla u)\vert \nabla u\vert ^{2}\,dS\leq
\lambda \int_{\Omega} \mathcal{F}^{(0,0)}_{\epsilon}(x,\nabla u) \vert u_{xx}\vert ^2\,dx+ C
\end{equation}
with a constant $C$ depending on $\lambda$,
$\underline{s}^-$, $\underline{s}^+$, $N$, $L$, $\alpha$, and $\|  u\|  _{2,\Omega}$.
\end{theorem}

\begin{proof}
Applying \cite[Corollary 4.1]{arora_shmarev2020}  to $\vert \nabla u\vert $ we obtain
\begin{equation}
\label{eq:trace-4}
\begin{split}
\int_{\partial \Omega}   \mathcal{F}^{(0,0)}_{\epsilon}(x,\nabla u) \vert \nabla u\vert ^2\,dS & \leq \delta
\int_{\Omega}  \mathcal{F}^{(0,0)}_{\epsilon}(x,\nabla u) \vert u_{xx}\vert ^2\,dx
\\
&
+ C' \int_{\Omega} \left(\vert \nabla u\vert ^{p(x)} + \vert \nabla u\vert ^{q(x)}\right)\,dx
\\
& + C'' \int_{\Omega} (\vert \nabla
u\vert ^{p(x)} + \vert \nabla
u\vert ^{q(x)})\vert \ln\vert \nabla u\vert \vert \,dx +C'''
\end{split}
\end{equation}
with independent of $u$ constants $C', C'', C'''$. {Choose $\ell_0 \in (0,1)$ and $\ell_{00} \in [r_\ast, r^{\#})$ such that $$\overline{s}(x)+\ell_0 < \underline{s}(x) + \ell_{00} < \underline{s}(x) + r^{\#}.$$
Using \eqref{eq:log} and Young's inequality we estimate
\[
(\vert \nabla u\vert ^{p(x)} + \vert \nabla u\vert ^{q(x)})\vert \ln \vert \nabla u\vert \vert
\leq C(\ell_0,p^\pm, q^\pm) (1+ \vert \nabla u\vert ^{\underline{s}(x)+ \ell_{00}})
\]
with a constant $C$ independent of $u$}. Using the above inequality and then applying Lemma \ref{le:final-ell} with $\varsigma= r^\sharp-\ell_{00}$ we continue \eqref{eq:trace-4} as follows:
\[
\begin{split}
\int_{\partial \Omega}   \mathcal{F}^{(0,0)}_{\epsilon}(x,\nabla u) \vert \nabla u\vert ^2\,dS & \leq \delta
\int_{\Omega}  \mathcal{F}^{(0,0)}_{\epsilon}(x,\nabla u) \vert u_{xx}\vert ^2\,dx
\\
&
+C \left(1+\alpha \int_{\Omega} \vert \nabla u\vert ^{\underline{s}(x)+ \ell_{00}} \,dx\right)
\\
& \leq (\delta+ \beta) \int_{\Omega}\mathcal{F}^{(0,0)}_{\epsilon}(x,\nabla u)\vert u_{xx}\vert ^2\,dx +C.
\end{split}
\]
\end{proof}

\section{Approximation of the regularized problem}
\label{sec:reg-problem}

Given $\epsilon>0$, let us consider the family of regularized unordered double phase parabolic equations:
\begin{equation}\label{eq:reg-prob}
\left\{
         \begin{alignedat}{2}
             {} \partial_t u - \div(\mathcal{F}^{(0,0)}_{\epsilon}(z,\nabla u)  \nabla u)
             & {}= f(z)
             && \quad\mbox{ in } \, Q_T,
             \\
             u & {}= 0
             && \quad\mbox{ on }\, \Gamma_T,
             \\
             u(\cdot,0)&{}= u_0
             && \quad\  \mbox{in}\, \ \Omega, \ \epsilon \in (0,1),
          \end{alignedat}
     \right.
\end{equation}
where $\mathcal{F}^{(0,0)}_{\epsilon}(z,\xi)= a(z) \beta^{\frac{p(z)-2}{2}}_{\epsilon}(\xi) + b(z) \beta^{\frac{q(z)-2}{2}}_{\epsilon}(\xi)$.
\begin{definition}
\label{def:weak:reg}
A function {$u_\epsilon:Q_T\mapsto \mathbb{R}$} is called {\bf strong solution} of problem \eqref{eq:reg-prob} if

\begin{enumerate}

\item $u_\epsilon \in  \mathbb{W}_{\overline{s}(\cdot)}(Q_T)$,
    $u_{\epsilon t} \in L^2(Q_T)$,
    $\vert \nabla u_\epsilon\vert  \in L^{\infty}(0,T; L^{r(\cdot)}(\Omega))$ with $r(z)= \max\{2,\overline{s}(z)\}$,

\item for every $\phi \in \mathbb{W}_{\overline{s}(\cdot)}(Q_T)$

\begin{equation}
\label{eq:def-reg}
\int_{Q_T} u_{\epsilon t} \phi ~dz + \int_{Q_T}
\mathcal{F}_{\epsilon}^{(0,0)}(z,\nabla u_\epsilon)\nabla u_\epsilon\cdot \nabla \phi~dz =
\int_{Q_T} f \phi \,dz,
\end{equation}
\item for every $\psi \in C_0^1(\Omega)$
\[
\int_{\Omega} (u_\epsilon(x,t)-u_0(x)) \psi  ~dx \to 0\quad\text{as $t \to
0$}.
\]
\end{enumerate}
\end{definition}

\subsection{Dense sets in $W^{1,p(\cdot)}_0(\Omega)$}
Let $\{\phi_i\}$ and $\{\lambda_i\}$ be the eigenfunctions and the corresponding eigenvalues of the Dirichlet problem for the Laplacian:

\begin{equation}
\label{eq:eigen}
(\nabla \phi_i,\nabla \psi)_{2,\Omega}=\lambda_i(\phi_i,\psi)\qquad \forall \psi\in H^{1}_0(\Omega).
\end{equation}
The functions $\phi_i$ form an orthogonal basis of $L^2(\Omega)$ and are mutually orthogonal in $H^1_0(\Omega)$. If $\partial\Omega\in C^k$, $k\geq 1$, then $\phi_i\in C^{\infty}(\Omega)\cap H^{k}(\Omega)$. Let us denote by $H^k_{\mathcal{D}}(\Omega)$ the subspace of the Hilbert space $H^{k}(\Omega)$ composed of the functions $f$ for which

\[
f=0,\quad \Delta f=0,\quad \ldots, \Delta^{[\frac{k-1}{2}]} f=0\quad \text{on $\partial\Omega$},\qquad H^{0}_{\mathcal{D}}(\Omega)=L^2(\Omega).
\]
The relations

\[
[f,g]_{k}=\begin{cases}
(\Delta^{\frac{k}{2}}f,\Delta^{\frac{k}{2}}g)_{2,\Omega} & \text{if $k$ is even},
\\
(\Delta^{\frac{k-1}{2}}f,\Delta^{\frac{k-1}{2}}g)_{H^1(\Omega)} & \text{if $k$ is odd}
\end{cases}
\]
define an equivalent scalar product on ${H}^{k}_{\mathcal{D}}(\Omega)$:
$
[f,g]_k=\displaystyle\sum_{i=1}^\infty \lambda_i^k f_ig_i$,
where $f_i$, $g_i$ are the Fourier coefficients of $f$, $g$ in the basis $\{\phi_i\}$ of $L^2(\Omega)$. The corresponding equivalent norm of $H^k_{\mathcal{D}(\Omega)}$ is defined by $\|  f\|  ^2_{{H}^{k}_{\mathcal{D}}(\Omega)}=[f,f]_k$. Let $f^{(m)}=\displaystyle \sum_{i=1}^{m}f_i\phi_i$ be the partial sum of the Fourier series of $f\in L^{2}(\Omega)$. We will rely on the following known assertions.
\begin{proposition}
\label{pro:series}
Let $\partial\Omega\in C^k$, $k\geq 1$. A function $f$ can be represented by the Fourier series in the system $\{\phi_i\}$,
convergent in the norm of $H^k(\Omega)$, if and only if $f\in H^{k}_{\mathcal{D}}(\Omega)$. If $f\in H^{k}_{\mathcal{D}}(\Omega)$, then the Fourier series is convergent, its sum is bounded by $C\|  f\|  _{H^k(\Omega)}$ with an independent of $f$ constant $C$, and $\|  f^{(m)}-f\|  _{H^k(\Omega)}\to 0$ as $m\to \infty$. If $k\geq [\frac{N}{2}]+1$, then the Fourier series in the system $\{\phi_i\}$ of every function $f\in H^{k}_{\mathcal{D}}(\Omega)$ converges to $f$ in $C^{k-[\frac{N}{2}]-1}(\overline{\Omega})$.
\end{proposition}

\begin{proposition}[\cite{DNR-2012}, Th. 4.7, Proposition 4.10]
\label{pro:density}
Let $\partial \Omega\in Lip$ and $p(x)\in C_{\rm log}(\overline{\Omega})$. Then the set $C^\infty_{0}(\Omega)$ is dense in $W_0^{1,p(\cdot)}(\Omega)$.
\end{proposition}

Let $k\in \mathbb{N}$ be so large that

\begin{equation}
\label{eq:k}
k\geq N\left(\frac{1}{2}+\frac{1}{N}-\frac{1}{q^+}\right), \quad q(x)=\max\{2,p(x)\},\quad q^+=\sup_{\Omega}q(x).
\end{equation}
Set $\mathcal{P}_m=\operatorname{span} \{\phi_1,\ldots,\phi_m\}$

\begin{lemma}
\label{le:density-1}
Assume $k,N,q$ satisfy \eqref{eq:k}. Then the set $\bigcup\limits_{m=1}^\infty \mathcal{P}_m$ is dense in $W^{1,q(\cdot)}_0(\Omega)$.
\end{lemma}

\begin{proof}
Take an arbitrary $u\in W_0^{1,q(\cdot)}(\Omega)$. We want to show that for every $\epsilon>0$ there exists $m(\epsilon)\in \mathbb{N}$ such that for every  $m\geq m(\epsilon)$ there exists $v_m\in \mathcal{P}_m$  satisfying $\|  v_m-u\|  _{W^{1,q(\cdot)}_0(\Omega)}<\epsilon$. By Proposition \ref{pro:density} there exists $v_{\epsilon}\in C_{0}^{\infty}(\Omega)\subset H_{\mathcal{D}}^k(\Omega)$ such that

\[
\|  u-v_\epsilon\|  _{W^{1,q(\cdot)}_0(\Omega)} <\frac{\epsilon}{2},\qquad v_\epsilon\in C_{0}^{\infty}(\Omega)\cap H_{\mathcal{D}}^k(\Omega).
\]
By Proposition \ref{pro:series} $\displaystyle v_\epsilon=\sum_{i=1}^{\infty}v_{\epsilon,i}\phi_i\in H_{\mathcal{D}}^k(\Omega) \subset H^k(\Omega)
$
and
\[
v_{\epsilon}^{(k)}=\sum_{i=1}^k v_{\epsilon,i}\phi_i\in \mathcal{P}_k, \qquad \text{$v_\epsilon^{(k)}\to v_\epsilon$ in $H_{\mathcal{D}}^k(\Omega)$ \ as \ $k \to \infty$}.
\]
Since $k,N,q$ satisfy condition \eqref{eq:k}, the embeddings $H_{\mathcal{D}}^k(\Omega) \subset W_0^{1, q^+}(\Omega) \subseteq W^{1,q(\cdot)}_0(\Omega)$ are continuous:

\[
\|  w\|  _{W^{1,q(\cdot)}_0(\Omega)}\leq
C\|  w\|  _{W_0^{1, q^+}(\Omega)}\leq  C'\|  w\|  _{H^k(\Omega)}\qquad \forall w\in H_{\mathcal{D}}^k(\Omega)
\]
with independent of $w$ constants $C$, $C'$. Given $\epsilon$, we may find $k(\epsilon)\in \mathbb{N}$ such that forall $k\geq k(\epsilon)$
\[
\|  v_\epsilon-v_{\epsilon}^{(k)}\|  _{W_0^{1, q(\cdot)}(\Omega)} \leq C'\|  v_\epsilon-v_{\epsilon}^{(k)}\|  _{H^{k}(\Omega)} =C' \big\|  v_\epsilon-\sum_{i=1}^k v_{\epsilon,i}\phi_i\big\|  _{H^k(\Omega)}<\frac{\epsilon}{2}.
\]
It follows that for every $k\geq k(\epsilon)$

\[
C''\|  u-v_\epsilon^{(k)}\|  _{W^{1,q(\cdot)}_0(\Omega)} <\frac{\epsilon}{2}+ \frac{\epsilon}{2}=\epsilon
\]
with a constant $C''>0$ independent of $u$, $\epsilon$, and $k$.
\end{proof}

\begin{corollary}
Let $q(\cdot)\in C_{{\rm log}}(\overline{Q}_T)$ and $\partial\Omega\in C^k$ with $k$ satisfying condition \eqref{eq:k}. Denote
\begin{equation}
\label{eq:N}
\mathcal{N}_m=\left\{w(x,t):\;w(x,t)= \sum_{i=1}^m \theta_i(t)\phi_i(x),\;\theta_i \in C^{0,1}[0,T]\right\}.
\end{equation}
The set $\bigcup_{m=1}^\infty\mathcal{N}_m$ is dense in $\mathbb{W}_{q(\cdot)}(Q_T)$.
\end{corollary}

\subsection{Dense sets in $W^{1,\mathcal{H}}_0(\Omega)$} Let $W^{1,\mathcal{H}}_0(\Omega)$ be the Musielak-Orlicz space defined in \eqref{eq:M-O} where $p_0$, $q_0$, $a_0$, $b_0$ are the exponents and coefficients from equation \eqref{eq:main} taken at the initial moment $t=0$.

\begin{proposition}
\label{pro:density-2}
Let $\partial\Omega \in Lip$ and $a_0,b_0,p_0,q_0\in C^{0,1}(\overline{\Omega})$. If

\[
\frac{\max\{r^+,\sigma^+\}}{s^-}\leq 1+\frac{1}{N},
\]
then $C^\infty_0(\Omega)\cap W^{1,\mathcal{H}}(\Omega)$ is dense in $W^{1,\mathcal{H}}_0(\Omega)$.
\end{proposition}

The question of density of smooth functions in the Musielak-Sobolev spaces was studied in several works. The assertion of Proposition \ref{pro:density-2} follows from \cite[Th.3.1]{Chlebicka-2019-1} or \cite[Th.6.4.7]{Hasto-Harjulehto-2019-book}. To check the fulfillment of all conditions listed in \cite{Hasto-Harjulehto-2019-book}, one may literally repeat the proof given in \cite[Theorem 2.21]{Blanco-2021} for the special case of the space $W^{1,\mathcal{H}}(\Omega)$ generated by the function $\mathcal{H}$ with $b_0\equiv 0$.

\begin{lemma}
\label{le:density-V}
Let $p,q,a_0,b_0$ satisfy the conditions of Proposition \ref{pro:density-2}. If $\partial\Omega\in C^k$ with

\begin{equation}
\label{eq:k-1}
k\geq N\left(\frac{1}{2}+\frac{1}{N}-\frac{1}{\,\max\{2,p^+,q^+\}}\right),
\end{equation}
then the set $\bigcup\limits_{m=1}^\infty \mathcal{P}_m$ is dense in $W^{1,\mathcal{H}}_0(\Omega)$.
\end{lemma}

\begin{proof}
The assertion follows by imitating the proof of Proposition \ref{pro:density}. Given a function $v\in W^{1,\mathcal{H}}_0(\Omega)$ and an arbitrary $\epsilon>0$, by Proposition \ref{pro:density-2} we may find $v_\epsilon\in C_{0}^\infty(\Omega)\subset C_{0}^{\infty}(\Omega) \cap H^{k}_{\mathcal{D}}(\Omega)$ such that $\|  v-v_\epsilon\|  _{W^{1,\mathcal{H}}}<\frac{\epsilon}{2}$, and then use Proposition \ref{pro:series} to approximate $v_\epsilon\in H^{k}_{\mathcal{D}}(\Omega)$ by the partial sums $v_\epsilon^{(m)}$.
\end{proof}

\subsection{Galerkin's method}
Let $\epsilon>0$ be a fixed parameter. The sequence $\{u^{(m)}_\epsilon\}$ of finite-dimensional Galerkin's approximations for the solutions of the regularized problem \eqref{eq:reg-prob} is sought in the form
 \begin{equation}\label{eq:coeff}
 u^{(m)}_\epsilon(x,t)= \sum_{j=1}^m u_{j}^{(m)}(t) \phi_j(x)
 \end{equation}
 where $\{\phi_j\}$ and $\{\lambda_j\}$ are the eigenfunctions and the corresponding eigenvalues of problem \eqref{eq:eigen}. The coefficients $u_j^{(m)}(t)$ are characterized as the solutions of the Cauchy problem for the system of $m$ ordinary differential equations
\begin{equation}\label{system}
  \left\{
         \begin{alignedat}{2}
             {} (u^{(m)}_j)'(t)
             & {}= - \int_{\Omega} \mathcal{F}^{(0,0)}_{\epsilon}(z, \nabla u_\epsilon^{(m)}) \nabla u^{(m)}_\epsilon \cdot \nabla \phi_j ~dx + \int_{\Omega} f \phi_j ~dx
             && ,
             \\
             u^{(m)}_j(0)
             & {}= (u_0^{(m)}, \phi_j)_{2, \Omega}, \quad j=1,2,\dots,m,
             &&
          \end{alignedat}
     \right.
\end{equation}
where the functions $u_0^{(m)}$ are chosen in such a way that

\begin{equation}
\label{eq:choice}
\begin{split}
& u_0^{(m)}= \sum_{j=1}^m (u_0, \phi_j)_{2, \Omega} \phi_j \in \operatorname{span}\{\phi_1, \phi_2, \dots, \phi_m\},
\\
&
\text{$u_0^{(m)}\to u_0$ in $W^{1,\mathcal{H}}_0(\Omega)$},
\quad \|  u_0^{(m)}\|  _{W^{1,\mathcal{H}}_0}\leq C
\end{split}
\end{equation}
where $C$ is independent of $m$. The existence of such a sequence follows from Lemma \ref{le:density-V}. By the Carath\'{e}odory existence theorem, for every $m$ system \eqref{system} has an absolute continuous solution $u_1^{(m)},\ldots u_m^{(m)}$, defined on an interval $(0,T_m)$. The possibility of continuation of this solution to the whole interval $(0,T)$ will follow from the a priori estimates derived in the next section.

\section{A priori estimates}
\label{sec:a-priori}
Throughout this section, when deriving the estimates for the approximations $u_\epsilon^{(m)}$ we always assume that $p$, $q$, $a$, $b$ and $\partial \Omega$ satisfy the conditions of Lemma \ref{le:density-V}.

\begin{lemma}\label{1st}
Let $\Omega$ be a bounded domain with the Lipschitz boundary. Assume that {$p(\cdot),q(\cdot)$} satisfy \eqref{assum1}, {$a(\cdot), b(\cdot)$} satisfy \eqref{eq:a-b},
$u_0 \in L^2(\Omega)$ and $f \in L^2(Q_T)$. Then $u^{(m)}_\epsilon$ satisfies the estimates
\begin{equation}\label{secderiboun}
\begin{split}
\sup_{t \in (0,T)} \|  u^{(m)}_\epsilon(\cdot,t)\|  ^2_{2,\Omega} & + \int_{Q_T} \mathcal{F}^{(0,0)}_{\epsilon}(z, \nabla u_\epsilon^{(m)}) \vert \nabla u^{(m)}_\epsilon\vert ^2 ~dz
\\
&
\leq C_1 {\rm e}^{ T} (\|  f\|  ^2_{2,Q_T}  + \|  u_0\|  ^2_{2,\Omega})
\end{split}
\end{equation}
and
\begin{equation}\label{gradbound}
\begin{split}
\int_{Q_{T}} & \mathcal{F}_0^{(0,0)}(z,\nabla u_{\epsilon}^{(m)})\vert \nabla u_{\epsilon}^{(m)}\vert ^2\,dz
\\
&
\leq C_2
\int_{Q_T}  \mathcal{F}^{(0,0)}_{\epsilon}(z, \nabla u_\epsilon^{(m)}) \vert \nabla u^{(m)}_\epsilon\vert ^2 ~dz + C_3
\end{split}
\end{equation}
with independent of $m$ and $\epsilon$ constants $C_i.$
\end{lemma}
\begin{proof}
Let us multiply each of equations in \eqref{system} by $u_j^{(m)}(t)$ and sum up the results for $j=1,2,\dots,m$:
\begin{equation}\label{est01}
\begin{split}
\frac{1}{2} & \frac{d}{dt} \|  u^{(m)}_\epsilon{(\cdot, t)}\|  ^2_{2,\Omega}
=\sum_{j=1}^m u_j^{(m)}(t) (u^{(m)}_j)'(t)
\\
&
= - \sum_{j=1}^m  u_j^{(m)}(t) \int_{\Omega} \mathcal{F}^{(0,0)}_{\epsilon}(z, \nabla u_\epsilon^{(m)}) \nabla u^{(m)}_\epsilon . \nabla \phi_j \,~dx
\\
&\quad \quad
+  \sum_{j=1}^m \int_{\Omega} f(x,t) \phi_j(x) u_j^{(m)}(t) ~dx \\
& = - \int_{\Omega} \mathcal{F}^{(0,0)}_{\epsilon}(z, \nabla u_\epsilon^{(m)}) \vert \nabla u^{(m)}_\epsilon\vert ^2 ~dx + \int_{\Omega} f(x,t) u^{(m)}_\epsilon ~dx.
\end{split}
\end{equation}
By employing the Cauchy inequality, we obtain
\begin{equation}
\label{eq:option-1}
\begin{split}
\frac{1}{2} \frac{d}{dt} \|  u^{(m)}_\epsilon {(\cdot, t)}\|  ^2_{2,\Omega}  & + \int_{\Omega} \mathcal{F}^{(0,0)}_{\epsilon}(z, \nabla u_\epsilon^{(m)}) \vert \nabla u^{(m)}_\epsilon\vert ^2 ~dx
\\
&
\leq  \frac{1}{2} \|  f{(\cdot, t)}\|  ^2_{2,\Omega} + \frac{1}{2} \|  u^{(m)}_\epsilon{(\cdot, t)}\|  ^2_{2,\Omega}.
\end{split}
\end{equation}
Now, by rephrasing the inequality in \eqref{eq:option-1} as
\begin{equation*}
\begin{split}
\frac{1}{2} \frac{d}{dt} \left( {\rm e}^{- t}\|  u^{(m)}_\epsilon{(\cdot, t)} \|  ^2_{L^2(\Omega)}\right) &+ {\rm e}^{-t} \int_{\Omega}\mathcal{F}^{(0,0)}_{\epsilon}(z, \nabla u_\epsilon^{(m)}) \vert \nabla u^{(m)}_\epsilon\vert ^2 ~dx \leq  \frac{{\rm e}^{-t}}{2} \|  f{(\cdot, t)}\|  ^2_{2,\Omega}
\end{split}
\end{equation*}
and integrating with respect to $0<t<T_m$, we arrive at the inequality
\begin{equation}\label{est03}
\begin{split}
\sup_{t \in (0,T_m)} \|  u^{(m)}_\epsilon(\cdot,t)\|  ^2_{L^2(\Omega)} & + \int_{Q_{T_m}} \mathcal{F}^{(0,0)}_{\epsilon}(z, \nabla u_\epsilon^{(m)}) \vert \nabla u^{(m)}_\epsilon\vert ^2 ~dx ~dt
\\
&
\leq C {\rm e}^{ T} \left(\|  f\|  ^2_{2,Q_T}  + \|  u_0\|  ^2_{2,\Omega}\right).
\end{split}
\end{equation}
where the constant $C$ is independent of $\epsilon$ and $m$. It follows that $u_{\epsilon}^{(m)}(\cdot,T_m)\in L^2(\Omega)$ and, thus, system \eqref{system} can be solved on an interval $(T_m,T_m+h)$ with some $h>0$. In the result we obtain \eqref{est03} with $T_m$ substituted by $T_m+h$. This process can be continued until we exhaust the interval $(0,T).$ Moreover, \eqref{gradbound} follows from \eqref{eq:null-eps}
and \eqref{secderiboun}.
\end{proof}

\begin{lemma}\label{second}
Let $\Omega$ be a bounded domain with the boundary $\partial \Omega \in C^k$ with $k\geq 2+[\frac{N}{2}]$. Assume that $p(\cdot),q(\cdot)$ satisfy conditions \eqref{assum1}, \eqref{eq:Lip-p-q}, \eqref{eq:gap-z} and
$a(\cdot), b(\cdot)$ satisfy conditions  \eqref{eq:a-b}. If $u_0 \in W^{1,\mathcal{H}}_0(\Omega)$, $f\in L^2(0,T;W_0^{1,2}(\Omega))$, then for a.e. $t\in (0,T)$ the following inequality holds:
\begin{equation}
\label{eq:aux-est}
\begin{split}
\frac{1}{2} \frac{d}{dt} & \|  \nabla
u^{(m)}_\epsilon{(\cdot, t)}\|  ^2_{2,\Omega}  + C_0 \int_{\Omega}
\mathcal{F}^{(0,0)}_{\epsilon}(z, \nabla u_\epsilon^{(m)})
\vert (u^{(m)}_\epsilon)_{xx}\vert ^2 \,dx
\\
& \leq  C_1 \left(1 +
\int_{\Omega}
\mathcal{F}^{(0,0)}_{\epsilon}(z, \nabla u_\epsilon^{(m)})\vert \nabla u_\epsilon^{(m)}\vert ^2\,dx
+\|  f{(\cdot, t)}\|  _{W^{1,2}_0(\Omega)}^2\right)
\end{split}
\end{equation}
where the constants $C_i$ are independent of $\epsilon$ and $m$.
\end{lemma}

\begin{proof}
Multiplying $j^{\text{th}}$ equation of \eqref{system} by $\lambda_j
u_j^{(m)}$ and then summing up the results for $j=1,2, \dots, m$, we obtain the equality
\begin{equation}
\label{eq:prime}
\begin{split}
\frac{1}{2} & \frac{d}{dt} \|  \nabla
u^{(m)}_\epsilon {(\cdot, t)}\|  ^2_{2,\Omega} =\sum_{j=1}^m \lambda_j
(u_j^{(m)})'(t) u_j^{(m)}(t)
\\
& = \sum_{j=1}^m \lambda_j u_j^{(m)} \int_{\Omega}
\div(\mathcal{F}^{(0,0)}_{\epsilon}(z, \nabla u_\epsilon^{(m)}) \nabla u_\epsilon^{(m)})\ \phi_j \,dx
\\
&
\qquad + \sum_{j=1}^m \lambda_j
u_{j}^{(m)} \int_{\Omega} f(x,t) \phi_j \,dx\\ &= - \int_{\Omega}
\div(\mathcal{F}^{(0,0)}_{\epsilon}(z, \nabla u_\epsilon^{(m)}) \nabla u_\epsilon^{(m)}) \ \Delta u^{(m)}_\epsilon \,dx + \int_{\Omega} f(x,t)
\Delta u^{(m)}_\epsilon \,dx.
\end{split}
\end{equation}
The assumption $\partial\Omega\in C^{k}$ with $k\geq 2+[\frac{N}{2}]$ yields fulfillment of condition \eqref{eq:k-1} and the inclusion $u_\epsilon^{(m)}(\cdot,t)\in C^{1}(\overline{\Omega})\cap H^k_{\mathcal{D}}(\Omega)$ for every fixed $t\in [0,T]$. It follows that $\left(\mathcal{F}^{(0,0)}_{\epsilon}(z, \nabla u_\epsilon^{(m)}) (
u^{(m)}_\epsilon)_{x_i}\right)_{x_i} \in L^{2}(\Omega)$, therefore we may transform the first term on the right-hand side of \eqref{eq:prime} using the Green formula two times (see \cite[Lemma 5.2]{arora_shmarev2020} or \cite[Lemma 3.2]{A-S} for the details):
\[
\begin{split}
- \int_{\Omega}  &   \mathrm{div}\left( \mathcal{F}^{(0,0)}_{\epsilon}(z, \nabla u_\epsilon^{(m)}) \nabla u_\epsilon^{(m)}\right) \,\Delta u_\epsilon^{(m)}\,dx
\\
&
= - \int_{\Omega} \mathcal{F}^{(0,0)}_{\epsilon}(z, \nabla u_\epsilon^{(m)}) \vert (u_{\epsilon}^{(m)})_{xx}\vert ^2\,dx +P_{1}+P_2+P_{\partial
\Omega} + P_{a,b},
\end{split}
\]
where $\mathbf{n}=(n_1,\ldots,n_N)$ is the outer normal vector to $\partial\Omega$. Here

\begin{equation*}
\begin{split}
P_1:&= \int_{\Omega} (2-p(z)) a(z) (\epsilon^2 + \vert \nabla
u^{(m)}_\epsilon\vert ^2)^{\frac{p(z)-2}{2}-1} \left(\sum_{k=1}^N
\left(\nabla u^{(m)}_\epsilon \cdot \nabla
(u^{(m)}_\epsilon)_{x_k} \right)^2 \right)\,dx\\
&  + \int_{\Omega}  (2-q(z)) b(z) (\epsilon^2 + \vert \nabla
u^{(m)}_\epsilon\vert ^2)^{\frac{q(z)-2}{2}-1} \left(\sum_{k=1}^N
\left(\nabla u^{(m)}_\epsilon \cdot \nabla
(u^{(m)}_\epsilon)_{x_k} \right)^2 \right)\,dx,
\end{split}
\end{equation*}
\begin{equation*}
\begin{split}
P_2&= - \int_{\Omega} \sum_{k,i=1}^N (u^{(m)}_\epsilon)_{x_k x_i }
(u^{(m)}_\epsilon)_{x_i} a(z)  (\epsilon^2 + \vert \nabla
u^{(m)}_\epsilon\vert ^2)^{\frac{p(z)-2}{2}} \frac{p_{x_k}}{2}
\ln(\epsilon^2 + \vert \nabla u^{(m)}_\epsilon\vert ^2)\,dx\\
& \quad \quad - \int_{\Omega} \sum_{k,i=1}^N (u^{(m)}_\epsilon)_{x_k x_i }
(u^{(m)}_\epsilon)_{x_i}  b(z) (\epsilon^2 + \vert \nabla
u^{(m)}_\epsilon\vert ^2)^{\frac{q(z)-2}{2}} \frac{q_{x_k}}{2}
\ln(\epsilon^2 + \vert \nabla u^{(m)}_\epsilon\vert ^2)\,dx,
\end{split}
\end{equation*}
\begin{equation*}
\begin{split}
P_{\partial \Omega}= - \int_{\partial\Omega} \mathcal{F}^{(0,0)}_{\epsilon}(z, \nabla u_\epsilon^{(m)})  \left(\Delta
u^{(m)}_\epsilon (\nabla u^{(m)}_\epsilon \cdot {\bf n})- \nabla
u^{(m)}_\epsilon \cdot \nabla (\nabla u^{(m)}_\epsilon \cdot{\bf n
})\right)\,dS,
\end{split}
\end{equation*}
\begin{equation*}
\begin{split}
P_{a,b}&= - \int_{\Omega} \sum_{i,k=1}^N a_{x_k} (u^{(m)}_\epsilon)_{x_i} (\epsilon^2 + \vert \nabla u^{(m)}_\epsilon\vert ^2)^{\frac{p(z)-2}{2}} {(u_\epsilon^{(m)})_{x_k x_i}}\,dx \\
& \quad - \int_{\Omega} \sum_{i,k=1}^N b_{x_k} (u^{(m)}_\epsilon)_{x_i} (\epsilon^2 + \vert \nabla u^{(m)}_\epsilon\vert ^2)^{\frac{q(z)-2}{2}} {(u_\epsilon^{(m)})_{x_k x_i}}\,dx.
\end{split}
\end{equation*}
By substituting the these equalities into \eqref{eq:prime}, we obtain the following inequality
\begin{equation}\label{mainest}
\begin{split}
\frac{1}{2} & \frac{d}{dt} \|  \nabla
u^{(m)}_\epsilon{(\cdot, t)}\|  ^2_{2,\Omega} + \int_{\Omega}  \mathcal{F}^{(0,0)}_{\epsilon}(z, \nabla u_\epsilon^{(m)})
\vert (u^{(m)}_\epsilon)_{xx}\vert ^2\,dx
\\
& = P_{1}+P_2+P_{\partial
\Omega} + P_{a,b} - \int_{\Omega}\nabla f \cdot
\nabla u_{\epsilon}^{(m)}\,dx
\\
& \leq P_{1}+P_2+P_{\partial
\Omega} + P_{a,b} +\frac{1}{2}\|  \nabla
u_{\epsilon}^{(m)}{(\cdot, t)}\|  _{2,\Omega}^{2}+\frac{1}{2}\|  f (\cdot, t)\|  _{W^{1,2}_{0}(\Omega)}^{2}.
\end{split}
\end{equation}
The terms on the right-hand side of \eqref{mainest} are estimated in four steps.

\medskip

\textbf{Step 1: estimate on $P_1$}. Since $a(\cdot)$ and $b(\cdot)$ are non-negative functions, the term $P_1$ can be merged into the left-hand side of \eqref{mainest}. Indeed:
\[
\begin{split}
P_1 & = \int_{\{ x\in \Omega :\ p(z) \geq 2\}}(2-p(z))\ldots + \int_{\{ x\in \Omega :\  p(z) < 2\}}(2-p(z))\ldots \\
&\qquad + \int_{\{ x\in \Omega :\ q(z) \geq 2\}}(2-q(z))\ldots + \int_{\{ x\in \Omega :\  q(z) < 2\}}(2-q(z))\ldots
\\
&
\leq \int_{\{p(z) < 2\}}
(2-p(z)) a(z) (\epsilon^2 + \vert \nabla
u^{(m)}_\epsilon\vert ^2)^{\frac{p(z)-2}{2}-1} \left(\sum_{k=1}^N
\left(\nabla u^{(m)}_\epsilon \cdot \nabla
(u^{(m)}_\epsilon)_{x_k} \right)^2 \right)\,dx\\
& + \int_{\{q(z) < 2\}}
(2-q(z)) b(z) (\epsilon^2 + \vert \nabla
u^{(m)}_\epsilon\vert ^2)^{\frac{q(z)-2}{2}-1} \left(\sum_{k=1}^N
\left(\nabla u^{(m)}_\epsilon \cdot \nabla
(u^{(m)}_\epsilon)_{x_k} \right)^2 \right)\,dx,
\end{split}
\]
whence
\[
\begin{split}
\vert P_1\vert  & \leq \max\{0,2-p^-\}\int_{\Omega} a(z) (\epsilon^2+\vert \nabla u_\epsilon^{(m)}\vert ^2)^{\frac{p(z)-2}{2}}
\vert (u^{(m)}_\epsilon)_{xx}\vert ^2\,dx
\\
&
\qquad + \max\{0,2-q^-\}\int_{\Omega} b(z)(\epsilon^2+\vert \nabla u_\epsilon^{(m)}\vert ^2)^{\frac{q(z)-2}{2}}
\vert (u^{(m)}_\epsilon)_{xx}\vert ^2\,dx.
\end{split}
\]

\medskip

\textbf{Step 2: estimate on $P_2$}. By the Cauchy inequality, for every $\beta_0>0$
\begin{equation}\label{est:J2:new}
\begin{split}
\vert P_2\vert  &\leq \frac{1}{2} \|  \nabla p \|  _{\infty,\Omega} \int_{\Omega} \left((a(z))^{\frac{1}{2}} (\epsilon^2 + \vert \nabla u^{(m)}_\epsilon\vert ^2)^{\frac{p(z)-2}{4}} \sum_{k,i=1}^N \vert (u^{(m)}_\epsilon)_{x_k x_i }\vert  \right) \\
& \quad \times \left( (a(z))^{\frac{1}{2}}\vert (u^{(m)}_\epsilon)_{x_i}\vert    \vert \ln(\epsilon^2 + \vert \nabla u^{(m)}_\epsilon\vert ^2)\vert   (\epsilon^2 + \vert \nabla u^{(m)}_\epsilon\vert ^2)^{\frac{p(z)-2}{4}} \right) ~dx\\
&\quad + \frac{1}{2} \|  \nabla q \|  _{\infty,\Omega} \int_{\Omega} \left( (b(z))^{\frac{1}{2}} (\epsilon^2 + \vert \nabla u^{(m)}_\epsilon\vert ^2)^{\frac{q(z)-2}{4}} \sum_{k,i=1}^N \vert (u^{(m)}_\epsilon)_{x_k x_i }\vert  \right) \\
& \quad  \times \left( (b(z))^{\frac{1}{2}} \vert (u^{(m)}_\epsilon)_{x_i}\vert    \vert \ln(\epsilon^2 + \vert \nabla u^{(m)}_\epsilon\vert ^2)\vert   (\epsilon^2 + \vert \nabla u^{(m)}_\epsilon\vert ^2)^{\frac{q(z)-2}{4}} \right) ~dx\\
& \leq \beta_0 \int_{\Omega} \mathcal{F}^{(0,0)}_{\epsilon}(z, \nabla u_\epsilon^{(m)})  \sum_{k,i=1}^N \vert (u^{(m)}_\epsilon)_{x_k x_i }\vert ^2 ~dx + \mathcal{L}
\end{split}
\end{equation}
where
\[
\mathcal{L}:= C_1 \int_{\Omega} \ln^2(\epsilon^2 + \vert \nabla u^{(m)}_\epsilon\vert ^2)  \mathcal{F}^{(0,0)}_{\epsilon}(z,\nabla
u^{(m)}_\epsilon) \vert \nabla u^{(m)}_\epsilon\vert ^2 \,dx, \ \quad C_1=C_1(a^+,b^+,N, \beta_0).
\]
Now we apply \eqref{eq:log} in the following form: for $\vartheta \in (0,1)$ and $y>0$

\begin{equation}\label{ineqq}
y^{\frac{p}{2}} \ln^2 y \leq \left\{
         \begin{alignedat}{2}
             {} y^{\frac{p + \vartheta}{2}} (y^{\frac{-\vartheta}{2}} \ln^2(y)) \leq C(\vartheta,p^+) (y^{\frac{p + \vartheta}{2}})
             & {}
             && \quad\mbox{if}\ y \geq 1,
             \\
             y^{\frac{p^-}{2}} \ln^2(y) \leq C(p^-) \quad \quad \quad \quad  \quad \quad \quad & {}
             && \quad\mbox{if}\ y \in (0,1).
          \end{alignedat}
     \right.
\end{equation}
To estimate $\mathcal{L}$ we fix $t\in (0,T)$, take some $\varsigma \in (0,r^\sharp)$, apply \eqref{ineqq} with  $\vartheta=r_1-\varsigma$ to the first term and $\vartheta= r_2-\varsigma$ to the second term with $r_1$, $r_2$ defined in \eqref{eq:r-i}, and then use the interpolation inequality \eqref{eq:complete-new}:

\begin{equation}\label{est:inter:posit}
\begin{split}
\mathcal{L} & \leq C\int_\Omega \mathcal{F}_\epsilon^{(\underline{s}(z)+r^\sharp-\varsigma,\ \underline{s}(z)+r^\sharp-\varsigma)}(z,\nabla u_\epsilon^{(m)})\vert \nabla u_\epsilon^{(m)}\vert ^2\,dx
\\
&
\leq \beta_1 \int_\Omega \mathcal{F}^{(0,0)}_\epsilon(z,\nabla u_\epsilon^{(m)})\vert (u_\epsilon^{(m)})_{xx}\vert ^2\,dx +C
\end{split}
\end{equation}
with any $\beta_1 \in (0,1)$ and a constant $C=C(C_1,\vartheta, \beta_1)$. Gathering \eqref{est:J2:new} and \eqref{est:inter:posit}, we finally obtain:
\[
\vert P_2\vert  \leq (\beta_0 + \beta_1)\int_{\Omega}\mathcal{F}^{(0,0)}_{\epsilon}(x,\nabla
u_\epsilon^{(m)}) \vert (u_\epsilon^{(m)})_{xx}\vert ^{2}\,dx+ C
\]
with a constant $C$ depending on $\beta_0$, $\beta_1$, $r_\ast$, $r^{\#}$, $\overline{s}^+$, $\underline{s}_-$ and $\|  f\|_{2,Q_T}$, $\|  u_0\|_{2,\Omega}$, but independent of $\epsilon$ and $m$.

\medskip

\textbf{Step 3: estimates on $P_{a,b}$.}

\[
\begin{split}
\vert P_{a,b}\vert  & \leq C\int_{\Omega}\left((\epsilon^2+\vert \nabla u_\epsilon^{(m)}\vert ^2)^{\frac{p-1}{2}}\vert (u_{\epsilon}^{(m)})_{xx}\vert +(\epsilon^2+\vert \nabla u_\epsilon^{(m)}\vert ^2)^{\frac{q-1}{2}}\vert (u_{\epsilon}^{(m)})_{xx}\vert \right)\,dx
\\
&
\equiv C(\mathcal{I}_1+\mathcal{I}_2)
\end{split}
\]
with the constant $C=\sup_{\Omega}\vert \nabla a\vert +\sup_{\Omega}\vert \nabla b\vert $. Since $\alpha\leq a+b$ by assumption,  $\sqrt{\alpha}\leq \sqrt{a+b}\leq \sqrt{a}+\sqrt{b}$.

\[
\sqrt{\alpha} \mathcal{I}_1 \leq \int_{\Omega}(\sqrt{a}+\sqrt{b})(\epsilon^2+\vert \nabla u_\epsilon^{(m)}\vert ^2)^{\frac{p-1}{2}}\vert (u_{\epsilon}^{(m)})_{xx}\vert \,dx\equiv \mathcal{K}^{(1)}_a+\mathcal{K}^{(1)}_b.
\]
Applying the Cauchy inequality we estimate
\[
\begin{split}
\mathcal{K}^{(1)}_a & =\int_{\Omega}\sqrt{a}(\epsilon^2+\vert \nabla u_\epsilon^{(m)}\vert ^2)^{\frac{p-1}{2}}\vert (u_{\epsilon}^{(m)})_{xx}\vert \,dx
\\
&
\equiv \int_{\Omega}\left(a(\epsilon^2+\vert \nabla u_\epsilon^{(m)}\vert ^2)^{\frac{p-2}{2}} \vert (u_{\epsilon}^{(m)})_{xx}\vert ^2\right)^{\frac{1}{2}} (\epsilon^2+\vert \nabla u_\epsilon^{(m)}\vert ^2)^{\frac{p}{4}}\,dx
\\
& \leq \beta_2 \int_{\Omega} a(\epsilon^2+\vert \nabla u_\epsilon^{(m)}\vert ^2)^{\frac{p-2}{2}} \vert (u_{\epsilon}^{(m)})_{xx}\vert ^2\,dx + C\int_{\Omega}(\epsilon^2+\vert \nabla u_\epsilon^{(m)}\vert ^2)^{\frac{p}{2}}\,dx,
\end{split}
\]
with an arbitrary $\beta_2>0$ and $C=C(\beta_2)$. By the same token

\[
\begin{split}
\mathcal{K}^{(1)}_b & =\int_{\Omega}\sqrt{b}(\epsilon^2+\vert \nabla u_\epsilon^{(m)}\vert ^2)^{\frac{p-1}{2}}\vert (u_{\epsilon}^{(m)})_{xx}\vert \,dx
\\
&
\equiv \int_{\Omega}\left(b(\epsilon^2+\vert \nabla u_\epsilon^{(m)}\vert ^2)^{\frac{q-2}{2}} \vert (u_{\epsilon}^{(m)})_{xx}\vert ^2\right)^{\frac{1}{2}} (\epsilon^2+\vert \nabla u_\epsilon^{(m)}\vert ^2)^{\frac{p-1}{2}-\frac{q-2}{4}}\,dx
\\
& \leq \beta_2 \int_{\Omega} b(\epsilon^2+\vert \nabla u_\epsilon^{(m)}\vert ^2)^{\frac{q-2}{2}} \vert (u_{\epsilon}^{(m)})_{xx}\vert ^2\,dx + C\int_{\Omega}(\epsilon^2+\vert \nabla u_\epsilon^{(m)}\vert ^2)^{\frac{2p-q}{2}}\,dx,
\end{split}
\]
with an arbitrary $\beta_2>0$. The last integrals in the estimates for $\mathcal{K}^{(1)}_a$ and $\mathcal{K}^{(1)}_b$ are estimate by virtue of Lemma \ref{le:final-ell}, provided that

\begin{equation}
\label{eq:rho}
p(z)< \underline{s}(z)+\rho, \quad 2p(z)-q(z)<\underline{s}(z)+\rho
\end{equation}
for some $\rho\in (r_\ast,r^\sharp)$, which is true by virtue of \eqref{eq:gap-z}. 
Indeed: it is sufficient to claim that

\[
\begin{split}
& p(z)\leq \overline{s}(z)<\underline{s}(z)+r^\sharp \quad \Leftrightarrow \qquad \overline{s}(z)-\underline{s}(z)<r^\sharp,
\\
& 2p(z)-q(z)\leq 2\overline{s}(z)-\underline{s}(z)<\underline{s}(z)+r^\sharp \qquad \Leftrightarrow \qquad 2(\overline{s}(z)-\underline{s}(z))<r^\sharp=2r_\ast,
\end{split}
\]
Since $\rho \in (r_\ast,r^\sharp)$, there exists $\varsigma  \in (0,r_\ast)$ such that $\rho = r^\sharp-\varsigma$ and  \eqref{eq:rho} holds. Splitting $\Omega$ into two parts,

\[
\Omega_-=\{x\in: \vert \nabla u_\epsilon^{(m)}\vert <\epsilon\},\qquad \Omega^+=\Omega\setminus \Omega_-,
\]
we estimate the last integral in the estimate for $\mathcal{K}^{(1)}_a$ with the use of Lemma \ref{le:final-ell} and \eqref{est:transfer}:

\[
\begin{split}
\int_{\Omega}(\epsilon^2+\vert \nabla u_\epsilon^{(m)}\vert ^{2})^{\frac{p}{2}}\,dx & \leq C+ \int_{\Omega}(\epsilon^2+\vert \nabla u_\epsilon^{(m)}\vert ^{2})^{\frac{\underline{s}+r^\sharp-\varsigma}{2}}\,dx\\
& \leq C+\frac{2}{\alpha}\int_{\Omega}\mathcal{F}_{\epsilon}^{(r_1-\varsigma, r_2-\varsigma)}(z,\nabla u_{\epsilon}^{(m)})\vert \nabla u_{\epsilon}^{(m)}\vert ^2\,dx
\\
& \leq C'+\beta_3\int_{\Omega}\mathcal{F}_{\epsilon}^{(0,0)}(z,\nabla u_{\epsilon}^{(m)})\vert (u_{\epsilon}^{(m)})_{xx}\vert ^2\,dx
\end{split}
\]
with an arbitrary $\beta_3>0$ and $C$ independent of $m$ and $\epsilon$. The corresponding term in the estimate on $\mathcal{K}^{(1)}_b$ is estimated similarly because

\[
\int_{\Omega}(\epsilon^2+\vert \nabla u_\epsilon^{(m)}\vert ^{2})^{\frac{2p-q}{2}}\,dx \leq C+ \int_{\Omega}(\epsilon^2+\vert \nabla u_\epsilon^{(m)}\vert ^{2})^{\frac{\underline{s}+r^\sharp-\varsigma}{2}}\,dx.
\]
Estimating of $\mathcal{I}_2$ is practically identical to estimating $\mathcal{I}_1$, the only difference is that the exponents $p$ and $q$ should be replaced. Gathering the above estimates we conclude that for every $\beta_{3}\in (0,1)$

\[
\vert P_{a,b}\vert \leq \beta_{3}\int_{\Omega}\mathcal{F}_{\epsilon}^{(0,0)}(z,\nabla u_\epsilon^{(m)})\vert (u_{\epsilon}^{(m)})_{xx}\vert ^2\,dx +C
\]
with a constant $C=C(\beta_3)$ independent of $m$ and $\epsilon$.

\

\textbf{Step 4: estimate on $P_{\partial \Omega}$.}
To estimate $P_{\partial \Omega}$ we use Lemma \ref{le:trace-old} and Theorem \ref{th:trace-main}:
\begin{equation*}
\begin{split}
\vert P_{\partial \Omega}\vert  & \leq  \left\vert \int_{\partial\Omega} \mathcal{F}^{(0,0)}_{\epsilon}(z,\nabla u_{\epsilon}^{(m)})\left(\Delta
u^{(m)}_\epsilon (\nabla u^{(m)}_\epsilon \cdot {\bf n})- \nabla
u^{(m)}_\epsilon \cdot \nabla (\nabla u^{(m)}_\epsilon \cdot{\bf n
})\right)\,dS \right\vert \\
& \leq C \int_{\partial \Omega} \mathcal{F}^{(0,0)}_{\epsilon}(z,\nabla u_{\epsilon}^{(m)})\vert \nabla u_\epsilon^{(m)}\vert ^2 \,dS\\
& \leq \beta_4 \int_{\Omega} \mathcal{F}^{(0,0)}_{\epsilon}(z, \nabla u_\epsilon^{(m)}) \vert (u_\epsilon^{(m)})_{xx}\vert ^2 \,dx
\\
&\qquad
+  C\left(1 +\int_{\Omega} \mathcal{F}_0^{(0,0)}(z,\nabla u_\epsilon^{(m)})\vert \nabla u_\epsilon^{(m)}\vert ^2\,dx \right)
\end{split}
\end{equation*}
with an arbitrary $\beta_4 \in (0,1)$ and $C$ depending upon $\beta_4$, $p$, $q$, $a$, $b$, $\partial \Omega$ and their differential properties, but not on $\epsilon$ and $m$. To complete the proof and obtain \eqref{eq:aux-est}, we gather the estimates of $P_1$, $P_2$, $P_{a,b}$, $P_{\partial \Omega}$ and choose $\beta_i$ so small that
\[
\min\{1,s^--1\}-\sum_{i=0}^4\beta_i>0.
\]
\end{proof}

\begin{lemma}
  \label{le:est-1}
Under the conditions of Lemma \ref{second}
  \begin{equation}
    \label{eq:ineq-0}
    \begin{split}
\sup_{(0,T)}\|  \nabla
u^{(m)}_\epsilon{(\cdot, t)}\|  ^2_{2,\Omega} & + \int_{Q_T}
\mathcal{F}^{(0,0)}_{\epsilon}(z, \nabla u_\epsilon^{(m)})
\vert (u^{(m)}_\epsilon)_{xx}\vert ^2 \,dz
\\
&
\leq  C 
\left(1+\|  \nabla u_{0}\|  _{2,\Omega}^{2}+
\|  f\|  _{L^{2}(0,T;W^{1,2}_0(\Omega))}^2\right)
\end{split}
\end{equation}
and for any $\varsigma \in (0, r^{\#})$,
\begin{equation}
\label{eq:ineq-high}
\int_{Q_T}\vert \nabla u_{\epsilon}^{(m)}\vert ^{\underline{s}(z)+r^\sharp-\varsigma}\,dz\leq C'',
\end{equation}
with independent of $m$ and $\epsilon$ constants $C$, $C'$, $C''$.
\end{lemma}

\begin{proof}
Let us integrate \eqref{eq:aux-est} over the interval $(0,\tau)$ with $\tau\leq T$ and apply \eqref{secderiboun} to estimate the right-hand side:
\[
\begin{split}
\|  \nabla
u^{(m)}_\epsilon{(\cdot, \tau)}\|  ^2_{2,\Omega} &  + C_0 \int_{Q_\tau}
\mathcal{F}^{(0,0)}_{\epsilon}(z, \nabla u_\epsilon^{(m)})
\vert (u^{(m)}_\epsilon)_{xx}\vert ^2 \,dz
\\
& \leq C\left(1+\|  u_0\|  ^2_{W^{1,2}_0(\Omega)}+\|  f\|  ^2_{L^{2}(0,\tau;W^{1,2}_0(\Omega))}\right).
\end{split}
\]
Since $\tau\in (0,T)$ is arbitrary, inequality \eqref{eq:ineq-0} follows.

Estimate \eqref{eq:ineq-high} is an immediate consequence of Theorem \ref{th:integr-par} and inequality \eqref{eq:null-eps}: for every $\varsigma \in (0,r_\ast)$ and $r_1(z)$, $r_2(z)$ from the conditions of Theorem \ref{th:integr-par}

\[
\begin{split}
\int_{Q_T}\mathcal{F}_{0}^{(r_1-\varsigma, r_2-\varsigma)}(z,\nabla u_{\epsilon}^{(m)})\vert \nabla u_{\epsilon}^{(m)}\vert ^2\,dz & \leq C+\int_{Q_T}\mathcal{F}_{\epsilon}^{(r_1-\varsigma, r_2-\varsigma)}(z,\nabla u_{\epsilon}^{(m)})\vert \nabla u_{\epsilon}^{(m)}\vert ^2\,dz
\\
& \leq \delta \int_{Q_T}
\mathcal{F}^{(0,0)}_{\epsilon}(z, \nabla u_\epsilon^{(m)})
\vert (u^{(m)}_\epsilon)_{xx}\vert ^2 \,dz +C.
\end{split}
\]
The case $\varsigma \in [r_\ast, r^\sharp)$, is reduced to the considered one as in Lemma \ref{le:final-ell}.
\end{proof}

\begin{corollary}
Under the conditions of Lemma \ref{le:est-1}
\begin{equation}
\label{eq:ineq-high-1}
\begin{split}
& \|  (\epsilon^2+\vert \nabla u_\epsilon^{(m)}\vert ^2)^{\frac{p(z)-2}{2}}\nabla u_\epsilon^{(m)}\|  _{\overline{s}(\cdot))^{'},Q_T}
\\
&\qquad
+ \|  (\epsilon^2+\vert \nabla u_\epsilon^{(m)}\vert ^2)^{\frac{q(z)-2}{2}}\nabla u_\epsilon^{(m)}\|  _{(\overline{s}(\cdot))^{'},Q_T}\leq C
\end{split}
\end{equation}
with a constant $C$ independent of $m$ and $\epsilon$.
\end{corollary}

\begin{proof}
Condition \eqref{eq:gap-z} entails the inequality
\[
\max\left\{\frac{\overline{s}(z)(p(z)-1)}{\overline{s}(z)-1}, \frac{\overline{s}(z)(q(z)-1)}{\overline{s}(z)-1}\right\}\leq \overline{s}(z)\leq \underline{s}(z)+\varsigma \quad \text{for some} \ \varsigma \in (0, r^\sharp).
\]
By Young's inequality, the assertion follows then from \eqref{eq:ineq-high}, \eqref{eq:interchange} and \eqref{eq:null-eps}:
\[
\begin{split}
\int_{Q_T}(\epsilon^2+\vert \nabla u_\epsilon^{(m)}\vert ^2)^{\frac{\overline{s}(z)(p(z)-1)}{2(\overline{s}(z)-1)}}\,dz & + \int_{Q_T}(\epsilon^2+\vert \nabla u_\epsilon^{(m)}\vert ^2)^{\frac{\overline{s}(z)(q(z)-1)}{2(\overline{s}(z)-1})}\,dz
\\
&
\leq C\left(1+\int_{Q_T}\vert \nabla u_\epsilon^{(m)}\vert ^{\underline{s}(z)+\varsigma}\,dz\right)\leq C.
\end{split}
\]
\end{proof}

\begin{lemma}\label{le:time-der}
Assume the conditions of Lemma
\ref{second}. If \eqref{eq:choice} holds, then
\begin{equation}
\label{timederiest}
\begin{split}
& \|  (u^{(m)}_\epsilon)_t\|  ^2_{2,Q_T}
\\
&+\sup_{(0,T)}\int_{\Omega}
\left(a(z)(\epsilon^2+\vert \nabla u_{\epsilon}^{(m)}\vert ^2)^{\frac{p(z)}{2}} + b(z)(\epsilon^2+\vert \nabla u_{\epsilon}^{(m)}\vert ^2)^{\frac{q(z)}{2}}\right)\,dx \\
& \leq C\left(1+
\int_\Omega \mathcal{F}_0^{(0,0)}((x,0),\nabla u_0)\vert \nabla u_0\vert ^2\,dx\right)
+\|  f\|  _{2,Q_T}^2
\end{split}
\end{equation}
with an independent of $m$ and $\epsilon$ constant $C$.
\end{lemma}
\begin{proof}
Multiplying \eqref{system} by $(u^{(m)}_j)_t$ and summing over
$j=1,2,\dots,m$ we obtain the equality
\begin{equation}\label{est013}
\int_{\Omega} (u^{(m)}_\epsilon)^2_t \,dx + \int_{\Omega}
\mathcal{F}^{(0,0)}_{\epsilon}(z, \nabla u_\epsilon^{(m)})
\nabla u^{(m)}_\epsilon \cdot \nabla (u^{(m)}_\epsilon)_t \,dx =
\int_{\Omega} f (u^{(m)}_\epsilon)_t \,dx.
\end{equation}
The straighforward computation leads to the equality
\[
\begin{split}
& \mathcal{F}^{(0,0)}_{\epsilon}(z, \nabla u_\epsilon^{(m)})
\nabla u^{(m)}_\epsilon \cdot \nabla (u^{(m)}_\epsilon)_t
\\
&
=
\frac{d}{dt} \left( \frac{ a(\epsilon^2 + \vert \nabla u^{(m)}_\epsilon
\vert ^2)^{\frac{p}{2}}}{p} + \frac{ b(\epsilon^2 + \vert \nabla u^{(m)}_\epsilon
\vert ^2)^{\frac{q}{2}}}{q}\right)
\\ & + \frac{a p_t (\epsilon^2 + \vert \nabla u^{(m)}_\epsilon
\vert ^2)^{\frac{p}{2}}}{{p^2}} \left(1- \frac{p}{2}
\ln((\epsilon^2 + \vert \nabla u^{(m)}_\epsilon \vert ^2))\right) -  \frac{ a_t (\epsilon^2 + \vert \nabla u_\epsilon^{(m)}\vert ^2)^\frac{p}{2}}{p}\\
& + \frac{b q_t (\epsilon^2 + \vert \nabla u^{(m)}_\epsilon
\vert ^2)^{\frac{q}{2}}}{{q^2}} \left(1- \frac{q}{2}
\ln((\epsilon^2 + \vert \nabla u^{(m)}_\epsilon \vert ^2))\right) -  \frac{ b_t (\epsilon^2 + \vert \nabla u_\epsilon^{(m)}\vert ^2)^\frac{q}{2}}{q}.
\end{split}
\]
Using this equality we rewrite \eqref{est013} as

\begin{equation}\label{est014}
\begin{split}
& \|  (u^{(m)}_\epsilon)_t(\cdot,t)\|  ^2_{2,\Omega} + \frac{d}{dt}\int_{\Omega} \left(\frac{a(z)(\epsilon^2 + \vert \nabla u^{(m)}_\epsilon \vert ^2)^{\frac{p(z)}{2}}}{p(z)} + \frac{b(z) (\epsilon^2 + \vert \nabla u^{(m)}_\epsilon \vert ^2)^{\frac{q(z)}{2}}}{q(z)}\right)  \,dx\\
&=\int_{\Omega} f (u^{(m)}_\epsilon)_t \,dx  - \int_{\Omega} \frac{a(z) p_t (\epsilon^2 + \vert \nabla u^{(m)}_\epsilon \vert ^2)^{\frac{p(z)}{2}}}{p^2(z)} \left(1- \frac{p(z)}{2} \ln(\epsilon^2 + \vert \nabla u^{(m)}_\epsilon \vert ^2)\right)\,dx\\
& \qquad  - \int_{\Omega} \frac{b(z) q_t(z) (\epsilon^2 + \vert \nabla u^{(m)}_\epsilon
\vert ^2)^{\frac{q(z)}{2}}}{q^2(z)} \left(1- \frac{q(z)}{2}
\ln((\epsilon^2 + \vert \nabla u^{(m)}_\epsilon \vert ^2))\right)\,dx
\\
& \qquad + \int_{\Omega} \frac{a_t (\epsilon^2 + \vert \nabla u_\epsilon^{(m)}\vert ^2)^\frac{p(z)}{2}}{p(z)} \,dx + \int_{\Omega} \frac{b_t (\epsilon^2 + \vert \nabla u_\epsilon^{(m)}\vert ^2)^\frac{q(z)}{2}}{q(z)} \,dx
\\
&
\equiv \int_{\Omega} f (u^{(m)}_\epsilon)_t \,dx +\mathcal{I}_1+\mathcal{I}_2+\mathcal{I}_3 +\mathcal{I}_4.
\end{split}
\end{equation}
The first term on the right-hand side of \eqref{est014} is estimated by the Cauchy inequality:
\begin{equation}\label{est016}
\left\vert \int_{\Omega} f (u^{(m)}_\epsilon)_t \,dx\right\vert  \leq
\frac{1}{2} \|  (u^{(m)}_\epsilon)_t {(\cdot,t)}\|  ^2_{2,\Omega} +
\frac{1}{2}\|  f{(\cdot,t)}\|  ^2_{2,\Omega}.
\end{equation}
To estimate $\mathcal{I}_i$, fix a number $\ell \in (r_\ast, r^{\#})$ 
such that
\[
\overline{s}(z) < \underline{s}(z) + \ell < \underline{s}(z) + r^{\#}.
\]
Applying the Young inequality and using \eqref{eq:log} in the terms with the logarithmic growth we obtain:
\begin{equation*}
\begin{split}
\sum_{i=1}^{4}\vert \mathcal{I}_i\vert  & \leq
C\left(1+  \int_{\Omega} \vert \nabla u_\epsilon^{(m)}\vert ^{\underline{s}(z) + \ell} \,dx + \int_{\Omega} \mathcal{F}_{\epsilon}^{(r_1-\varsigma, r_2-\varsigma)}(z,\nabla u_{\epsilon}^{(m)})\vert \nabla u_{\epsilon}^{(m)}\vert ^2 \,dx \right)
\end{split}
\end{equation*}
with $r_i$ defined in \eqref{eq:r-i}. The required inequality \eqref{timederiest} follows after
gathering the above estimates, integrating the result in $t$,
and using \eqref{eq:choice} and \eqref{eq:ineq-high}.
\end{proof}

\section{The regularized problem}
\label{sec:reg-existence}
We are now in position to prove the existence and uniqueness of strong solutions of the regularized problem \eqref{eq:reg-prob} via passing to the limit as $m \to \infty$ in the equation satisfied by the functions $u^{(m)}_\epsilon$. We proceed in two steps. In the first step we prove the existence and establish the regularity properties of solutions to the problems in a smooth domain $\Omega$,  in the second step we will extend these results to the domains with $\partial\Omega\in C^2$.

Let us assume that $\partial\Omega \in C^{k}$ with $k\geq 2+[\frac{N}{2}]$ and $\epsilon\in (0,1)$ is a fixed parameter. Under the assumptions of Theorem \ref{th:main-result-1} on the rest of the data,
there exists a sequence of Galerkin approximations $\{u^{(m)}_\epsilon\}$ given by \eqref{eq:coeff} with the coefficients $u_j^{(m)}(t)$ defined  from the system of ordinary differential equations \eqref{system}. The functions $u^{(m)}(z)$ satisfy estimates
\eqref{secderiboun}, \eqref{gradbound}, \eqref{eq:ineq-0},
\eqref{eq:ineq-high}, \eqref{eq:ineq-high-1} and \eqref{timederiest}. The finite-dimensional approximations $u_{\epsilon}^{(m)}$ are constructed in the cylinders $\Omega\times (0,T_m)$, but the uniform estimates \eqref{secderiboun} and \eqref{timederiest} mean that $u_{\epsilon}^{(m)}(\cdot,T_m)\in W^{1,2}_0(\Omega)\cap W^{1,\mathcal{H}}_0(\Omega)$. This inclusion allows one to continue $u_{\epsilon}^{(m)}$ to an interval $(T_m,T_m+h)$. Continuing this procedure we exhaust the interval $(0,T)$ and obtain the approximate solution $u_\epsilon^{(m)}$ on the whole of the cylinder $Q_T$.

\subsection{Existence and uniqueness of strong solution}
\begin{theorem}
\label{th:exist-reg} Let $u_0$, $f$, $p$, $q$, $a$
satisfy the conditions of Theorem \ref{th:main-result-1} with $\partial\Omega\in C^k$, $k\geq 2+[\frac{N}{2}]$. Then for every
$\epsilon\in (0,1)$ problem \eqref{eq:reg-prob} has a unique  solution $u_\epsilon$ which satisfies the estimates
\begin{equation}
\label{eq:est-strong-eps-1}
\begin{split}
& \|  u_\epsilon\|  _{{ \mathbb{W}_{\overline{s}(\cdot)}(Q_T)}}\leq C_0,
\\
&
\operatorname{ess}\sup_{(0,T)}\|  u_\epsilon{(\cdot, t)}\|  _{2,\Omega}^{2}
+
\|  u_{\epsilon t}\|  _{2,Q_T}^{2}+ \operatorname{ess}\sup_{(0,T)}\|  \nabla u_\epsilon{(\cdot, t)}\|  _{2,\Omega}^{2}
\\
&
+
\operatorname{ess}\sup_{(0,T)}\int_{\Omega}
\left(a(z)(\epsilon^2+\vert \nabla u_\epsilon\vert ^2)^{\frac{p(z)}{2}}+b(z)(\epsilon^2+\vert \nabla u_\epsilon\vert ^2)^{\frac{q(z)}{2}}\right)
dx\leq C_1.
\end{split}
\end{equation}
Moreover, the solution $u_\epsilon$ has the following property of global higher
integrability of the gradient: for every $\varsigma \in (0, r^{\#})$,

\begin{equation}
\label{eq:grad-high-eps} \int_{Q_T} \vert \nabla u_{\epsilon}\vert ^{\underline{s}(z)+r^\sharp-\varsigma} \, dz \leq C_2.
\end{equation}
The constants $C_0$, $C_1$, $C_2$ depend on the data but are independent of $\epsilon$.
\end{theorem}
\begin{proof}
The uniform with respect to $m$ and $\epsilon$ estimates \eqref{secderiboun}, \eqref{gradbound}, \eqref{eq:ineq-0},
\eqref{eq:ineq-high}, \eqref{eq:ineq-high-1} and \eqref{timederiest}  enable one to extract a subsequence
$u^{(m)}_{\epsilon}$ (for which we keep the same notation), and
functions $u_\epsilon$, $\mathcal{A}_{1, \epsilon}$, $\mathcal{A}_{2, \epsilon}$ such that

\begin{equation}
\label{eq:conv}
\begin{split}
& u^{(m)}_\epsilon \to u_\epsilon \quad \text{$\star$-weakly in $
L^\infty(0,T;L^2(\Omega))$}, \quad \text{$(u^{(m)}_{\epsilon})_t \rightharpoonup (u_{\epsilon})_t$ in
$L^2(Q_T)$},
\\
&\text{$\nabla u^{(m)}_\epsilon \rightharpoonup \nabla
u_\epsilon$  in $(L^{r(\cdot)}(Q_T))^N$}, \quad r(z)= \max\{2, \overline{s}(z)\},
\\
& \text{$\beta_\epsilon^{\frac{p(z)-2}{2}}(\nabla
u^{(m)}_\epsilon) \nabla
u^{(m)}_\epsilon\rightharpoonup \mathcal{A}_{1, \epsilon}$ in
$(L^{\overline{s}'(\cdot)}(Q_T))^N$,}
\\
& \text{$\beta_\epsilon^{\frac{q(z)-2}{2}}(\nabla
u^{(m)}_\epsilon) \nabla
u^{(m)}_\epsilon\rightharpoonup \mathcal{A}_{2, \epsilon}$ in
$(L^{\overline{s}'(\cdot)}(Q_T))^N$},
\end{split}
\end{equation}
where the third and fourth lines follow from the uniform estimate \eqref{eq:ineq-high-1}. The continuous Sobolev's embedding implies uniform boundedness of the functions $u^{(m)}_\epsilon$ and $(u^{(m)}_\epsilon)_t$ in $L^{\infty}(0,T;W_0^{1,\underline{s}^-}(\Omega))$ and
$L^2(0,T; L^2(\Omega))$ respectively, while $W_0^{1,r(\cdot,t)}(\Omega) \subseteq
W^{1,\underline{s}^-}_0(\Omega)\hookrightarrow L^2(\Omega)$. By \cite[Sec.8, Corollary 4]{simon-1987} the sequence $\{u^{(m)}_\epsilon\}$ is relatively compact in $C^0([0,T];L^2(\Omega))$ and there exists a subsequence $\{u^{(m_k)}_\epsilon\}$, which we also assume coinciding with $\{u^{(m)}_{\epsilon}\}$, such that $u^{(m)}_\epsilon \to u_\epsilon$ in
$C([0,T];L^2(\Omega))$ and a.e. in $Q_T.$
Let $\mathcal{N}_m$ be the finite-dimensional sets defined in \eqref{eq:N}.
Fix some $m\in \mathbb{N}$. By the method of construction
$u^{(m)}_\epsilon\in \mathcal{N}_m$. Since $\mathcal{N}_{k}\subset
\mathcal{N}_{m}$ for $k<m$, then for every $\xi_k\in
\mathcal{N}_k$ with $k\leq m$
\begin{equation}
\label{eq:ident-m} \int_{Q_T} u^{(m)}_{\epsilon t} \xi_k \,dz +
\int_{Q_T} \mathcal{F}^{(0,0)}_{\epsilon}(z, \nabla u_\epsilon^{(m)}) \nabla u^{(m)}_\epsilon
\cdot \nabla \xi_k \,dz = \int_{Q_T} f \xi_k \,dz.
\end{equation}
Since $\bigcup\limits_{k \in \mathbb{N}} \mathcal{N}_k$ is dense in $\mathbb{W}_{\overline{s}(\cdot)}(Q_T)$, for every $\xi\in \mathbb{W}_{\overline{s}(\cdot)}(Q_T)$ there exists a sequence such that $\xi_k\in \mathcal{N}_k$ and $\mathcal{N}_k\ni \xi_k\to \xi\in \mathbb{W}_{\overline{s}(\cdot)}(Q_T)$.
If $\zeta_m\rightharpoonup \zeta$ in $L^{\overline{s}'(\cdot)}(Q_T)$, then for every $\eta \in L^{\overline{s}(\cdot)}(Q_T)$ we have
\[
a(z) \eta \in L^{\overline{s}(\cdot)}(Q_T) \quad  \text{and}\quad \displaystyle\int_{Q_T}a \zeta_m \eta \,dz\to \int_{Q_T}a \zeta \eta \,dz.
\]
Using this fact we pass to the limit as $m\to\infty$ in \eqref{eq:ident-m} with a fixed $k$, and then letting $k\to \infty$, we conclude that

\begin{equation}
\label{eq:ident-lim} \int_{Q_T} u_{\epsilon t} \xi \,dz +
\int_{Q_T} a(z) \ \mathcal{A}_{1, \epsilon} \cdot \nabla \xi \,dz + \int_{Q_T} b(z) \ \mathcal{A}_{2, \epsilon} \cdot \nabla \xi \,dz = \int_{Q_T} f \xi
\,dz
\end{equation}
for all $\xi \in {\mathbb{W}_{\overline{s}(\cdot)}(Q_T)}$. The
classical argument based on monotonicity of the flux function (see Proposition \ref{pro:strict-monotone}),
the uniform estimates \eqref{eq:ineq-high}, and the density of $\bigcup\limits_{k \in \mathbb{N}} \mathcal{N}_k$ in ${\mathbb{W}_{\overline{s}(\cdot)}(Q_T)}$, allow one to identify the limit vectors $\mathcal{A}_{1, \epsilon}$ and $\mathcal{A}_{2, \epsilon}$ as follows (see, e.g., \cite[Theorem 6.1]{arora_shmarev2020} for the details): for every $\psi\in \mathbb{W}_{\overline{s}(\cdot)}(Q_T)$
\[
\int_{Q_T}\left(a(z) \mathcal{A}_{1, \epsilon}  + b(z) \mathcal{A}_{2, \epsilon}\right)\cdot\nabla \psi\,dz= \int_{Q_T}\mathcal{F}^{(0,0)}_{\epsilon}(z, \nabla u_\epsilon) \nabla u_{\epsilon} \cdot \nabla \psi\,dz.
\]
The initial condition for $u_\epsilon$ is fulfilled by continuity
because $u_\epsilon\in C^0([0,T];L^{2}(\Omega))$.

To prove the uniqueness we argue by contradiction. Let $v,w$ be two solutions of problem \eqref{eq:reg-prob}. Take an arbitrary $\rho \in (0,T]$. The uniform estimates \eqref{eq:ineq-high} allow us to take $v-w \in \mathbb{W}_{\overline{s}(\cdot)}(Q_T)$ for the test function in equalities \eqref{eq:def} for $v$ and $w$ in the cylinder $Q_\rho=\Omega\times (0,\rho)$. Subtracting these equalities and using the monotonicity of the flux (see Proposition \ref{pro:strict-monotone}) we arrive at the inequality

\[
\frac{1}{2}\|  v-w\|  _{2,\Omega}^{2}(\rho) =\int_{Q_\tau}(v-w)(v-w)_t\,dz\leq 0.
\]
It follows that $v(x,\rho)=w(x,\rho)$ a.e. in $\Omega$ for every $\rho \in [0,T]$.

Estimates \eqref{eq:est-strong-eps-1} follow from the uniform in $m$ estimates on the functions $u_{\epsilon}^{(m)}$ and their derivatives, the properties of weak convergence \eqref{eq:conv}, and the lower semicontinuity of the modular. Inequality \eqref{eq:ineq-high} yields that for every $\ell \in (0,r^\sharp)$ the sequence $\{\nabla u_{\epsilon}^{(m)}\}$ contains a subsequence which converges to $\nabla u_\epsilon$ weakly in $(L^{\underline{s}(\cdot)+\ell}(Q_T))^N$, whence \eqref{eq:grad-high-eps}.
\end{proof}

\subsection{Strong and almost everywhere convergence of the gradients}
\begin{lemma}[Strong convergence of $\{\nabla u_{\epsilon}^{(m)}\}$]
\label{lem:pointwise-epsilon}
Assume the conditions of Theorem \ref{th:exist-reg}. There is a subsequence of the sequence $\{u_{\epsilon}^{(m)}\}$ such that

\[
\nabla u_{\epsilon}^{(m)}\to \nabla u_\epsilon\quad \text{in $L^{\underline{s}(\cdot)}(Q_T)$ and a.e. in $Q_T$ as $m\to \infty$}.
\]
\end{lemma}

\begin{proof}
The sequence $\{u_\epsilon^{(m)}\}$ possesses the convergence properties \eqref{eq:conv}. It follows from the weak convergence $\nabla u_\epsilon^{(m)}\rightharpoonup \nabla u_\epsilon$ in $L^{\overline{s}(\cdot)}(Q_T)$, the strong convergence $u_\epsilon^{(m)}\to u_{\epsilon}$ in $L^{2}(Q_T)$, and the Mazur Lemma, see \cite[Corollary 3.8, Chapter 3]{Brezis-book}, that there exists a sequence $\{v_\epsilon^{(m)}\}$ such that $v_\epsilon^{(m)}\in \mathcal{N}_m$, each $v_\epsilon^{(m)}$ is a convex combination of $
\{u_\epsilon^{(1)},u_\epsilon^{(2)},\ldots,u_\epsilon^{(m)}\}$, and

\begin{equation}
\label{eq:conv-hull}
\text{$v_\epsilon^{(m)}\to u_{\epsilon}$ in $\mathbb{W}_{\overline{s}(\cdot)}(Q_T)$}.
\end{equation}
Let us define $w_m\in \mathcal{N}_m$ as follows:

\[
\|  u_\epsilon-w_m\|  _{\mathbb{W}_{\overline{s}(\cdot)}(Q_T)}= \operatorname{dist}(u_\epsilon,\mathcal{N}_m)\equiv \min\{\|  u_\epsilon-w\|  _{\mathbb{W}_{\overline{s}(\cdot)}(Q_T)}: \,w\in \mathcal{N}_m\}.
\]
Because of \eqref{eq:conv-hull} such $w_m$ exists and

\begin{equation}
\label{eq:distance}
\operatorname{dist}(u_\epsilon,\mathcal{N}_m)=\|  u_\epsilon-w_m\|  _{\mathbb{W}_{\overline{s}(\cdot)}(Q_T)}\leq \|  u_\epsilon-v_\epsilon^{(m)}\|  _{\mathbb{W}_{\overline{s}(\cdot)}(Q_T)}\to 0\quad \text{as $m\to \infty$}.
\end{equation}
By the properties of the modular and the strong convergence $u_\epsilon^{(m)}\to u_{\epsilon}$ in $L^{2}(Q_T)$ we also have

\begin{equation}
\label{eq:distance-mod}
\begin{split}
& \|  u_\epsilon^{(m)}-w_m\|  _{2,Q_T}^2 +\int_{Q_T}\vert \nabla (u_\epsilon - w_m)\vert ^{\overline{s}(z)}\,dz
\\
& \leq 2\|  u_\epsilon^{(m)}-u_\epsilon\|  ^2_{2,Q_T}+2\|  u_{\epsilon}-w_m\|  ^2_{2,Q_T}+ \int_{Q_T}\vert \nabla (u_\epsilon - w_m)\vert ^{\overline{s}(z)}\,dz
\to 0
\end{split}
\end{equation}
as $m\to \infty$. Gathering \eqref{eq:def-reg} and \eqref{system} with the test function $u_\epsilon^{(m)} \in \mathcal{N}_m$  we obtain the equality

\[
\begin{split}
\int_{Q_T} & \left(\mathcal{F}_{\epsilon}^{(0,0)}(z,\nabla u_\epsilon^{(m)})\nabla u_\epsilon^{(m)}-\mathcal{F}_{\epsilon}^{(0,0)}(z,\nabla u_\epsilon)\nabla u_{\epsilon}\right)\nabla u_\epsilon^{(m)}\,dz
\\
&
=-\int_{Q_T}((u_\epsilon^{(m)})_t-u_{\epsilon t})u_\epsilon^{(m)}\,dz,
\end{split}
\]
which can be written in the form

\begin{equation}
\label{eq:defs}
\begin{split}
\int_{Q_T} & \left(\mathcal{F}_{\epsilon}^{(0,0)}(z,\nabla u_\epsilon^{(m)})\nabla u_\epsilon^{(m)}-\mathcal{F}_{\epsilon}^{(0,0)}(z,\nabla u_\epsilon)\nabla u_{\epsilon}\right)\nabla (u_\epsilon^{(m)}-u_\epsilon)\,dz
\\
& =-\int_{Q_T}((u_\epsilon^{(m)})_t-u_{\epsilon t})u_\epsilon^{(m)}\,dz
\\
&\quad
- \int_{Q_T} \left(\mathcal{F}^{(0,0)}_{\epsilon}(z,\nabla u_\epsilon^{(m)})\nabla u_\epsilon^{(m)}-\mathcal{F}^{(0,0)}_{\epsilon}(z,\nabla u_\epsilon)\nabla u_{\epsilon}\right)\nabla u_\epsilon\,dz.
\end{split}
\end{equation}
Taking $w_m\in \mathcal{N}_m$ for the test-function in \eqref{eq:def-reg}, \eqref{system}, we find that

\begin{equation}
\label{eq:test-w}
\begin{split}
 \int_{Q_T} & \left(\mathcal{F}^{(0,0)}_{\epsilon}(z,\nabla u_\epsilon^{(m)})\nabla u_\epsilon^{(m)}-\mathcal{F}^{(0,0)}_{\epsilon}(z,\nabla u_\epsilon)\nabla u_{\epsilon}\right)\nabla w_m\,dz
\\
& \quad
+\int_{Q_T} ((u_\epsilon^{(m)})_t-u_{\epsilon t})w_m\,dz=0.
\end{split}
\end{equation}
Adding \eqref{eq:test-w} to the right-hand side of \eqref{eq:defs} and using notation \eqref{eq:prelim-notation}, we obtain the equality

\begin{equation}
\label{eq:start}
\begin{split}
& \mathcal{G}_\epsilon(\nabla u_\epsilon^{(m)},\nabla u_\epsilon)
= -\int_{Q_T} ((u_\epsilon^{(m)})_t-u_{\epsilon t})(u_\epsilon^{(m)}-w_m)\,dz
\\
&
\qquad +  \int_{Q_T} \left(\mathcal{F}^{(0,0)}_{\epsilon}(z,\nabla u_\epsilon^{(m)})\nabla u_\epsilon^{(m)}-\mathcal{F}^{(0,0)}_{\epsilon}(z,\nabla u_\epsilon)\nabla u_{\epsilon}\right)\nabla (u_\epsilon-w_m)\,dz.
\end{split}
\end{equation}
The first term on the right-hand side of \eqref{eq:start} tends to zero as $m\to \infty$ because $(u_{\epsilon}-u_{\epsilon}^{(m)})_t$ are uniformly bounded in $L^2(Q_T)$ and $\|  u_\epsilon^{(m)}-w_m\|  _{2,Q_T}\to 0$ by the choice of $w_m$. By \eqref{eq:ineq-high-1} and due to the choice of $w_m$, the second term of \eqref{eq:start} is bounded by $C\|  \nabla u_{\epsilon}-\nabla w_m\|  _{\overline{s}(\cdot),Q_T}$ and also tends to zero as $m\to \infty$. Hence, $\mathcal{G}_{\epsilon}(\nabla u_\epsilon^{(m)},\nabla u_\epsilon)\to 0$ as $m\to \infty$. It follows now from Lemma \ref{le:cont-convergence} that

\[
\int_{Q_T}\vert \nabla (u_{\epsilon}^{(m)}-u_\epsilon)\vert ^{\underline{s}(z)}\,dz\to 0\quad \text{when $m\to \infty$.}
\]
By Riesz-Fischer Theorem $\nabla u^{(m)}_\epsilon\to \nabla u_{\epsilon}$ a.e. in $Q_T$ (up to a subsequence).
\end{proof}

\subsection{Second-order regularity}
\begin{theorem}
\label{th:flux-a.e.}
Let the conditions of Theorem \ref{th:exist-reg} hold. Then:

\begin{itemize}
\item[{\rm (i)}]  $(\mathcal{F}_{\epsilon}^{(0,0)})^\frac{1}{2}(z,\nabla u_{\epsilon})D_iu_{\epsilon}\in L^2(0,T;W^{1,2}(\Omega))$, $i=1,2,\ldots,N$, and

\[
\|  (\mathcal{F}_{\epsilon}^{(0,0)})^{\frac{1}{2}}(z,\nabla u_{\epsilon}) D_iu_{\epsilon}\|  _{L^2(0,T;W^{1,2}(\Omega))}\leq M,\quad i=1,2,\ldots,N,
\]
with an independent of $\epsilon$ constant $M$;

\item[{\rm (ii)}]
    $D_{ij}^2u_\epsilon\in L^{\underline{s}(\cdot)}_{loc}(Q_T\cap \{z:\,\max\{p(z), q(z)\}< 2\})$, $i,j=1,2,\ldots,N$, and

    \[
    \sum_{i,j=1}^{N}\|  D_{ij}^2u_\epsilon\|  _{\underline{s}(\cdot),Q_T\cap \{z:\,\max\{p(z), q(z)\}< 2\}}\leq M'
    \]
    with an independent of $\epsilon$ constant $M'$.
\end{itemize}
\end{theorem}

\begin{proof} (i)
The almost everywhere convergence $\nabla u_{\epsilon}^{(m)}\to \nabla u_{\epsilon}$ in $Q_T$ implies

\begin{equation}
\label{eq:flux-conv-m}
(\mathcal{F}_\epsilon^{(0,0)})^{\frac{1}{2}}(z,\nabla u_{\epsilon}^{(m)})\nabla u_{\epsilon}^{(m)}\to (\mathcal{F}_\epsilon^{(0,0)})^{\frac{1}{2}}(z,\nabla u_{\epsilon})\nabla u_{\epsilon}\quad \text{a.e. in $Q_T$}.
\end{equation}
Due to the uniform estimates \eqref{eq:ineq-0}, \eqref{eq:ineq-high}, and inequality \eqref{eq:log}, for every $i,j=1,2,\ldots,N$
\[
\begin{split}
& \left\|  D_i \left((\mathcal{F}_{\epsilon}^{(0,0)})^{\frac{1}{2}}(z,\nabla u_{\epsilon}^{(m)}) D_ju_{\epsilon}^{(m)}\right)\right\|  _{2,Q_T}^{2}
\leq C\int_{Q_T}\mathcal{F}^{(0,0)}_{\epsilon}(z,\nabla u_{\epsilon}^{(m)})\vert u_{\epsilon xx}^{(m)}\vert ^2\,dz
\\
& \qquad
+C'\int_{Q_T}\left((\epsilon^2+\vert \nabla u_{\epsilon}^{(m)}\vert ^2)^{\frac{p(z)}{2}}+(\epsilon^2+\vert \nabla u_{\epsilon}^{(m)}\vert ^2)^{\frac{q(z)}{2}}\right)\vert \ln(\epsilon^2+\vert \nabla u_{\epsilon}^{(m)}\vert ^2)\vert \,dz
\\
& \qquad \leq  C''\left(1+ \int_{Q_T}\vert \nabla u_{\epsilon}^{(m)}\vert ^{\overline{s}(z)+\mu}\,dz\right) \leq C'''\left(1+ \int_{Q_T}\vert \nabla u_{\epsilon}^{(m)}\vert ^{\underline{s}(z)+r}\,dz\right)\leq M
\end{split}
\]
with some $r \in (r_\ast
, r^{\sharp})$ and a constant $M$ depending on
the norm of $u_0$ in $W^{1,\mathcal{H}}_0(\Omega)$, the norm of $f$ in $L^{2}(0,T;W^{1,2}_0(\Omega))$, the constants in conditions \eqref{eq:a-b}, and $N$, $p^\pm$, $q^\pm$. When estimating the terms with the logarithmic growth, we used \eqref{eq:log} with $\mu$ and $r$ chosen according to the inequalities $0<\mu<r-r_\ast/2
$, $
r_\ast<r<r^\sharp$ and
\[
\overline{s}(z) - \underline{s}(z) + \mu < 
r_\ast+ \mu \leq r < r^\sharp .
\]
Therefore, there is a subsequence $\{u_\epsilon^{(m_k)}\}$ (we may assume that it coincides with the whole sequence) and functions $\Theta_{ij}\in L^{2}(Q_T)$  such that

\[
D_i \left((\mathcal{F}_\epsilon^{(0,0)})^{\frac{1}{2}}(z,\nabla u_{\epsilon}^{(m)})D_ju_\epsilon^{(m)}\right) \rightharpoonup \Theta_{ij}\in L^{2}(Q_T)\quad \text{as $m\to \infty$}.
\]
The uniform global higher integrability of the gradients \eqref{eq:grad-high-eps} implies the existence of $\delta>0$ such that

\[
(\mathcal{F}_{\epsilon}^{(0,0)})^{\frac{1}{2}}(z,\nabla u_{\epsilon}^{(m)}) D_ju_{\epsilon}^{(m)}\quad \text{are uniformly bounded in $L^{2+\delta}(Q_T)$}.
\]
By Lemma \ref{lem:pointwise-epsilon} $(\mathcal{F}_{\epsilon}^{(0,0)})^{\frac{1}{2}}(z,\nabla u_{\epsilon}^{(m)}) D_ju_{\epsilon}^{(m)}$ converge pointwise to $(\mathcal{F}_{\epsilon}^{(0,0)})^{\frac{1}{2}}(z,\nabla u_{\epsilon}) D_ju_{\epsilon}$. It follows then from the Vitali convergence theorem that

\[
\text{$(\mathcal{F}_\epsilon^{(0,0)})^{\frac{1}{2}}(z,\nabla u_{\epsilon}^{(m)}) D_ju_{\epsilon}^{(m)}\to (\mathcal{F}_\epsilon^{(0,0)})^{\frac{1}{2}}(z,\nabla u_{\epsilon}) D_ju_{\epsilon}$ in $L^{2}(Q_T)$}.
\]
For every $\psi\in C^{\infty}(Q_T)$ with $\operatorname{supp}\psi \Subset Q_T$ and $i,j=1,\ldots,N$

\[
\begin{split}
& \left(D_i \left((\mathcal{F}_\epsilon^{(0,0)})^{\frac{1}{2}}(z,\nabla u_{\epsilon}^{(m)}) D_ju_{\epsilon}^{(m)}\right),\psi\right)_{2,Q_T}
\\
&
= - \left((\mathcal{F}_\epsilon^{(0,0)})^{\frac{1}{2}}(z,\nabla u_{\epsilon}^{(m)}) D_ju_{\epsilon}^{(m)},D_i\psi\right)_{2,Q_T}
\to  - \left((\mathcal{F}_{\epsilon}^{(0,0)})^{\frac{1}{2}}(z,\nabla u_{\epsilon}) D_ju_{\epsilon},D_i\psi\right)_{2,Q_T}
\end{split}
\]
as $m\to \infty$. Thus, it is necessary that

\[
\Theta_{ij}=D_i\left((\mathcal{F}_\epsilon^{(0,0)})^{\frac{1}{2}}(z,\nabla u_{\epsilon}) D_ju_{\epsilon}\right)\in L^{2}(Q_T)\quad \text{and}\quad \|  \Theta_{ij}\|  _{2,Q_T}^2\leq M.
\]

\medskip

(ii) Let us denote $\mathcal{D}_{p,q}=Q_T\cap \{z:\,\max\{p(z), q(z)\} < 2\}$ and take a set $\mathcal{B} \Subset \mathcal{D}_{p,q}$. By Young's inequality, \eqref{gradbound}, \eqref{eq:ineq-0} and due to the assumption $a(z) + b(z) \geq \alpha>0$, for every $i,j=1,2,\ldots,N$

\begin{equation}
\label{eq:int-bound}
\begin{split}
& \int_{\mathcal{B}}\vert D^2_{ij}u_{\epsilon}^{(m)}\vert ^{\underline{s}(z)}\,dz
\\
& \qquad \leq C\left( 1 + \int_{\mathcal{B}}a(z)\vert D^2_{ij}u_{\epsilon}^{(m)}\vert ^{p(z)}\,dz + \int_{\mathcal{B}}b(z)\vert D^2_{ij}u_{\epsilon}^{(m)}\vert ^{q(z)}\,dz\right) \\
\end{split}
\end{equation}
with a constant $C'$ independent of $\epsilon$ and $m$. Using the Young inequality we estimate the first integral on the right-hand side with the help of the inequality

\[
\begin{split}
a\vert D_{ij}^{2}u_{\epsilon}^{(m)}\vert ^{p} & =\dfrac{a^{1-\frac{p}{2}}}{(\epsilon^2+\vert \nabla u_\epsilon^{(m)}\vert ^2)^{\frac{p-2}{2}\frac{p}{2}}}\left(a(\epsilon^2+\vert \nabla u_\epsilon^{(m)}\vert ^2)^{\frac{p-2}{2}}\vert D_{ij}^{2}u_\epsilon^{(m)}\vert ^2\right)^{\frac{p}{2}}
\\
& \leq C_1 (\epsilon^2+\vert \nabla u_\epsilon^{(m)}\vert ^2)^{\frac{p}{2}} + C_2 a(\epsilon^2+\vert \nabla u_\epsilon^{(m)}\vert ^2)^{\frac{p-2}{2}}\vert D_{ij}^{2}u_\epsilon^{(m)}\vert ^2.
\end{split}
\]
The same inequality with $p$ and $a$ replaced by $q$ and $b$ is applied to the second integral. By virtue of \eqref{eq:grad-high-eps} and \eqref{eq:ineq-0} the right-hand side of \eqref{eq:int-bound} is bounded uniformly with respect to $m$ and $\epsilon$ by a constant $C$. It follows that there exist $\theta_{ij} \in L^{\underline{s}(\cdot)}(\mathcal{D}_{p,q})$ such that $D^{2}_{ij}u_{\epsilon}^{(m)}\rightharpoonup \theta_{ij}$ in $L^{\underline{s}(\cdot)}(\mathcal{B})$ (up to a subsequence). Since $\nabla u_\epsilon^{(m)}\rightharpoonup \nabla u_{\epsilon}$ in $L^{r(\cdot)}(Q_T)$ with $r(z)= \max\{2, \overline{s}(z)\}$, then for every $\psi\in C^{\infty}_0(\mathcal{B})$
\[
(\theta_{ij},\psi)_{2,Q_T}=\lim_{m\to \infty}(D^{2}_{ij}u^{(m)}_\epsilon,\psi)_{2,Q_T}=- \lim_{m\to \infty}(D_{i}u^{(m)}_\epsilon,D_j\psi)_{2, Q_T}=(D_iu_\epsilon,D_j\psi)_{2,Q_T}.
\]
It follows that $\theta_{ij}=D^{2}_{ij}u_{\epsilon}$, and $\|  D^{2}_{ij}u_{\epsilon}\|  _{\underline{s}(\cdot),\mathcal{B}}\leq C$ by the lower semicontinuity of the modular.
\end{proof}

\subsection{Regularized problem with $\partial\Omega\in C^2$}

Let $\partial\Omega\in C^2$. By Proposition \ref{pro:density-2} and due to the density of $C^{\infty}([0,T];C_0^{\infty}(\Omega)$ in $L^2(0,T;W_0^{1,2}(\Omega))$ there exist sequences $\{v_{0\delta}\}$, $\{f_\delta\}$ with the following properties:

\[
    \begin{split}
    & v_{0\delta}\in C_0^{\infty}(\Omega),\quad \text{$v_{0\delta}\to u_0$ in $W^{1,\mathcal{H}}_0(\Omega)$},\quad \operatorname{supp}v_{0\delta}\Subset \Omega,
    \\
    & f_\delta\in C^{\infty}([0,T];C_0^{\infty}(\Omega)),\quad \text{$f_\delta\to f$ in $L^2(0,T;H^1_0(\Omega))$},
    \\
    &\operatorname{supp}f_\delta(\cdot,t)\Subset \Omega\;\;\text{for all $t\in (0,T)$}.
    \end{split}
    \]
    Let $d(x)=\operatorname{dist}(x,\partial \Omega)\equiv \inf\{\vert x-y\vert :\;y\in \partial\Omega\}$ be the distance from the point $x\in \Omega$ to the boundary. By \cite[Lemma 14.16]{GT} there exists $\mu>0$ such that $d(x)\in C^2(\Gamma_\mu)$, where $\Gamma_\mu=\{x\in \overline{\Omega}: d(x)<\mu\}$. For $x\in \Gamma_{\mu/2}\setminus \Gamma_{\mu/4}=\{x\in \Omega:\,\mu/4<d(z)<\mu/2\}$ and $0<\sigma<\mu/4$ we consider the mollified distance $d_\sigma(x)=d(y)\star\phi_\sigma(x-y)$ where $\phi_\sigma(\cdot)$ denotes the Friedrichs mollifier. Since $d(x)\in C^2(\Gamma_\mu)$, it follows from the well known properties of the mollifier that for $x\in \Gamma_\beta\setminus \Gamma_{\gamma}$, $0<\sigma<\gamma<\beta$, and the multi-index $\alpha$, $0\leq \vert \alpha\vert \leq 2$,
    \[
    \vert D^\alpha_xd_\sigma(x)-D_x^{\alpha}d(x)\vert \leq \max\{\vert D^\alpha_xd(x)-D_y^{\alpha}d(y)\vert :\,\vert x-y\vert <\sigma<\gamma\}\to 0
    \]
    as $\sigma\to 0$. Moreover, $d_\sigma(x)\in C^{\infty}\left(\Gamma_\beta\setminus \Gamma_\gamma\right)$. Let us take the sequence of disjoint intervals $J_k=(a_k,b_k)$ with the endpoints $a_k=\mu2^{-(2k+1)}$, $b_k=\mu 2^{-2k}$, $k\in \mathbb{N}$, and the centers $c_k=(a_k+b_k)/2$. Let us also take the sequence of co-centered intervals
    \[
    I_k=\left(c_k-\dfrac{L_k}{4},c_k+\dfrac{L_k}{4}\right)\Subset J_k,\qquad L_k=b_k-a_k.
    \]
    Let $0<\sigma<c_k-L_k/4$. By Sard's theorem the set of critical values of $d_\sigma (x)$ in $I_k$ has zero measure. It follows that for every sufficiently large $k\in \mathbb{N}$ we may find $\delta_k\in I_k$ and $\sigma_k>0$ so small that
    \[
    \Sigma_{\delta_k}=\{x\in \Omega:\,d_{\sigma_k}(x)=\delta_k\}\subset \{x\in \Omega:\,a_k<d(x)<b_k\}.
    \]
    The surfaces $\Sigma_{\delta_k}$ have no common points, are $C^\infty$-smooth, and their  parametrizations are uniformly bounded in $C^2$. The domains $\Omega_{\delta_k}$ bounded by $\Sigma_{\delta_k}$ form an expanding sequence covering $\Omega$ when $k\to \infty$.

    Let $\{u_{\epsilon,\delta_k}\}$ be the sequence of strong solutions of problem \eqref{eq:reg-prob} in the cylinders $Q^{(\delta_k)}_T=\Omega_{\delta_k}\times (0,T)$ with the data $v_{0\delta_k}$, $f_{\delta_k}$, $\operatorname{supp} v_{0\delta_k}\Subset \Omega_{\delta_k}$, $\operatorname{supp} f_{\delta_k}(\cdot,t)\Subset \Omega_{\delta_k}$ for a.e. $t\in (0,T)$. By $\widetilde u_{\epsilon,\delta_k}$ we denote the zero continuations of $u_{\epsilon,\delta_k}$ from $Q_T^{(\delta_k)}$ to the cylinder $Q_T$. By Theorems \ref{th:exist-reg}, \ref{th:flux-a.e.} the continued functions satisfy the uniform estimates

    \begin{equation}
    \label{eq:est-loc-loc}
    \begin{split}
   \mathrm{(i)} & \qquad  \|  \partial_t\widetilde u_{\epsilon,\delta_k}\|  _{2,Q_T}
    \\
    &
    \qquad\quad  +\operatorname{ess} \sup_{(0,T)}\int_{\Omega}\left(1+\mathcal{F}_\epsilon^{(0,0)}(z,\nabla \widetilde u_{\epsilon,\delta_k})\right)\vert \nabla \widetilde u_{\epsilon,\delta_k}\vert ^2\,dx \leq C_{1},
   \\
   \mathrm{(ii)} & \qquad \operatorname{ess} \sup_{(0,T)}\|  \widetilde u_{\epsilon,\delta_k}\|  ^2_{2,\Omega} + \int_{Q_T}\mathcal{F}_\epsilon^{(0,0)}(z,\nabla \widetilde u_{\epsilon,\delta_k})\vert \nabla \widetilde u_{\epsilon,\delta_k}\vert ^2\,dz\leq C_{2},
    \\
    \mathrm{(iii)}& \qquad \left\|  D_{x_i}\left(\left(\mathcal{F}_\epsilon^{(0,0)}(z,\nabla \widetilde u_{\epsilon,\delta_k})\right)^{\frac{1}{2}}D_{x_j}\widetilde u_{\epsilon,\delta_k}\right)
\right\|  _{2,Q_T}\leq C_{3},\quad
\\
\mathrm{(iv)} & \qquad
\int_{Q_{T}}\vert \nabla \widetilde u_{\epsilon,\delta_k}\vert ^{\underline{s}(z)+r}\,dz\leq C_4,\qquad r\in (0,r^\sharp),
    \end{split}
    \end{equation}
    with constants $C_i$ depending on the data, but independent of $\epsilon$ and $\delta_k$; $C_1$, $C_3$, $C_4$ depend also on the $C^2$-norm of the parametrization of $\partial\Omega$. Using \eqref{eq:est-loc-loc} we may choose a subsequence with the following convergence properties:

    \[
    \begin{split}
    & \text{$\widetilde u_{\epsilon,\delta_k}\to u_{\epsilon}$ $\star$-weakly in $L^{\infty}(0,T;L^2(\Omega))$},
    \quad
    \text{$\partial_t\widetilde u_{\epsilon,\delta_k}\rightharpoonup \partial_tu_{\epsilon}$ in $L^2(Q_T)$},
    \\
    &
    \text{$\nabla \widetilde u_{\epsilon,\delta_k}\rightharpoonup \nabla u_\epsilon$ in $L^{\underline{s}(\cdot)+r}$},
    \quad \text{$a(z)\vert \nabla \widetilde u_{\epsilon,\delta_k}\vert ^{p-2}\nabla \widetilde u_{\epsilon,\delta_k} \rightharpoonup \mathcal{A}_{1,\epsilon}$ in $L^{p'(\cdot)}(Q_T)$},
    \\
    &
    \text{$b(z)\vert \nabla \widetilde u_{\epsilon,\delta_k})\vert ^{q-2}\nabla \widetilde u_{\epsilon,\delta_k} \rightharpoonup \mathcal{A}_{2,\epsilon}$ in $L^{q'(\cdot)}(Q_T)$}.
    \end{split}
    \]
    To identify the limits $\mathcal{A}_{i,\epsilon}$ and to prove that $u_{\epsilon}$ is a solution of problem \eqref{eq:reg-prob} in $Q_T$ we imitate the proof of Theorem \ref{th:exist-reg} and use Proposition \ref{pro:density-2}. The proof of uniqueness does not require any changes.

    The higher integrability of the gradient follows immediately from \eqref{eq:est-loc-loc} (iv).  To prove the second-order regularity we need the pointwise convergence of $\nabla \widetilde u_{\epsilon,\delta_k}$ in $Q_T$. We mimic the proof of Lemma 7.1. By Mazur's Lemma there is a sequence of convex combinations of $\{\widetilde u_{\epsilon,\delta_1}, \widetilde u_{\epsilon,\delta_2}, \dots,\widetilde u_{\epsilon,\delta_k}\}$ that converges to $u_\epsilon$ strongly in $\mathbb{W}_{\overline{s}(\cdot)}(Q_T)$. Let us denote this sequence by $\{w_{\delta_k}\}$, $\operatorname{supp} w_{\delta_k}\subseteq Q_{T}^{(\delta_k)}$. Let $W_{\delta_k}\in \mathbb{W}_{\overline{s}(\cdot)}(Q_T)$ with $\operatorname{supp}W_{\delta_k}\subseteq Q_T^{(\delta_k)}$ be defined as follows:

\[
\begin{split}
& \|  W_{\delta_k}-u_\epsilon\|  _{\mathbb{W}_{\overline{s}(\cdot)}(Q_T)}
\\
&
=\inf\left\{\|  w-u_\epsilon\|  _{\mathbb{W}_{\overline{s}(\cdot)}(Q_T)}\,\vert \, w\in \mathbb{W}_{\overline{s}(\cdot)}(Q_T), \ \operatorname{supp}w \subseteq Q_T^{(\delta_k)} \right\}\to 0\quad \text{as $\delta_k\to 0$}.
\end{split}
\]
In the identities

\[
\begin{split}
& \int_{Q_T} \left(\phi \partial_tu_\epsilon+\mathcal{F}_{\epsilon}^{(0,0)}(z,\nabla u_\epsilon)\nabla u_{\epsilon}\cdot\nabla \phi\right)\,dz=\int_{Q_T}f\phi\,dz\qquad \phi\in \mathbb{W}_{\overline{s}(\cdot)}(Q_T),
\\
& \int_{Q_T} \left(\phi \partial_t\widetilde u_{\epsilon,\mu}+\mathcal{F}_{\epsilon}^{(0,0)}(z,\nabla \widetilde u_{\epsilon,\mu})\nabla \widetilde u_{\epsilon,\mu}\cdot\nabla \phi\right)\,dz=\int_{Q_T}\widetilde f_{\mu}\phi\,dz \qquad \phi\in \mathbb{W}_{\overline{s}(\cdot)}(Q_T^{(\mu)})
\end{split}
\]
we may take for the test function $\phi=W_{\delta_k}-\widetilde u_{\epsilon,\mu}$ with $\delta_k \geq \mu$, which means that $\operatorname{supp}W_{\delta_k}\subseteq \operatorname{supp}\widetilde u_{\epsilon,\mu}$. Combining the results we obtain

\[
\begin{split}
& \mathcal{G}_{\epsilon}(\nabla u_\epsilon,\nabla \widetilde u_{\epsilon,\mu}) =
\int_{Q_T} \left(\mathcal{F}_{\epsilon}^{(0,0)}(z,\nabla u_{\epsilon})\nabla u_\epsilon - \mathcal{F}_{\epsilon}^{(0,0)}(z,\nabla \widetilde u_{\epsilon, \mu})\nabla \widetilde u_{\epsilon,\mu}\right)\cdot\nabla (u_\epsilon-\widetilde u_{\epsilon,\mu})\,dz
\\
& = \int_{Q_T}(f-\widetilde f_{\mu})(W_{\delta_k}-\widetilde u_{\epsilon,\mu}) - \int_{Q_T}(W_{\delta_k}-\widetilde u_{\epsilon,\mu})\partial_t(u_\epsilon-\widetilde u_{\epsilon,\mu})\,dz
\\
& \qquad -\int_{Q_T} \left(\mathcal{F}_{\epsilon}^{(0,0)}(z,\nabla u_{\epsilon})\nabla u_{\epsilon} - \mathcal{F}_{\epsilon}^{(0,0)}(z,\nabla \widetilde u_{\epsilon,\mu})\nabla \widetilde u_{\epsilon,\mu}\right)\cdot\nabla (W_{\delta_k}-u_\epsilon)\,dz.
\end{split}
\]
All terms on the right-hand side tend to zero as $\mu\to 0$. By virtue of Proposition \ref{pro:strict-1}, Lemma \ref{le:cont-emb-1} and Lemma \ref{le:cont-convergence} we have

\[
\int_{Q_T}\vert \nabla (u_\epsilon-\widetilde u_{\epsilon, \mu})\vert ^{\underline{s}(z)}\,dz\to 0\quad \text{and a.e. in $Q_T$ as $\mu\to 0$}.
\]
The second-order regularity follows now exactly as in the proof of Theorem \ref{th:flux-a.e.}. The above arguments are summarized in the following assertion.

\begin{theorem}
The assertions of Theorems \ref{th:exist-reg} and \ref{th:flux-a.e.} remain true for the domain $\Omega$ with $\partial \Omega\in C^2$.
\end{theorem}

\section{The degenerate problem}
\label{sec:proofs-main-results}
\subsection{Existence and uniqueness of strong solution: proof of Theorem \ref{th:main-result-1}}
 Let $\{u_\epsilon\}$ be the family of strong solutions of the regularized problems \eqref{eq:reg-prob} satisfying estimates \eqref{eq:est-strong-eps-1}. These uniform in $\epsilon$ estimates enable one to extract a sequence $\{u_{\epsilon_k}\}$ and find functions $u\in {\mathbb{W}_{\overline{s}(\cdot)}(Q_T)}$, $\mathcal{A}_1,\,\mathcal{A}_2 \in (L^{(\overline{s}(\cdot))^{'}}(Q_T))^N$ with the following properties:
\begin{equation}
\label{eq:conv-degenerate}
\begin{split}
& u_{\epsilon_k} \to u \quad \text{$\star$-weakly in $
L^\infty(0,T;L^2(\Omega))$}, \ \qquad \text{$u_{\epsilon_k t} \rightharpoonup u_{t}$ in
$L^2(Q_T)$},
\\
& \text{$\nabla u_{\epsilon_k} \rightharpoonup \nabla
u$  in $(L^{r(\cdot)}(Q_T))^N$} \quad \text{with} \ r(z)= \max\{2, \overline{s}(z)\},
\\
& \text{$\beta_{\epsilon_k}^{\frac{p(z)-2}{2}}(\nabla
u_{\epsilon_k}) \nabla
u_{\epsilon_k}\rightharpoonup \mathcal{A}_{1}$ in
$(L^{(\overline{s}(\cdot))^{'}}(Q_T))^N$,}
\\
& \text{$\beta_{\epsilon_k}^{\frac{q(z)-2}{2}}(\nabla
u_{\epsilon_k}) \nabla
u_{\epsilon_k}\rightharpoonup \mathcal{A}_{2}$ in
$(L^{(\overline{s}(\cdot))^{'}}(Q_T))^N$},
\end{split}
\end{equation}
where in third and fourth lines we make use of the uniform estimates \eqref{eq:est-strong-eps-1}. Moreover, $u\in C^0([0,T];L^2(\Omega))$. Each of $u_{\epsilon_k}$ satisfies the identity
\begin{equation}
\label{eq:ident-lim-k} \int_{Q_T} u_{\epsilon_k t} \xi \,dz +
\int_{Q_T} \mathcal{F}^{(0,0)}_{\epsilon_k}(z,\nabla u_{\epsilon_k})\nabla u_{\epsilon_k} \cdot \nabla \xi \,dz = \int_{Q_T} f \xi
\,dz
\end{equation}
for every $\xi \in \mathbb{W}_{\overline{s}(\cdot)}(Q_T)$, which yields

\begin{equation}
\label{eq:ident-prelim}
\int_{Q_T} u_{ t} \xi \,dz +
\int_{Q_T} (a(z) \mathcal{A}_1 + b(z) \mathcal{A}_2) \cdot \nabla \xi \,dz = \int_{Q_T} f \xi
\,dz.
\end{equation}
Identification of $\mathcal{A}_1$ and $\mathcal{A}_2$ is based on the monotonicity of the flux $\mathcal{F}_{\epsilon}^{(0,0)}(z,\xi)\xi$ with  $\epsilon\in [0,1)$, the argument is a literal repetition of the proof given in \cite[Theorem 2.1]{arora_shmarev2020}. We obtain the following equality:

\[
\int_{Q_T}\left(a(z) \mathcal{A}_{1}  + b(z) \mathcal{A}_{2}-\mathcal{F}_{0}^{(0,0)}(z,\nabla u)\right)\nabla u \cdot \nabla \phi\,dz=0\quad \forall \phi\in \mathbb{W}_{\overline{s}(\cdot)}(Q_T).
\]
Since $u\in C([0,T];L^{2}(\Omega))$, the initial condition is fulfilled by continuity. Estimates \eqref{eq:strong-est} follow from the uniform in $\epsilon$ estimates of Theorem \ref{th:exist-reg} and the lower semicontinuity of the modular. Uniqueness of a strong solution is an immediate consequence of the monotonicity of the flux.

\subsection{Continuity with respect to the data. Proof of Theorem \ref{th:energy}} Let $u$, $v$ be the strong solutions of problem \eqref{eq:main} with the data $\{u_0,f\}$ and $\{v_0,g\}$. The energy identity \eqref{eq:energy} follows if we take $u$ for the test function in \eqref{eq:def}. Let us take some $t,t+h\in [0,T]$. The function $w=u-v$ is an admissible test-function in identities \eqref{eq:def} for $u$ and $v$.  Combining these identities in the cylinder $\Omega\times (t,t+h)$ we obtain

\[
\begin{split}
\int_{t}^{t+h}\int_{\Omega}w_tw\,dz & + \int_{t}^{t+h}\int_{\Omega} \left(\mathcal{F}_0^{(0,0)}(z,\nabla u)\nabla u- \mathcal{F}_0^{(0,0)}(z,\nabla v)\nabla v\right)\cdot\nabla w\,dz
\\
&
= \int_{t}^{t+h}\int_{\Omega}(f-g)w\,dz.
\end{split}
\]
Let us divide this equality by $h$ and send $h\to 0$. By the Lebesgue differentiation theorem, for a.e. $t\in (0,T)$ each term has a limit, whence for a.e. $t\in (0,T)$

\begin{equation}
\label{eq:stab-1}
\begin{split}
\dfrac{1}{2}\dfrac{d}{dt}\left(\|  w\|  ^{2}_{2,\Omega}(t)\right) & +\int_{\Omega} \left(\mathcal{F}_0^{(0,0)}(z,\nabla u)\nabla u- \mathcal{F}_0^{(0,0)}(z,\nabla v)\nabla v\right)\cdot\nabla w\,dx
\\
&
= \int_{\Omega}(f-g)w\,dx.
\end{split}
\end{equation}
By Proposition \ref{pro:strict-monotone} the second term on the left-hand side is nonnegative. Dropping this term and applying the Cauchy inequality to the right-hand side we arrive at the differential inequality

\[
\dfrac{d}{dt}\left(\|  w\|  ^{2}_{2,\Omega}(t)\right) \leq \|  w(t)\|  ^2_{2,\Omega} + \|  (f-g)(t)\|  ^2_{2,\Omega},
\]
which can be integrated: for all $t\in [0,T]$

\begin{equation}
\label{eq:stab-2}
\begin{split}
\|  w(t)\|  _{2,\Omega}^2 & \leq {\rm e}^{t}\|  w(0)\|  ^2_{2,\Omega}+ \int_0^{t}{\rm e}^{\tau}\|  f-g\|  ^2_{2,\Omega}(\tau)\,d\tau
\\
& \leq {\rm e}^T\left(\|  w(0)\|  ^2_{2,\Omega}+\|  f-g\|  _{2,Q_T}^2\right).
\end{split}
\end{equation}
Let us revert to \eqref{eq:stab-1}, integrate it in $t$, and then drop the first nonnegative term in the resulting relation. In the notation \eqref{eq:prelim-notation}

\[
2\mathcal{G}_0(\nabla u-\nabla v)\leq \|  u_0-v_0\|  _{2,\Omega}^2+\|  f-g\|  ^2_{2,Q_T}+\|  w\|  _{2,Q_T}^2.
\]
According to \eqref{eq:stab-2}, the last term on the right-hand side tends to zero if the first two terms tend to zero. The assertion follows now from Proposition \ref{pro:strict-1} and Lemma \ref{le:cont-convergence}.

\subsection{Strong convergence of the gradients}

\begin{lemma}
\label{le:a-e-conv}
Let the conditions of Theorem \ref{th:main-result-1} be fulfilled and $\{u_\epsilon\}$ be the sequence of solutions of the regularized problems \eqref{eq:reg-prob}. Then $\nabla u_\epsilon \to \nabla u$ in $L^{\underline{s}(\cdot)}(Q_T)$ and a.e. in $Q_T$ (up to a subsequence).
\end{lemma}
\begin{proof}
The limit function $u\in \mathbb{W}_{\overline{s}(\cdot)}(Q_T)$ satisfies the identity
\begin{equation}
\label{eq:limit-eq}
\int_{Q_T} u_t\xi\,dz+\int_{Q_T} \mathcal{F}_0^{(0,0)}(z,\nabla u)\nabla u\cdot\nabla \xi \,dz=\int_{Q_T} f \xi \,dz\qquad \forall \xi\in \mathbb{W}_{\overline{s}(\cdot)}(Q_T).
\end{equation}
As distinguished from the case of the finite-dimensional Galerkin's approximations, the inclusions $u_\epsilon,u\in \mathbb{W}_{\overline{s}(\cdot)}(Q_T)$ allow us to take $u_\epsilon-u$ for the test-function in identities \eqref{eq:def-reg} and \eqref{eq:limit-eq}. Let us subtract these identities and rearrange the result in the following way:
\begin{equation}
\label{eq:rearrangement-1}
\begin{split}
\mathcal{G}_{\epsilon} & (\nabla u_\epsilon,\nabla u) \equiv \int_{Q_T}\left(\mathcal{F}^{(0,0)}_{\epsilon}(z,\nabla u_\epsilon)\nabla u_\epsilon-\mathcal{F}^{(0,0)}_{\epsilon}(z,\nabla u)\nabla u\right)\cdot \nabla (u_\epsilon-u)\,dz
\\
& = -\int_{Q_T}(u_{\epsilon}-u)_t(u_{\epsilon}-u)\,dz
\\
&
\qquad - \int_{Q_T}\left(\mathcal{F}^{(0,0)}_{\epsilon}(z,\nabla u)\nabla u- \mathcal{F}^{(0,0)}_{0}(z,\nabla u)\nabla u \right)\cdot \nabla (u_\epsilon-u)\,dz
\\
& \equiv  J_1(\epsilon)+J_2(\epsilon).
\end{split}
\end{equation}
By the choice of $\{u_\epsilon\}$
\[
J_1(\epsilon)=-\int_{Q_T}(u_{\epsilon t}-u_t)(u_\epsilon-u)\,dz\to 0 \quad \text{when $\epsilon\to 0$}
\]
as the product of weakly and strongly convergent sequences. By the generalized H\"older inequality \eqref{eq:Holder}

\[
\begin{split}
\vert J_2(\epsilon)\vert  & \leq 2 a^+\|  \nabla (u_\epsilon-u)\|  _{p(\cdot),Q_T} \|  (\epsilon^2+\vert \nabla u\vert ^2)^{\frac{p(\cdot)-2}{2}}\nabla u-\vert \nabla u\vert ^{p(\cdot)-2}\nabla u\|  _{p'(\cdot),Q_T}
\\
&
\qquad + 2 b^+\|  \nabla (u_\epsilon-u)\|  _{q(\cdot),Q_T} \|  (\epsilon^2+\vert \nabla u\vert ^2)^{\frac{q(\cdot)-2}{2}}\nabla u-\vert \nabla u\vert ^{q(\cdot)-2}\nabla u\|  _{q'(\cdot),Q_T}.
 \end{split}
\]
By \eqref{eq:conv-degenerate} the first factors in both terms on the right-hand side are bounded by an independent of $\epsilon$ constant $C$. To show that the second factors  tend to zero as $\epsilon\to 0$, it is sufficient to check that this is true for the modulars

\[
\mathcal{M}_{1, \epsilon}=\int_{Q_T} \sigma_{1,\epsilon}^{p'(z)}(z)\,dz, \qquad \sigma_{1, \epsilon}(z)\equiv \vert  (\epsilon^2+\vert \nabla u\vert ^2)^{\frac{p(z)-2}{2}}\nabla u- \vert \nabla u\vert ^{p(z)-2}\nabla u\vert ,
\]
\[
\mathcal{M}_{2, \epsilon}=\int_{Q_T} \sigma_{2,\epsilon}^{q'(z)}(z)\,dz, \qquad \sigma_{2,\epsilon}(z)\equiv \vert (\epsilon^2+\vert \nabla u\vert ^2)^{\frac{q(z)-2}{2}}\nabla u- \vert \nabla u\vert ^{q(z)-2}\nabla u\vert .
\]
We will consider in detail the integral $\mathcal{M}_{1,\epsilon}$. It is asserted that $\sigma_{1, \epsilon}(z) \to 0$ as $\epsilon\to 0$ for a.e. $z\in Q_T$. Indeed: for every $z\in Q_T$ either $\vert \nabla u(z)\vert =0$, whence  $\sigma_{1,\epsilon}(z)=0$ for all $\epsilon$, or $\vert \nabla u(z)\vert =\delta>0$ and

\[
\sigma_{1,\epsilon}(z)\leq \left\vert (\epsilon^2 +\delta^2)^{\frac{p(z)-2}{2}}- \delta^{p(z)-2}\right\vert \delta=\left\vert (\epsilon^2 +\delta^2)^{\frac{p(z)-1}{2}}- \delta^{p(z)-1}\right\vert \to 0\quad \text{as $\epsilon\to 0$}.
\]
The functions $\sigma_{1,\epsilon}^{p'(z)}(z)$ have the independent of $\epsilon$ integrable majorant:

\[
\begin{split}
\sigma^{p'(z)}_{1,\epsilon}(z) & \leq \left(\vert \nabla u\vert ^{p(z)-1}+\vert \nabla u\vert (\epsilon^2+\vert \nabla u\vert ^2)^{\frac{p(z)-2}{2}}\right)^{\frac{p(z)}{p(z)-1}}
\\
& \leq \left(\vert \nabla u\vert ^{p(z)-1}+(\epsilon^2+\vert \nabla u\vert ^2)^{\frac{p(z)-1}{2}}\right)^{\frac{p(z)}{p(z)-1}}
 \leq C\left(1+\vert \nabla u\vert ^{p(z)}\right)
\end{split}
\]
with a constant $C=C(p^\pm)$. By the dominated convergence theorem $\mathcal{M}_{1,\epsilon}\to 0$ as $\epsilon\to 0$. The same arguments show that $\mathcal{M}_{2,\epsilon}\to 0$ as $\epsilon\to 0$. Returning to \eqref{eq:rearrangement-1} we conclude that $\mathcal{G}_\epsilon(\nabla u_\epsilon,\nabla u)\to 0$ as $\epsilon\to 0$. By Lemma \ref{le:cont-convergence} $\vert \nabla (u_\epsilon-u)\vert \to 0$ in $L^{\underline{s}(\cdot)}(Q_T)$, whence the pointwise convergence $\nabla u_\epsilon\to \nabla u$ a.e. in $Q_T$.
\end{proof}

\subsection{Second-order regularity. Proof of Theorem \ref{th:global-reg}}
Fix $i,j\in \{1,2,\ldots,N\}$. By Theorem \ref{th:flux-a.e.} and Lemma \ref{le:a-e-conv}, there exists $\eta_{ij}\in L^{2}(Q_T)$ such that
\[
\begin{split}
& \text{$D_{j}\left((\mathcal{F}_{\epsilon_k}^{(0,0)})^{\frac{1}{2}}(z,\nabla u_{\epsilon_k})D_i u_{\epsilon_k}\right)\rightharpoonup \eta_{ij}$ in $L^{2}(Q_T)$},
\\
& (\mathcal{F}^{(0,0)}_{\epsilon_k})^{\frac{1}{2}}(z,\nabla u_{\epsilon_k})\nabla u_{\epsilon_k}\to (\mathcal{F}^{(0,0)}_{0})^{\frac{1}{2}}(z,\nabla u)\nabla u
\\
&\qquad \qquad \qquad
\equiv \left( a(z) \vert \nabla u\vert ^{p(z)-2} + b(z) \vert \nabla u\vert ^{q(z)-2} \right)^{\frac{1}{2}}\nabla u\quad \text{a.e. in $Q_T$}.
\end{split}
\]
By virtue of  \eqref{eq:grad-high-eps} $\|  (\mathcal{F}_{\epsilon_k}^{(0,0)})^{\frac{1}{2}}(z,\nabla u_{\epsilon_k})\nabla u_{\epsilon_k}\|  _{2+\delta,Q_T}$ are uniformly bounded for some $\delta>0$ whence by the Vitali convergence theorem
$
(\mathcal{F}_{\epsilon_k}^{(0,0)})^{\frac{1}{2}}(z,\nabla u_{\epsilon_k})\nabla u_{\epsilon_k}\to (\mathcal{F}_{0}^{(0,0)})^{\frac{1}{2}}(z,\nabla u)\nabla u$ in $L^{2}(Q_T)$.
It follows that $\eta_{ij}=D_{j}\left((\mathcal{F}_{0}^{(0,0)})^{\frac{1}{2}}(z,\nabla u)D_i u\right)$: for every $\phi \in C_0^{\infty}(\overline{Q}_T)$

\[
\begin{split}
-(\eta_{ij},\phi)_{2,Q_T} & =-\lim_{k\to \infty}\left(D_{j}\left((\mathcal{F}_{\epsilon_k}^{(0,0)})^{\frac{1}{2}}(z,\nabla u_{\epsilon_k})D_i u_{\epsilon_k}\right),\phi\right)_{2,Q_T}
\\
& = \lim_{k\to \infty}\left((\mathcal{F}_{\epsilon_k}^{(0,0)})^{\frac{1}{2}}(z,\nabla u_{\epsilon_k})D_i u_{\epsilon_k},D_j\phi\right)_{2,Q_T}
\\
&
= \left((\mathcal{F}_{0}^{(0,0)})^{\frac{1}{2}}(z,\nabla u)D_i u,D_j\phi\right)_{2,Q_T}.
\end{split}
\]

Let $\inf_{Q_T}\overline{s}(z)< 2$ and, thus, $\mathcal{D}_{p,q}=Q_T\cap \{z:\,\overline{s}(z)< 2\}\not=\emptyset$. Arguing as in the proof of Theorem \ref{th:flux-a.e.} we find that for every $\mathcal{B} \Subset \mathcal{D}_{p,q}$

\[
\begin{split}
\int_{\mathcal{B}}\vert D^2_{ij}u_{\epsilon}\vert ^{\underline{s}(z)}\,dz & \leq C
\end{split}
\]
with a constant $C$ independent of $\epsilon$ and $\mathcal{B}$. It follows that $D^2_{ij}u_{\epsilon_k}\rightharpoonup \zeta_{ij}\in L^{\underline{s}(\cdot)}(\mathcal{B})$ (up to a subsequence). Because of the weak convergence $\nabla u_{\epsilon_k}\rightharpoonup \nabla u$ in $L^{r(\cdot)}(Q_T)$ with $r(z)= \max\{2, \overline{s}(z)\}$, it is necessary that $\zeta_{ij}=D^{2}_{ij}u$. The
estimate $\|  D^{2}_{ij}u\|  _{\underline{s}(\cdot),\mathcal{B}}\leq C$ follows from the uniform estimate on $D^2_{ij}u_\epsilon$ as in the proof of Theorem \ref{th:flux-a.e.}.

\begin{remark}[Global boundedness of strong solutions]
Let the conditions of Theorem \ref{th:main-result-1} be fulfilled and, in addition, $f \in L^1(0,T; L^\infty(\Omega))$ and $u_0 \in L^\infty(\Omega)$. Then the strong solution of problem \eqref{eq:main} is bounded and satisfies the estimate
\begin{equation}\label{est:bdd}
\|  u(\cdot,t)\|  _{\infty,\Omega} \leq \|  u_0\|  _{\infty, \Omega} + \int_0^t  \|  f(\cdot, \tau)\|  _{\infty, \Omega}~d\tau
\end{equation}
\end{remark}
The proof is an imitation of the proof of \cite[Th.4.3]{ant-shm-book-2015}. Testing \eqref{eq:def} with the function $u_k^{2r-1}$, $u_k:= \min\{\vert u\vert , k\} \sign(u)$, $r\in \mathbb{N}$ and $k\geq K=1+\|  u_0\|  _{\infty,\Omega}$, we arrive at an integral inequality for the function $y_r(t)\equiv \|  u_k(t)\|  _{2r,\Omega}$. The conclusion follows because the solutions of this inequality are bounded uniformly with respect to $r$.

\bibliographystyle{siam}
\bibliography{d-ph} 

\end{document}